\documentclass{amsart}

\usepackage{amssymb}
\usepackage{mathrsfs}
\usepackage{multirow}
\usepackage{hhline}
\usepackage{bbm}
\usepackage{tikz}
\usetikzlibrary{decorations.pathmorphing,cd}
\usepackage[colorlinks=true]{hyperref}
\usepackage[capitalise]{cleveref}

\newtheorem{theorem}{Theorem}[section]
\newtheorem*{theorem*}{Theorem}
\newtheorem{proposition}[theorem]{Proposition}
\newtheorem{lemma}[theorem]{Lemma}
\newtheorem{corollary}[theorem]{Corollary}

\theoremstyle{definition}
\newtheorem{definition}[theorem]{Definition}

\newtheorem{claim}{Claim}
\crefname{claim}{Claim}{Claims}
\theoremstyle{remark}
\newtheorem*{remark}{Remark}
\newtheorem*{convention}{Convention}
\newtheorem*{abuse}{(Abuse of) notation}

\crefname{step}{Step}{Steps}
\mathchardef\mhyphen="2D

\newcommand{\C}{\mathscr{C}}
\newcommand{\D}{\mathscr{D}}
\newcommand{\V}{\mathscr{V}}

\newcommand{\E}{\mathcal{E}}
\renewcommand{\H}{\mathcal{H}}
\newcommand{\I}{\mathcal{I}}
\newcommand{\J}{\mathcal{J}}

\newcommand{\NN}{{\mathbb{N}}}

\newcommand{\shuffle}{\mathbf{Shfl}}
\newcommand{\celll}{\operatorname{cell}}

\newcommand{\id}{\mathrm{id}}
\newcommand{\op}{^{\mathrm{op}}}
\newcommand{\co}{^{\mathrm{co}}}

\newcommand{\colim}{\operatorname{colim}}
\newcommand{\im}{\operatorname{im}}

\newcommand{\two}{\mathbbm{2}}
\newcommand{\cell}{\Theta_2}
\newcommand{\spine}{\Xi}
\newcommand{\horn}{\Lambda}

\newcommand{\hattheta}{\widehat{\Theta_2}}
\newcommand{\hatdelta}{\widehat{\Delta}}
\newcommand{\wreath}{\hatdelta\wr\hatdelta}

\newcommand{\twoCat}{2\mhyphen\underline{\mathrm{Cat}}}
\newcommand{\Cat}{\underline{\mathrm{Cat}}}
\newcommand{\CAT}{\underline{\mathrm{CAT}}}
\newcommand{\Set}{\underline{\mathrm{Set}}}

\newcommand{\nq}{{[n;\mathbf{q}]}}
\newcommand{\np}{{[n;\mathbf{p}]}}

\newcommand{\nqd}{{[n-1;\mathbf{q'}]}}
\newcommand{\nqdd}{{[n-1;\mathbf{q''}]}}

\renewcommand{\mp}{{[m;\mathbf{p}]}}

\newcommand{\pp}{\mathbf{p}}
\newcommand{\qq}{\mathbf{q}}

\newcommand{\aalpha}{{\boldsymbol{\alpha}}}
\newcommand{\bbeta}{{\boldsymbol{\beta}}}
\newcommand{\iid}{\mathbf{id}}
\newcommand{\zzero}{{\boldsymbol{0}}}

\newcommand{\incl}{\hookrightarrow}
\newcommand{\defeq}{\overset{\text{def}}{=}}
\newcommand{\glue}{\arrow [dr, phantom, "\ulcorner" very near end, "\lrcorner" very near start, start anchor = south, end anchor = north]}
\newcommand{\pullback}{\arrow [dr, phantom, "\lrcorner" very near start, start anchor = south, end anchor = north]}

\newcommand{\tensor}{\otimes}
\newcommand{\ttensor}{\boxtimes}
\newcommand{\A}{\mathscr{A}}
\newcommand{\B}{\mathscr{B}}

\title{Inner horns for 2-quasi-categories}
\author{Yuki Maehara}
\address{Centre of Australian Category Theory, Macquarie University, NSW 2109, Australia}
\email{yuki.maehara@mq.edu.au}


\begin{document}

\begin{abstract}
	Dimitri Ara's \emph{2-quasi-categories}, which are certain presheaves over Andr\'{e} Joyal's \emph{2-cell category} $\Theta_2$, are an example of a concrete model that realises the abstract notion of $(\infty,2)$-category.
	In this paper, we prove that the 2-quasi-categories and the fibrations into them can be characterised using the \emph{inner horn inclusions} and the \emph{equivalence extensions} introduced by David Oury.
	These maps are more tractable than the maps that Ara originally used and therefore our result can serve as a combinatorial foundation for the study of 2-quasi-categories.
\end{abstract}

\maketitle

\section{Introduction}
There are several different models for $(\infty,1)$-categories, \emph{e.g.}~quasi-categories, complete Segal spaces, simplicial categories, etc.
Amongst these the most prominent is the presentation of $(\infty,1)$-categories as quasi-categories.
In addition to their being the most economical model among the geometric ones, many authors, most notably Andr\'{e} Joyal \cite{Joyal:Kan,Joyal:applications} and Jacob Lurie \cite{HTT,HA}, have shown that one can ``do category theory'' in quasi-categories.
In a similar vein, our ultimate goal is to ``do 2-category theory'' in 2-quasi-categories.

As their name suggests, 2-quasi-categories are an $(\infty,2)$-analogue to the $(\infty,1)$-notion of quasi-categories.
In \cite{Ara:nqcat}, Dimitri Ara constructed for each $n \ge 1$ a model structure on $\widehat{\Theta_n}$ which presents $(\infty,n)$-categories.
The \emph{$n$-quasi-categories} are the fibrant objects in $\widehat{\Theta_n}$ with respect to this structure.
In the case $n=1$ Ara's model structure coincides with Joyal's.
The fibrant objects of the case $n=2$, the 2-quasi-categories, are the subject of this work.

We originally wanted to understand the $\cell$-version of the (\emph{lax}) \emph{Gray tensor product}, and the main result of this paper was developed as a combinatorial tool for proving that tensor product to be left Quillen (which will be done in a future paper).
In \cite{Ara:nqcat}, Ara characterised not only the 2-quasi-categories, but also the fibrations into them.
More precisely, he proved them to be exactly those maps with the right lifting property with respect to a set $\J_A$ of monomorphisms.
Thus to prove the tensor product is left Quillen, it would suffice to check that it interacts nicely with the maps in $\J_A$.
However, the definition of $\J_A$ is complicated and not very easy to deal with.
The purpose of this paper is to provide an alternative set which is combinatorially more tractable.

More specifically, we show the set $\J_O$ of \emph{inner horn inclusions} and \emph{equivalence extensions}, introduced by David Oury in his PhD thesis \cite{Oury}, can be used in place of $\J_A$.
These maps are constructed from their simplicial counterparts using the \emph{box product} $\square : \wreath \to \hattheta$, analogously to how the bisimplicial horns may be constructed from the simplicial ones using the functor $\square : \hatdelta \times \hatdelta \to \widehat{\Delta \times \Delta}$.
The precise construction and other background material will be reviewed in \cref{background}.

The most technical (and also the longest) section of this paper is \cref{O-anodyne extensions suffice} where we compare the sets $\J_A$ and $\J_O$ and the class of trivial cofibrations.
In \cref{alternative horns} we consider a different notion of inner horn, namely the sub-$\cell$-sets of the representables generated by all but one codimension-one faces.
\cref{horizontal equivalence section} is very short and devoted to proving that the infinite family of horizontal equivalences (contained in both $\J_A$ and $\J_O$) can in fact be replaced by a single map as long as we keep the inner horn inclusions in the defining set of monomorphisms.
In \cref{last section} we prove the main theorem of the work (\cref{main theorem}).
\cref{teaser} illustrates how this theorem will be used in our future paper to prove that the Gray tensor product is left Quillen.

\section{Background}\label{background}

\subsection{Simplicial sets and shuffles}\label{simplicial sets}
As usual, we denote by $\Delta$ the category of non-empty finite ordinals $[n] \defeq \{0, \dots, n\}$ and order-preserving maps.
The morphisms in $\Delta$ will be called \emph{simplicial operators}.
We often denote a simplicial operator $\alpha : [m] \to [n]$ by its ``image'' $\{\alpha(0),\dots,\alpha(m)\}$; \emph{e.g.}~$\{0,2\} = \delta^1 : [1] \to [2]$ is the 1st elementary face operator.

We will write $\hatdelta$ for the category $[\Delta\op,\Set]$ of \emph{simplicial sets}, and write $\Delta[n]$ for the presheaf represented by $[n] \in \Delta$.
If $X \in \hatdelta$ is a simplicial set, $x \in X_n$ and $\alpha : [m] \to [n]$ is a simplicial operator, then we will write $x \cdot \alpha$ for the image of $x$ under $X(\alpha)$.

\begin{definition}
	An \emph{$(m,n)$-shuffle} is a non-degenerate $(m+n)$-simplex in the product $\Delta[m] \times \Delta[n]$.
\end{definition}
Equivalently, an $(m,n)$-shuffle $\langle\alpha,\alpha'\rangle$ consists of two surjections
\[
\begin{gathered}
	\alpha : [m+n] \to [m],\\
	\alpha': [m+n] \to [n]
\end{gathered}
\]
in $\Delta$ such that $\alpha(i) + \alpha'(i) = i$ for all $i \in [m+n]$.
We write $\shuffle(m,n)$ for the set of $(m,n)$-shuffles.
Note that an $(m,n)$-shuffle $\langle \alpha, \alpha' \rangle$ is uniquely determined by the surjection $\alpha : [m+n] \to [m]$ since $\alpha'$ can be recovered as $\alpha'(i) = i-\alpha(i)$.
Thus the pointwise order on $\Delta([m+n],[m])$ induces a partial order $\le$ on $\shuffle(m,n)$.
We have drawn in \cref{Hasse} two copies of $\shuffle(2,2)$, where each vertex $\langle \alpha, \alpha' \rangle$ is labelled with $\alpha$ (left) or the corresponding grid-path (right) which we describe now.
\begin{figure}
\[
\begin{tikzpicture}
\draw[gray, thin] (0,0)--(0,1)--(-1,2)--(0,3)--(1,2)--(0,1)(0,3)--(0,4);
\node[fill = white, scale = 0.7] at (0,0) {$\{0,0,0,1,2\}$};
\node[fill = white, scale = 0.7] at (0,1) {$\{0,0,1,1,2\}$};
\node[fill = white, scale = 0.7] at (-1,2) {$\{0,0,1,2,2\}$};
\node[fill = white, scale = 0.7] at (1,2) {$\{0,1,1,1,2\}$};
\node[fill = white, scale = 0.7] at (0,3) {$\{0,1,1,2,2\}$};
\node[fill = white, scale = 0.7] at (0,4) {$\{0,1,2,2,2\}$};
\end{tikzpicture}\hspace{30pt}
\begin{tikzpicture}
\draw[thick]
(-0.2,-0.2)--(-0.2,0.2)--(0.2,0.2)
(-0.2,0.8)--(-0.2,1)--(0,1)--(0,1.2)--(0.2,1.2)
(-1.2,1.8)--(-1.2,2)--(-0.8,2)--(-0.8,2.2)
(0.8,1.8)--(1,1.8)--(1,2.2)--(1.2,2.2)
(-0.2,2.8)--(0,2.8)--(0,3)--(0.2,3)--(0.2,3.2)
(-0.2,3.8)--(0.2,3.8)--(0.2,4.2);
\draw[gray, thin]
(0,0.35)--(0,0.7)
(-0.3,1.3)--(-0.7,1.7)
(0.3,1.3)--(0.7,1.7)
(-0.7,2.3)--(-0.3,2.7)
(0.7,2.3)--(0.3,2.7)
(0,3.3)--(0,3.65);
\end{tikzpicture}
\]
\caption{$\shuffle(2,2)$}\label{Hasse}
\end{figure}

We can visualise $(m,n)$-shuffles as paths on the $m \times n$ grid from the lower-left corner to the upper-right corner.
For example, the path in \cref{example shuffle} corresponds to the $(3,2)$-shuffle $\langle \{0,0,1,2,2,3\}, \{0,1,1,1,2,2\} \rangle$.
(If either $m=0$ or $n=0$ then the ``grid'' becomes a line segment.
In this case we have a unique path connecting the two endpoints, which corresponds to having a unique $(m,n)$-shuffle.)
This motivates the following notation.

\begin{figure}
	\[
	\begin{tikzpicture}
	\draw[gray!50!white, thin] (0,0) -- (3,0) (0,1) -- (3,1) (0,2) -- (3,2) (0,0) -- (0,2) (1,0) -- (1,2) (2,0) -- (2,2) (3,0) -- (3,2);
	\draw[thick] (0,0) -- (0,1) -- (2,1) -- (2,2) -- (3,2);
	\node[fill = white, scale = 0.7] at (0,0) {0};
	\node[fill = white, scale = 0.7] at (0,1) {1};
	\node[fill = white, scale = 0.7] at (1,1) {2};
	\node[fill = white, scale = 0.7] at (2,1) {3};
	\node[fill = white, scale = 0.7] at (2,2) {4};
	\node[fill = white, scale = 0.7] at (3,2) {5};
	\end{tikzpicture}
	\]
	\caption{$\langle \{0,0,1,2,2,3\}, \{0,1,1,1,2,2\} \rangle$}\label{example shuffle}
\end{figure}
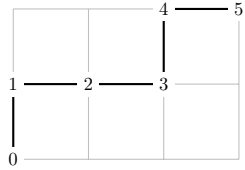

\begin{definition}\label{corners}
Given an $(m,n)$-shuffle $\langle \alpha, \alpha' \rangle$, we will write:
\begin{itemize}
	\item $\lrcorner\langle\alpha,\alpha'\rangle$ for the set of all $1 \le i \le m+n-1$ such that
	\[
	\alpha(i+1) = \alpha(i) = \alpha(i-1)+1
	\]
	(or equivalently $\alpha'(i+1) = \alpha'(i)+1 = \alpha'(i-1)+1$) holds; and
	\item $\ulcorner\langle\alpha,\alpha'\rangle$ for the set of all $1 \le i \le m+n-1$ such that 
	\[
	\alpha(i+1)=\alpha(i)+1=\alpha(i-1)+1
	\]
	(or equivalently $\alpha'(i+1) = \alpha'(i) = \alpha'(i-1)+1$) holds.
\end{itemize}
\end{definition}
For example, if $\langle\alpha,\alpha'\rangle$ is the $(3,2)$-shuffle depicted in \cref{example shuffle}, then $\lrcorner\langle\alpha,\alpha'\rangle = \{3\}$ and $\ulcorner\langle\alpha,\alpha'\rangle = \{1,4\}$.
The following propositions are straightforward to prove.
\begin{proposition}\label{dominating shuffle}
	Let $\langle\alpha,\alpha'\rangle, \langle\beta,\beta'\rangle$ be $(m,n)$-shuffles.
	Suppose $\alpha(i) = \beta(i)$ (and so $\alpha'(i) = \beta'(i)$) for each $i \in \lrcorner\langle\alpha,\alpha'\rangle$.
	Then $\langle\alpha,\alpha'\rangle \le \langle\beta,\beta'\rangle$.
\end{proposition}

\begin{proposition}\label{immediate predecessor of shuffle}
	Let $\langle\alpha,\alpha'\rangle$ be an $(m,n)$-shuffle and suppose $i \in \lrcorner\langle\alpha,\alpha'\rangle$.
	Then $\langle\alpha,\alpha'\rangle$ has an immediate predecessor $\langle\beta,\beta'\rangle$ such that $\langle\alpha,\alpha'\rangle \circ \delta^i = \langle \beta,\beta'\rangle \circ \delta^i$.
	Moreover, this condition determines $\langle\beta,\beta'\rangle$ uniquely and induces a bijection between $\lrcorner\langle\alpha,\alpha'\rangle$ and the set of immediate predecessors of $\langle\alpha,\alpha'\rangle$.
	Similarly, there is a bijection between $\ulcorner\langle\alpha,\alpha'\rangle$ and the set of immediate successors of $\langle\alpha,\alpha'\rangle$.
\end{proposition}

For $1 \le i \le m+n-1$, the grid-path corresponding to $\langle\alpha,\alpha'\rangle \in \shuffle(m,n)$ locally looks like:
\[
\begin{tikzpicture}[baseline = 8, scale = 0.7]
\draw (0,0) -- (1,0) -- (1,1);
\node[scale = 0.7, fill = white] at (1,0) {$i$};
\end{tikzpicture}\hspace{5pt}, \hspace{10pt}
\begin{tikzpicture}[baseline = 8, scale = 0.7]
\draw (0,0) -- (0,1) -- (1,1);
\node[scale = 0.7, fill = white] at (0,1) {$i$};
\end{tikzpicture}\hspace{5pt}, \hspace{10pt}
\begin{tikzpicture}[baseline = -3, scale = 0.7]
\draw (0,0) -- (1.4,0);
\node[scale = 0.7, fill = white] at (0.7,0) {$i$};
\end{tikzpicture} \hspace{10pt} \text {or} \hspace{10pt}
\begin{tikzpicture}[baseline = 12, scale = 0.7]
\draw (0,0) -- (0,1.4);
\node[scale = 0.7, fill = white] at (0,0.7) {$i$};
\end{tikzpicture}
\]
This observation can be formalised as follows.

\begin{proposition}\label{points}
	Let $\langle\alpha,\alpha'\rangle$ be an $(m,n)$-shuffle.
	Then for any $1 \le i \le m+n-1$, precisely one of the following holds:
	\begin{itemize}
		\item $i \in \lrcorner\langle\alpha,\alpha'\rangle$;
		\item $i \in \ulcorner\langle\alpha,\alpha'\rangle$;
		\item $\alpha^{-1}(\alpha(i)) = \{i\}$; or
		\item $(\alpha')^{-1}(\alpha'(i)) = \{i\}$.
	\end{itemize}
\end{proposition}

\subsection{The category $\cell$}\label{Theta_2}
The category $\Delta$ can be seen as the full subcategory of $\Cat$ spanned by the free categories $[n]$ generated by linear graphs:
\[
\begin{tikzcd}
0 \arrow [r] & 1 \arrow [r] & \dots \arrow [r] & n
\end{tikzcd}
\]
Similarly, Joyal's \emph{2-cell category} $\cell$ is the full subcategory of $\twoCat$ spanned by the free 2-categories $[n;q_1,\dots,q_n]$ generated by ``linear-graph-enriched linear graphs'':
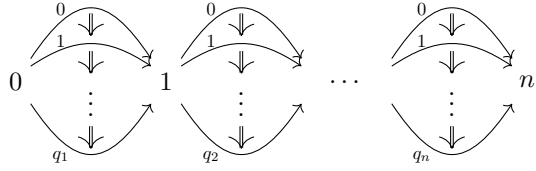
\[
\begin{tikzpicture}[scale = 2]
\node at (0,0) {0};
\node at (1,0) {1};
\node at (2.2,0) {$\dots$};
\node at (3.4,0) {$n$};
\draw[->] (0.1,0.15) .. controls (0.4,0.6) and (0.6,0.6) .. (0.9,0.15);
\draw[->] (0.1,0.1) .. controls (0.4,0.3) and (0.6,0.3) .. (0.9,0.1);
\draw[->] (0.1,-0.15) .. controls (0.4,-0.6) and (0.6,-0.6) .. (0.9,-0.15);
\draw[->, double] (0.5,0.45) -- (0.5,0.3);
\draw[->, double] (0.5,-0.3) -- (0.5,-0.45);
\draw[->, double] (0.5,0.2) -- (0.5,0.05);
\node at (0.5,-0.1) {$\vdots$};
\node[scale = 0.7] at (0.3,0.48) {0};
\node[scale = 0.7] at (0.3,0.27) {1};
\node[scale = 0.7] at (0.3,-0.5) {$q_1$};
\draw[->] (1.1,0.15) .. controls (1.4,0.6) and (1.6,0.6) .. (1.9,0.15);
\draw[->] (1.1,0.1) .. controls (1.4,0.3) and (1.6,0.3) .. (1.9,0.1);
\draw[->] (1.1,-0.15) .. controls (1.4,-0.6) and (1.6,-0.6) .. (1.9,-0.15);
\draw[->, double] (1.5,0.45) -- (1.5,0.3);
\draw[->, double] (1.5,-0.3) -- (1.5,-0.45);
\draw[->, double] (1.5,0.2) -- (1.5,0.05);
\node at (1.5,-0.1) {$\vdots$};
\node[scale = 0.7] at (1.3,0.48) {0};
\node[scale = 0.7] at (1.3,0.27) {1};
\node[scale = 0.7] at (1.3,-0.5) {$q_2$};
\draw[->] (2.5,0.15) .. controls (2.8,0.6) and (3,0.6) .. (3.3,0.15);
\draw[->] (2.5,0.1) .. controls (2.8,0.3) and (3,0.3) .. (3.3,0.1);
\draw[->] (2.5,-0.15) .. controls (2.8,-0.6) and (3,-0.6) .. (3.3,-0.15);
\draw[->, double] (2.9,0.45) -- (2.9,0.3);
\draw[->, double] (2.9,-0.3) -- (2.9,-0.45);
\draw[->, double] (2.9,0.2) -- (2.9,0.05);
\node at (2.9,-0.1) {$\vdots$};
\node[scale = 0.7] at (2.7,0.48) {0};
\node[scale = 0.7] at (2.7,0.27) {1};
\node[scale = 0.7] at (2.7,-0.5) {$q_n$};
\end{tikzpicture}
\]
whose hom-categories are given by
\[
\hom(k,\ell) = \left\{\begin{array}{cl}
[q_{k+1}] \times \cdots \times [q_\ell] & \text {if $k \le \ell$,}\\
\varnothing & \text {if $k > \ell$}.
\end{array}\right.
\]

More precisely, $\cell$ has objects $\nq = [n; q_1, \dots, q_n]$ where $n, q_k \in \NN$ for each $k$.
A morphism $[\alpha; \aalpha] = [\alpha; \alpha_{\alpha(0)+1}, \dots, \alpha_{\alpha(m)}]: \mp \to \nq$ consists of simplicial operators $\alpha : [m] \to [n]$ and $\alpha_k : [p_\ell] \to [q_k]$ for each $k \in [n]$ such that there exists (necessarily unique) $\ell \in [m]$ with $\alpha(\ell-1)<k\le\alpha(\ell)$.
By a \emph{cellular operator} we mean a morphism in $\cell$.
Clearly $[0]$ is a terminal object in $\cell$, and we will write $! : \nq \to [0]$ for any cellular operator into $[0]$.

\begin{remark}
	Here we are describing $\cell = \Delta \wr \Delta$ as an instance of Berger's \emph{wreath product} construction.
	For any given category $\C$, the wreath product $\Delta \wr \C$ may be thought of as the category of free $\C$-enriched (or more accurately $\widehat{\C}$-enriched) categories generated by linear $\C$-enriched graphs.
	The precise definition can be found in \cite[Definition 3.1]{Berger:wreath}.
\end{remark}

\begin{remark}
	The notation for objects (and maps) in $\cell$ varies from author to author.
	(This is partly because some authors introduce a notation for objects in a general wreath product category which can be specialised to $\cell = \Delta\wr\Delta$ while others are interested in the particular category $\cell$ and hence able to adopt a more economical notation.)
	For example, the object we denote by $\nq = [n;q_1,\dots,q_n]$ would be denoted as:
	\begin{itemize}
		\item $([q_1],\dots,[q_n])$ in \cite{Berger:wreath};
		\item $\nq = ([n],[-1],[q_1],\dots,[q_n],[-1])$ in \cite{Oury}; 
		\item $([n],[q_1],\dots,[q_n])$ in \cite{Rezk:cartesian}; and
		\item $\langle q_1, \dots, q_n \rangle$ in \cite{Watson}.
	\end{itemize}
	In \cite{Ara:nqcat} an object in $\cell$ (or more generally in $\Theta_n$) is specified using the \emph{table of dimensions}; see \emph{loc.~cit.}~for details.
\end{remark}
The category $\Delta$ has an automorphism $(-)\op$ which is the identity on objects and sends $\alpha : [m] \to [n]$ to $\alpha\op : [m] \to [n]$ given by $\alpha\op(i) = n-\alpha(m-i)$.
This induces two automorphisms on $\cell$, namely:
\begin{itemize}
	\item $(-)\co : \cell \to \cell$, which sends $[\alpha;\aalpha] : \mp \to \nq$ to
	\[
	[\alpha;\alpha\op_{\alpha(0)+1},\dots,\alpha\op_{\alpha(m)}] : \mp \to \nq;
	\]
	and
	\item $(-)\op : \cell \to \cell$, which sends $[\alpha;\aalpha] : \mp \to \nq$ to
	\[
	[\alpha\op;\alpha_{\alpha(m)},\dots,\alpha_{\alpha(0)+1}] : [m;p_m,\dots,p_1] \to [n;q_n,\dots,q_1].
	\]
\end{itemize}

\subsection{Face maps in $\cell$}
There is a Reedy category structure on $\cell$ defined as follows; see \cite[Proposition 2.11]{Bergner;Rezk:Reedy} or \cite[Lemma 2.4]{Berger:nerve} for a proof.
\begin{definition}
	The \emph{dimension} of $\nq$ is $\dim\nq \defeq n + \sum_{k=1}^n q_k$.
	A cellular operator $[\alpha;\aalpha] : \mp \to \nq$ is a \emph{face operator} if $\alpha$ is monic and $\{\alpha_k:\alpha(\ell-1) < k \le \alpha(\ell)\}$ is jointly monic for each $1 \le \ell \le m$.
	It is a \emph{degeneracy operator} if $\alpha$ and all $\alpha_k$ are surjective.
\end{definition}
\begin{definition}
	A simplicial operator $\alpha : [m] \to [n]$ is \emph{inert} if it is a subinterval inclusion, that is, if $\alpha(i+1) = \alpha(i)+1$ for $0 \le i \le m-1$.
\end{definition}
\begin{definition}
We say a face map $[\alpha;\aalpha] : \mp \to \nq$ is:
\begin{itemize}
	\item \emph{inner} if $\alpha$ and all $\alpha_k$ preserve the top and bottom elements, and otherwise \emph{outer};
	\item \emph{horizontal} if each $\alpha_k$ is surjective;
	\item \emph{vertical} if $\alpha = \id$; and
	\item \emph{inert} if $\alpha$ and all $\alpha_k$ are inert.
\end{itemize}
(Examples of each kind can be found in \cref{faces}.)
A horizontal face map of the form $[\delta^k;\aalpha]$ will be called a \emph{$k$-th horizontal face}.
\end{definition}

By the \emph{codimension} of a face map $[\alpha;\aalpha] : \mp \to \nq$, we mean the difference $\dim\nq - \dim\mp$.
We will in particular be interested in the face maps of codimension 1, which we call \emph{hyperfaces}.
Such a map $[\alpha;\aalpha]$ has precisely one of the following forms:
\begin{itemize}
	\item for $n \ge 1$, $\nq$ always has a unique \emph{0-th horizontal face}
	\[
	\delta_h^0 \defeq [\delta^0;\iid] : [n-1;q_2,\dots,q_n] \to \nq
	\]
	which has codimension 1 if and only if $q_1 = 0$;
	\item similarly, if $q_n = 0$ then the unique \emph{$n$-th horizontal face}
	\[
	\delta_h^n \defeq [\delta^n;\iid] : [n-1;q_1,\dots,q_{n-1}] \to \nq
	\]
	has codimension 1;
	\item for each $1 \le k \le n-1$, there is a family of \emph{$k$-th horizontal hyperfaces}
	\[
	\delta_h^{k;\langle\beta,\beta'\rangle} \defeq [\delta^k;\aalpha] : [n-1;q_1,\dots,q_{k-1},q_k+q_{k+1},q_{k+2},\dots,q_n] \to \nq
	\]
	indexed by $\langle \beta,\beta' \rangle \in \shuffle(q_k,q_{k+1})$ where $\alpha_\ell = \id$ for $k \neq \ell \neq k+1$, $\alpha_k = \beta$ and $\alpha_{k+1} = \beta'$; and
	\item for each $1 \le k \le n$ satisfying $q_k \ge 1$ and for each $0 \le i \le q_k$, the \emph{$(k;i)$-th vertical hyperface}
	\[
	\delta_v^{k;i} \defeq [\id;\aalpha] : [n;q_1,\dots,q_{k-1},q_k-1,q_{k+1},\dots,q_n] \to \nq
	\]
	is given by $\alpha_k = \delta^i$ and $\alpha_\ell = \id$ for $\ell \neq k$.
\end{itemize}

\begin{convention}
	Strictly speaking, we are giving the same name to different cellular operators, and this can lead to confusion.
	So in the rest of this paper, we will assume the codomain of any cellular operator denoted by $\delta$ (with some decoration) is always \emph{whatever is called $\nq$ at that point} (or some cellular subset of $\cell\nq$ as described in \cref{cellular sets}).
	When this is not the case, we will indicate the codomain $\mp$ either by writing $\delta\mp$ instead of $\delta$, or by drawing $\delta$ as an arrow $\begin{tikzcd} {[m';\mathbf{p'}]} \arrow [r, "\delta"] & {\mp}.\end{tikzcd}$
\end{convention}

In \cref{faces}, we have listed various faces of $[2;0,2]$.
We will briefly describe how to read the pictures.
In the first row is the ``standard picture'' of $[2;0,2]$, in which we have nicely placed its objects ($\begin{tikzpicture}[baseline = -3]\filldraw (0,0) circle [radius = 1pt]; \end{tikzpicture}$), generating 1-cells ($\begin{tikzpicture}[baseline = -3]\draw[->] (0,0)--(0.5,0); \end{tikzpicture}$) and generating 2-cells ($\begin{tikzpicture}[baseline = -3]\draw[->, double] (0,0)--(0.5,0); \end{tikzpicture}$).
In the rest of the table, a face operator $[\alpha;\aalpha] : \mp \to [2;0,2]$ is illustrated as the standard picture of $\mp$ appropriately distorted so that the $\ell$-th object appears in the $\alpha(\ell)$-th position and each generating 1-cell lies roughly where the factors of its image used to.
In the third row (where $\alpha_1$ is not injective), we have left small gaps between the generating 1-cells so that they do not intersect with each other.
\begin{table}
\begin{tabular}{|c||c|c|c|c|c|c|}\hline
	& picture & domain & inner/outer & horizontal & vertical & inert \\\hhline{|=#=|=|=|=|=|=|}
	$\id$ &
	$\begin{tikzpicture}[baseline = -3]
	\filldraw
	(0,0) circle [radius = 1pt]
	(1,0) circle [radius = 1pt]
	(2,0) circle [radius = 1pt];
	\draw[->] (0.1,0)--(0.9,0);
	\draw[->] (1.1,0.15) .. controls (1.4,0.6) and (1.6,0.6) .. (1.9,0.15);
	\draw[->] (1.1,-0.15) .. controls (1.4,-0.6) and (1.6,-0.6) .. (1.9,-0.15);
	\draw[->] (1.1,0)--(1.9,0);
	\draw[->, double] (1.5,-0.05)--(1.5,-0.4);
	\draw[->, double] (1.5,0.4)--(1.5,0.05);
	\end{tikzpicture}$
	& $[2;0,2]$ & inner & \checkmark & \checkmark & \checkmark \\\hline
	
	$\delta_h^0 = [\delta^0;\id]$ &
	$\begin{tikzpicture}[baseline = -3]
	\draw[gray!50!white, fill = gray!50!white]
	(0,0) circle [radius = 1pt];
	\draw[->,gray!50!white] (0.1,0)--(0.9,0);
	\filldraw
	(1,0) circle [radius = 1pt]
	(2,0) circle [radius = 1pt];
	\draw[->] (1.1,0.15) .. controls (1.4,0.6) and (1.6,0.6) .. (1.9,0.15);
	\draw[->] (1.1,-0.15) .. controls (1.4,-0.6) and (1.6,-0.6) .. (1.9,-0.15);
	\draw[->] (1.1,0)--(1.9,0);
	\draw[->, double] (1.5,0.4)--(1.5,0.05);
	\draw[->, double] (1.5,-0.05)--(1.5,-0.4);
	\end{tikzpicture}$
	& $[1;2]$ & outer & \checkmark & $\times$ & \checkmark \\\hline
	
	$\delta_h^{1;\langle!,\id\rangle} = [\delta^1;!,\id]$ &
	$\begin{tikzpicture}[baseline = -3]
	\draw[white] (0,0) circle [radius = 1pt];
	\filldraw
	(0,0) circle [radius = 1pt]
	(2,0) circle [radius = 1pt];
	\draw[->, yshift = 2pt] (0.1,0) -- (1,0) .. controls (1.4,0.6) and (1.6,0.6) .. (1.9,0.15);
	\draw[->, yshift = -2pt] (0.1,0) -- (1,0) .. controls (1.4,-0.6) and (1.6,-0.6) .. (1.9,-0.15);
	\draw[->] (0.1,0)--(1.9,0);
	\draw[->, double] (1.5,0.45)--(1.5,0.05);
	\draw[->, double] (1.5,-0.05)--(1.5,-0.45);
	\end{tikzpicture}$
	& $[1;2]$ & inner & \checkmark & $\times$ & $\times$ \\\hline
	 
	$\delta_v^{2;0} = [\id;\id,\delta^0]$ &
	$\begin{tikzpicture}[baseline = -3]
	\filldraw
	(0,0) circle [radius = 1pt]
	(1,0) circle [radius = 1pt]
	(2,0) circle [radius = 1pt];
	\draw[->] (0.1,0)--(0.9,0);
	\draw[->, gray!50!white] (1.1,0.15) .. controls (1.4,0.6) and (1.6,0.6) .. (1.9,0.15);
	\draw[->] (1.1,-0.15) .. controls (1.4,-0.6) and (1.6,-0.6) .. (1.9,-0.15);
	\draw[->] (1.1,0)--(1.9,0);
	\draw[->, double] (1.5,-0.05)--(1.5,-0.4);
	\draw[->, double, gray!50!white] (1.5,0.4)--(1.5,0.05);
	\end{tikzpicture}$
	& \multirow{2}{*}[-12pt]{$[2;0,1]$} & \multirow{2}{*}[-12pt]{outer} & \multirow{2}{*}[-12pt]{$\times$} & \multirow{2}{*}[-12pt]{\checkmark} & \multirow{2}{*}[-12pt]{\checkmark} \\
	
	$\delta_v^{2;2} = [\id;\id,\delta^2]$ &
	$\begin{tikzpicture}[baseline = -3]
	\filldraw
	(0,0) circle [radius = 1pt]
	(1,0) circle [radius = 1pt]
	(2,0) circle [radius = 1pt];
	\draw[->] (0.1,0)--(0.9,0);
	\draw[->] (1.1,0.15) .. controls (1.4,0.6) and (1.6,0.6) .. (1.9,0.15);
	\draw[->, gray!50!white] (1.1,-0.15) .. controls (1.4,-0.6) and (1.6,-0.6) .. (1.9,-0.15);
	\draw[->] (1.1,0)--(1.9,0);
	\draw[->, double, gray!50!white] (1.5,-0.05)--(1.5,-0.4);
	\draw[->, double] (1.5,0.4)--(1.5,0.05);
	\end{tikzpicture}$
	& & & & & \\\hline
	
	$\delta_v^{2;1} = [\id;\id,\delta^1]$ &
	$\begin{tikzpicture}[baseline = -3]
	\filldraw
	(0,0) circle [radius = 1pt]
	(1,0) circle [radius = 1pt]
	(2,0) circle [radius = 1pt];
	\draw[->] (0.1,0)--(0.9,0);
	\draw[->] (1.1,0.15) .. controls (1.4,0.6) and (1.6,0.6) .. (1.9,0.15);
	\draw[->] (1.1,-0.15) .. controls (1.4,-0.6) and (1.6,-0.6) .. (1.9,-0.15);
	\draw[->, double] (1.5,0.4)--(1.5,-0.4);
	\end{tikzpicture}$
	& $[2;0,1]$ & inner & $\times$ & \checkmark & $\times$ \\\hline

	$\delta_h^2 = [\delta^2;\id]$ &
	$\begin{tikzpicture}[baseline = -3]
	\filldraw
	(0,0) circle [radius = 1pt]
	(1,0) circle [radius = 1pt];
	\filldraw[gray!50!white]
	(2,0) circle [radius = 1pt];
	\draw[->] (0.1,0)--(0.9,0);
	\draw[->, gray!50!white] (1.1,0.15) .. controls (1.4,0.6) and (1.6,0.6) .. (1.9,0.15);
	\draw[->, gray!50!white] (1.1,-0.15) .. controls (1.4,-0.6) and (1.6,-0.6) .. (1.9,-0.15);
	\draw[->, gray!50!white] (1.1,0)--(1.9,0);
	\draw[->, double, gray!50!white] (1.5,-0.05)--(1.5,-0.4);
	\draw[->, double, gray!50!white] (1.5,0.4)--(1.5,0.05);
	\end{tikzpicture}$
	& $[1;0]$ & outer & \checkmark & $\times$ & \checkmark \\\hline
	
	$[\{0\}]$ &
	$\begin{tikzpicture}[baseline = -3]
	\filldraw
	(0,0) circle [radius = 1pt];
	\filldraw[gray!50!white]
	(1,0) circle [radius = 1pt]
	(2,0) circle [radius = 1pt];
	\draw[->, gray!50!white] (0.1,0)--(0.9,0);
	\draw[->, gray!50!white] (1.1,0.15) .. controls (1.4,0.6) and (1.6,0.6) .. (1.9,0.15);
	\draw[->, gray!50!white] (1.1,-0.15) .. controls (1.4,-0.6) and (1.6,-0.6) .. (1.9,-0.15);
	\draw[->, gray!50!white] (1.1,0)--(1.9,0);
	\draw[->, double, gray!50!white] (1.5,-0.05)--(1.5,-0.4);
	\draw[->, double, gray!50!white] (1.5,0.4)--(1.5,0.05);
	\end{tikzpicture}$
	& $[0]$ & outer & \checkmark & $\times$ & \checkmark \\\hline
\end{tabular}
\caption{Some faces of $[2;0,2]$}\label{faces}
\end{table}

The hyperfaces of $\nq$ are precisely the maximal faces of $\nq$ in the following sense.
\begin{proposition}[{\cite[Proposition 6.2.4]{Watson}}]\label{Watson}
	Any face map $[\alpha;\aalpha] : \mp \to \nq$ of positive codimension factors through a hyperface of $\nq$.
\end{proposition}
We will also need the following outer version of this proposition.
\begin{proposition}\label{inert factors through outer}
	Any outer face map $[\alpha;\aalpha] : \mp \to \nq$ factors through an outer hyperface of $\nq$.
\end{proposition}
\begin{proof}
	Recall that $[\alpha;\aalpha]$ is inner (= non-outer) if and only if $\alpha$ and all $\alpha_k$ preserve the top and bottom elements.
	We will consider the cases where either $\alpha$ or some $\alpha_k$ does not preserve the top elements; the other cases can be treated dually.
	\begin{itemize}
		\item[(i)] If $\alpha(m) \neq n$ and $q_n = 0$ then we can factorise $[\alpha;\aalpha]$ as 
		\[
		\begin{tikzcd}
		{\mp}
		\arrow [r, "{[\beta;\aalpha]}"] &
		{[n-1;q_1,\dots,q_{n-1}]}
		\arrow [r, "\delta_h^n"] &
		{\nq}
		\end{tikzcd}
		\]
		where $\beta : [m] \to [n-1]$ is given by $\beta(i) = \alpha(i)$.
		\item[(ii)] If $\alpha(m) \neq n$ and $q_n \ge 1$ then we can factorise $[\alpha;\aalpha]$ as
		\[
		\begin{tikzcd}
		{\mp}
		\arrow [r, "{[\alpha;\aalpha]}"] &
		{[n;q_1,\dots,q_{n-1},q_n-1]}
		\arrow [r, "\delta_v^{n;0}"] &
		{\nq}.
		\end{tikzcd}
		\]
		\item[(iii)] If $\alpha_k(p_\ell) \neq q_k$ for some $\alpha(\ell-1) < k \le \alpha(\ell)$ then we can factorise $[\alpha;\aalpha]$ as
		\[
		\begin{tikzcd}
		{\mp}
		\arrow [r, "{[\alpha;\bbeta]}"] &
		{[n;q_1,\dots,q_k-1,\dots,q_n]}
		\arrow [r, "\delta_v^{k;q_k}"] &
		{\nq}
		\end{tikzcd}
		\]
		where $\beta_k: [p_\ell] \to [q_k-1]$ is given by $\beta_k(i) = \alpha_k(i)$ and $\beta_{k'} = \alpha_{k'}$ for ${k'} \neq k$.
	\end{itemize}
\end{proof}

\subsection{Cellular sets}\label{cellular sets}
We will write $\hattheta$ for the category $[\Theta_2\op,\Set]$ of \emph{cellular sets}.
If $X$ is a cellular set, $x \in X_{n;\qq} \defeq X(\nq)$ and $[\alpha;\aalpha] :\mp \to \nq$ is a cellular operator, then we will write $x \cdot [\alpha;\aalpha]$ for the image of $x$ under $X([\alpha;\aalpha])$.
The Reedy structure on $\cell$ is (\emph{EZ} and hence) \emph{elegant}, which means the following.
\begin{theorem}[{\cite[Corollary 4.5]{Bergner;Rezk:Reedy}}]\label{elegance}
	For any cellular set $X$ and for any $x \in X_{m;\pp}$, there is a unique way to express $x$ as $x = y \cdot [\alpha;\aalpha]$ where $[\alpha;\aalpha] : \mp \to \nq$ is a degeneracy operator and $y \in X_{n;\qq}$ is non-degenerate.
\end{theorem}

\begin{definition}
A \emph{cellular subset} of $X \in \hattheta$ is a subfunctor of $X$.
If $S$ is a set of cells in $X \in \hattheta$ (not necessarily closed under the action of cellular operators), the smallest cellular subset $\overline{S}$ of $X$ containing $S$ is given by
\[
\overline{S}_{m;\pp} = \{s \cdot [\alpha;\aalpha] : s \in S_{n;\qq}, \begin{tikzcd} {\mp} \arrow [r, "{[\alpha;\aalpha]}"] & \nq \end{tikzcd}\}.
\]
We call $\overline{S}$ the cellular subset of $X$ \emph{generated by $S$}.
\end{definition}

\begin{abuse}
We will write $\cell\nq$ for the presheaf represented by $\nq \in \cell$.
If $[\alpha;\aalpha] : \mp \to \nq$ is a cellular operator, then the corresponding map $\cell\mp \to \cell\nq$ will also be denoted by $[\alpha;\aalpha]$.
Moreover, if $X \subset \cell\nq$ is a cellular subset and there exists a (necessarily unique) factorisation
\[
\begin{tikzcd}[row sep = small, column sep = small]
& X
\arrow [dd, "\subset", hook] \\
{\cell\mp}
\arrow [ur, dashed]
\arrow [dr, swap, "{[\alpha;\aalpha]}"] & \\
& {\cell\nq}
\end{tikzcd}
\]
then we abuse the notation and write $[\alpha;\aalpha]$ for the dashed map too.
Note that the domain of $[\alpha;\aalpha]$ is still the representable one and so \emph{$[\alpha;\aalpha]$ always corresponds to a single cell in its codomain.}
The convention introduced in \cref{Theta_2} extends to this context in the sense that any map in $\hattheta$ denoted by $\delta$ (with some decoration) will always have as codomain \emph{some cellular subset of $\cell\nq$} unless indicated otherwise.
\end{abuse}

There is a functor $\cell \to \Delta$ given by sending $[\alpha;\aalpha] : \mp \to \nq$ to $\alpha : [m] \to [n]$.
We will regard $\hatdelta$ as a full subcategory of $\hattheta$ via the embedding $\hatdelta \to \hattheta$ induced by this functor.
Hence the square
\[
\begin{tikzcd}
\Cat
\arrow [r]
\arrow [d, "N", swap] &
\twoCat
\arrow [d, "N"] \\
\hatdelta
\arrow [r, hook, "\subset"] &
\hattheta
\end{tikzcd}
\]
commutes up to isomorphism, where the upper horizontal map sends each category to the obvious locally discrete 2-category, and the vertical maps are the nerve functors induced by the inclusions $\Delta \incl \Cat$ and $\cell \incl \twoCat$.

\subsection{The category $\wreath$}\label{generalised wreath}
Most content of \cref{generalised wreath,box product,O-anodyne} is taken from David Oury's PhD thesis \cite{Oury}.

In this subsection, we will describe Oury's \emph{generalised wreath product} $\wreath$ which should be thought of as a category of presentations of certain cellular sets in terms of their ``horizontal'' and ``vertical'' components.
The \emph{box product} $\square : \wreath \to \widehat{\Delta \wr \Delta} = \hattheta$ defined in \cref{box product} then realises such presentations into actual cellular sets.
As mentioned in the introduction, the latter functor should be thought of as analogous to the box product functor $\square : \hatdelta \times \hatdelta \to \widehat{\Delta \times \Delta}$ for bisimplicial sets, hence the name.
In \cref{O-anodyne} we will use these tools to turn simplicial inner horns into cellular ones.

We start by going back to the representable cellular sets and ``decomposing'' them into simplicial sets, to motivate the definition of $\wreath$.

Since the ``length'' of $\nq$ is $n$, the horizontal component of $\cell\nq$ should be $\Delta[n]$.
The description of the hom-categories of $\nq$ tells us that the vertical component of $\cell\nq$ should assign the product $\Delta[q_{k+1}] \times \dots \times \Delta[q_\ell]$ to each 1-simplex $\{k,\ell\}$ in $\Delta[n]$.
The resulting functor $\chi_1 : \Delta[n]_1 \to \hatdelta$ (where $\Delta[n]_m$ is the set of $m$-simplices in $\Delta[n]$ regarded as a discrete category) then encodes the $\Cat$-enriched graph structure of $\nq$.
The (free) horizontal composition is witnessed by the canonical isomorphism
\[
\chi_1(\alpha \cdot \{0,1\}) \times \chi_1(\alpha \cdot \{1,2\}) \to \chi_1(\alpha \cdot \{0,2\})
\]
for each 2-simplex $\alpha : [2] \to [n]$.
These isomorphisms can be organised into a single natural isomorphism
\[
\begin{tikzcd}
{\Delta[n]_2} \arrow [d, swap, "{- \cdot \{0,2\}}"] \arrow [r, "\chi_2"] \arrow [dr, phantom, "\cong"{description}] & \hatdelta \times \hatdelta \arrow [d, "\times"]\\
{\Delta[n]_1} \arrow [r, "\chi_1", swap] & \hatdelta
\end{tikzcd}
\]
where the right vertical map is the binary product functor and $\chi_2$ is the unique functor induced by the universal property as in:
\[
\begin{tikzcd}
{\Delta[n]_1}
\arrow [r, "\chi_1"] &
\hatdelta \\
{\Delta[n]_2}
\arrow [u, "{- \cdot \{0,1\}}"]
\arrow [r, "\chi_2", dashed]
\arrow [d, swap, "{- \cdot \{1,2\}}"] &
\hatdelta \times \hatdelta
\arrow [u, swap, "\pi_1"]
\arrow [d, "\pi_2"] \\
{\Delta[n]_1}
\arrow [r, swap, "\chi_1"] &
\hatdelta
\end{tikzcd}
\]
These three squares can be seen as part of a pseudo-natural transformation
	\[
	\begin{tikzpicture}
	\node at (0,0) {$\Delta\op$};
	\node at (3,0) {$\CAT$};
	\node at (1.5,2) {$\Set$};
	\draw[->] (0.3,0.4) --node[left, scale = 0.7]{$\Delta[n]$} (1.2,1.6);
	\draw[->] (0.6,0) --node[below, scale = 0.7]{$\hatdelta^{(-)}$} (2.4,0);
	\draw[{Hooks[right]}->] (1.8,1.6) -- (2.7,0.4);
	\draw[double, ->] (1.5,1.3) --node[right, scale = 0.7]{$\chi$} (1.5,0.3);
	\end{tikzpicture}
	\]
into the pseudo-functor $\hatdelta^{(-)}$ which we now describe.
(Here $\CAT$ must be large enough to contain $\hatdelta$ and its powers as objects.)

The object part of $\hatdelta^{(-)}$ assigns to each $[m] \in \Delta$ the product $\hatdelta^m = \hatdelta \times \dots \times \hatdelta$ of $m$ copies of the category $\hatdelta$.
If $\beta : [k] \to [m]$ is a simplicial operator, then its image $\hatdelta^\beta : \hatdelta^m \to \hatdelta^k$ acts by
\[
\{S_j\}_{1 \le j \le m} \mapsto\left\{\prod_{\beta(i-1)<j\le \beta(i)}S_j\right\}_{1 \le i \le k}.
\]
Since $\hatdelta^\gamma \hatdelta^\beta$ is only naturally isomorphic (via suitably coherent isomorphisms) and not equal to $\hatdelta^{\beta\gamma}$, we obtain a pseudo-functor $\Delta\op \to \CAT$ instead of a strict (2-)functor.

We define the $[m]$-component $\chi_m : \Delta[n]_m \to \hatdelta^m$ of the pseudo-natural transformation $\chi$ by
\[
\chi_m(\alpha) = \left\lbrace \prod_{\alpha(i-1)<j \le \alpha(i)}\Delta[q_j]\right\rbrace_{1 \le i \le m}
\]
for each $\alpha : [m] \to [n]$.
To complete the description of $\chi$, we need to specify an appropriately coherent family of natural isomorphisms
\[
\begin{tikzcd}
{\Delta[n]_m} \arrow [d, swap, "- \cdot \beta"] \arrow [r, "\chi_m"] \arrow [dr, phantom, "\cong"{description}] & \hatdelta^m \arrow [d, "\hatdelta^\beta"]\\
{\Delta[n]_k} \arrow [r, "\chi_k", swap] & \hatdelta^k
\end{tikzcd}
\]
indexed by the simplicial operators $\beta : [k] \to [m]$.
But this amounts to giving an isomorphism
\[
\prod_{0 <  j \le m}\chi_1(\alpha\cdot\{j-1,j\}) \cong \chi_1(\alpha\cdot\{0,m\})
\]
for each $\alpha \in \Delta[n]_m$ compatible with the simplicial structure of $\Delta[n]$, and one can check that the canonical isomorphisms indeed form such a compatible family.
As we mentioned above for the case $m=2$, this isomorphism can be thought of as witnessing the $m$-ary horizontal composition.
The invertibility of this map says that $\nq$ is horizontally free, and the compatibility with the simplicial structure says that the horizontal composition is coherent in the sense that it is associative, the witnesses to associativity satisfy the pentagon law, and so on.

This ``decomposition'' provides a motivation for thinking of the objects in the following category as presentations of certain cellular sets.
\begin{definition}
	For any simplicial set $W$, let $\bigl(\wreath\bigr)_W$ denote the category of pseudo-natural transformations
	\[
	\begin{tikzpicture}
	\node at (0,0) {$\Delta\op$};
	\node at (3,0) {$\CAT$};
	\node at (1.5,2) {$\Set$};
	\draw[->] (0.3,0.4) --node[left, scale = 0.7]{$W$} (1.2,1.6);
	\draw[->] (0.6,0) --node[below, scale = 0.7]{$\hatdelta^{(-)}$} (2.4,0);
	\draw[{Hooks[right]}->] (1.8,1.6) -- (2.7,0.4);
	\draw[double, ->] (1.5,1.3) --node[right, scale = 0.7]{$\chi$} (1.5,0.3);
	\end{tikzpicture}
	\]
	and modifications between them.
\end{definition}

A morphism $\chi \to \chi'$ in the category $\bigl(\wreath\bigr)_W$ essentially amounts to a family of simplicial maps $\chi_1(\alpha) \to \chi'_1(\alpha)$ indexed by $\alpha \in W_1$ that is compatible with the pseudo-naturality isomorphisms in an appropriate sense.
In particular, we have the following proposition.

\begin{proposition}\label{bicategorical Yoneda}
	There is an equivalence of categories
	\[
	\bigl(\wreath\bigr)_{\Delta[n]}\simeq \hatdelta^n
	\]
	whose object part is given by evaluating each pseudo-natural transformation at the unique non-degenerate $n$-simplex in $\Delta[n]$.
\end{proposition}
\begin{proof}
	This is an instance of the bicategorical Yoneda lemma \cite[\textsection1.9]{Street:fibrations}.
\end{proof}

	If $f : W \to W'$ is a map in $\hatdelta$, then there is a functor $f^*:\bigl(\wreath\bigr)_{W'} \to \bigl(\wreath\bigr)_W$ given by composing with $f$,	\emph{i.e.}~$f^*(\chi)$ is the pseudo-natural transformation:
	\[
	\begin{tikzpicture}
	\node at (0,0) {$\Delta\op$};
	\node at (3,0) {$\CAT$};
	\node at (1.5,2) {$\Set$};
	\draw[->] (0.4,0.3) .. controls (0.8,0.3) and (1.3,1) .. node[right, scale = 0.7, near start]{$W'$} (1.3,1.5);
	\draw[->] (0.2,0.5) .. controls (0.2,1) and (0.7,1.7) .. node[left, scale = 0.7]{$W$} (1.1,1.7);
	\draw[->] (0.6,0) --node[below, scale = 0.7]{$\hatdelta^{(-)}$} (2.4,0);
	\draw[{Hooks[right]}->] (1.8,1.6) -- (2.7,0.4);
	\draw[double, ->] (1.5,1.3) --node[right, scale = 0.7]{$\chi$} (1.5,0.3);
	\draw[double, ->] (0.55,1.15) --node[above, scale = 0.7]{$f$} (0.95,0.85);
	\end{tikzpicture}
	\]
	Moreover, sending each $f$ to $f^*$ defines a (strict) functor $\bigl(\wreath\bigr)_{(-)} : \hatdelta\op \to \CAT$.
\begin{definition}
	The \emph{generalised wreath product} $\wreath$ is the total category of the Grothendieck construction of the functor $\bigl(\wreath\bigr)_{(-)}$.
\end{definition}
More explicitly, the category $\wreath$ has as objects the pairs $(W,\chi)$ as above and as morphisms pairs $(f,\omega) : (W,\chi) \to (W',\chi')$ where $f : W \to W'$ is a morphism of simplicial sets and $\omega : \chi \to f^*(\chi')$ is a modification between the pseudo-natural transformations.

\begin{remark}
	For any monoidal category $\V$, one can construct a similar category $\hatdelta\wr\V$ by replacing the pseudo-functor $\hatdelta^{(-)}$ with $\V^{(-)}$ (whose morphism part is defined using the monoidal structure).
	In fact, Oury originally described $\wreath$ as a particular instance of this general construction.
\end{remark}

\subsection{The functors $\square$ and $\square_n$}\label{box product}
We start by making precise the ``decomposition'' of representable cellular sets discussed in the previous subsection.
\begin{proposition}[{\cite[Observation 3.53 and Lemma 3.60]{Oury}}]
	Sending each $\nq$ to the image of
	\[
	\bigl(\Delta[q_1],\dots,\Delta[q_n]\bigr) \in \hatdelta^n
	\]
	under the equivalence $\hatdelta^n \simeq \bigl(\wreath\bigr)_{\Delta[n]}$ of \cref{bicategorical Yoneda} defines the object part of a full embedding $\cell \incl \wreath$.
\end{proposition}
\begin{definition}
	The \emph{box product} $\square : \wreath \to \hattheta$ is the nerve functor induced by this embedding.
\end{definition}
Note that the embedding being full is equivalent to the composite
\[
\begin{tikzcd} \cell \arrow [r, hook] & \wreath \arrow [r, "\square"] & \hattheta \end{tikzcd}
\]\
being naturally isomorphic to the Yoneda embedding.

\begin{remark}
	We will briefly describe how Oury's box product functor is related to Rezk's \emph{intertwining functor} \cite[\textsection4.4]{Rezk:cartesian}
	\[
	V : \Delta \wr \bigl[\C\op,\hatdelta\bigr] \to \bigl[(\Delta \wr \C)\op, \hatdelta\bigr].
	\]
	If the reader is not familiar with Rezk's work on $\Theta_n$-spaces, they may safely ignore this remark.
	One can check that restricting the intertwining functor to the obvious ``discrete'' objects yields
	\[
	V : \Delta \wr \bigl[\C\op,\Set\bigr] \to \bigl[(\Delta \wr \C)\op, \Set\bigr]
	\]
	and so in particular we obtain $V : \Delta \wr \hatdelta \to \hattheta$ for $\C = \Delta$.
	The domain of this functor is equivalent to the full subcategory of $\wreath$ spanned by the objects of the form $(\Delta[n],\chi)$, and $\begin{tikzcd} \Delta \wr \hatdelta \arrow[r, hook] & \wreath \arrow [r, "\square"] & \hattheta \end{tikzcd}$ is naturally isomorphic to $V$.
\end{remark}

Given any cartesian fibration $P : \mathscr{E} \to \B$ and $B \in \B$, let $\B_{/B}$ and $\mathscr{E}_B$ denote the slice and the fibre over $B$ respectively.
Then there is a functor
\[
H : \B_{/B}\times\mathscr{E}_B \to \mathscr{E}
\]
whose object part is given by sending each pair $(f, E)$ to the domain $f^*E$ of a cartesian lift $\tilde f : f^*E \to E$ of $f$.
For any map $g : A_1 \to A_2$ over $B$ and any map $e : E_1 \to E_2$ in $\E_B$, we can factor $e \circ \tilde f_1$ uniquely through the cartesian lift $\tilde f_2$ as in
\[
\begin{tikzpicture}
\node at (1.5,0) {$f_2^*E_2$};
\node at (0,2) {$f_1^*E_1$};
\node at (3,2) {$E_1$};
\node at (3,1) {$E_2$};
\node at (7,0) {$A_2$};
\node at (5.5,2) {$A_1$};
\node at (8.5,2) {$B$};
\draw[->] (0.5,2) -- (2.5,2);
\draw[->, dashed] (0.3,1.6) -- (1.2,0.4);
\draw[->] (3,1.7) -- (3,1.3);
\draw[->] (1.8,0.2) -- (2.7,0.8);
\draw[->] (6,2) -- (8,2);
\draw[->] (5.8,1.6) -- (6.7,0.4);
\draw[->] (7.3,0.4) -- (8.2,1.6);
\node[scale = 0.7] at (1.5,2.2) {$\tilde f_1$};
\node[scale = 0.7] at (3.2,1.5) {$e$};
\node[scale = 0.7] at (2.5,0.3) {$\tilde f_2$};
\node[scale = 0.7] at (7,2.2) {$f_1$};
\node[scale = 0.7] at (8,0.9) {$f_2$};
\node[scale = 0.7] at (6,0.9) {$g$};
\draw[rounded corners] (-0.5,-0.5) rectangle (3.5,2.5);
\draw[rounded corners] (5,-0.5) rectangle (9,2.5);
\draw[|->] (3.8,1) -- (4.7,1);
\node at (4.25,1.2) {$P$};
\end{tikzpicture}
\]
and this defines the morphism part of $H$.

\begin{definition}
	Let $\square_n$ denote the composite functor
	\[
	\square_n :
	\begin{tikzcd}
	\hatdelta_{/\Delta[n]} \times \underbrace{\hatdelta \times \dots \times \hatdelta}_n
	\arrow [r] &
	\hatdelta_{/\Delta[n]} \times \bigl(\wreath\bigr)_{\Delta[n]}
	\arrow [r, "H"] &
	\wreath
	\arrow [r, "\square"] &
	\hattheta
	\end{tikzcd}
	\]
	where the first map is induced by the equivalence of \cref{bicategorical Yoneda} and the second map is an instance of the above construction.
\end{definition}

Note that we have $\square_n\bigl(\id_{\Delta[n]};\Delta[q_1],\dots,\Delta[q_n]\bigr) \cong \cell\nq$.
\begin{proposition}[{\cite[Lemmas 3.74 and 3.77]{Oury}}]
	The functor $\square_n$ preserves:
	\begin{itemize}
		\item small colimits in the first variable; and
		\item small connected colimits in each of the other $n$ variables.
	\end{itemize}
\end{proposition}

\begin{definition}
	If $f: X \to Y$ is a map in $\hatdelta$, then we will write
	\[
	[\id;f] : \cell[1;X] \to \cell[1;Y]
	\]
	for its image under the functor $\square_1(\id_{\Delta[1]};-) : \hatdelta \to \hattheta$.
\end{definition}
This notation is motivated by the fact that $\square_1(\id_{\Delta[1]};-)$ extends the functor $\Delta \to \hattheta$ given by sending $\alpha :[m] \to [n]$ to $[\id;\alpha]:\cell[1;m] \to \cell[1;n]$.
It takes a simplicial set $X$ to its ``suspension'', \emph{i.e.}~the nerve of the following simplicially enriched category:
\[
\begin{tikzcd}
0
\arrow [loop left, "{\Delta[0]}"]
\arrow [r, bend left, "X"] &
1
\arrow [l, bend left, "\varnothing"]
\arrow [loop right, "{\Delta[0]}"]
\end{tikzcd}
\]

\subsection{Oury's elementary anodyne extensions}\label{O-anodyne}
Joyal's model structure for quasi-categories on $\hatdelta$ can be characterised using:
\begin{itemize}
	\item the \emph{boundary inclusions} $\partial \Delta[n] \incl \Delta[n]$;
	\item the (\emph{inner}) \emph{horn inclusions} $\horn^k[n] \incl \Delta[n]$; and
	\item the \emph{equivalence extension} $e : \Delta[0] \incl J$ which is the nerve of the inclusion $\{\lozenge\} \incl \{\lozenge \cong \blacklozenge\}$ into the chaotic category on two objects.
\end{itemize}
Oury constructs the $\cell$-version of those morphisms using the \emph{Leibniz box product} $\hat \square_n$ as follows.

First, we describe the Leibniz construction.
Suppose $F : \C_1 \times \dots \times \C_n \to \D$ is a functor and $\D$ has finite colimits.
Then the (\emph{$n$-ary}) \emph{Leibniz construction}
\[
\hat F : \C_1^\two \times \dots \times \C_n^\two \to \D^\two
\]
of $F$, where $\two = \{0 \to 1\}$ is the ``walking arrow'' category, is defined as follows.
Let $f_i : X^0_i \to X^1_i$ be an object in $\C_i^\two$ for each $i$.
Then the assignment $(\epsilon_1, \dots, \epsilon_n) \mapsto F(X_1^{\epsilon_1}, \dots, X_n^{\epsilon_n})$ defines a functor $G : \two^n \to \D$.
Denote by $I$ the inclusion of the full subcategory of $\two^n$ spanned by all non-terminal objects.
Then $G$ defines a cone under the diagram $GI$, so we obtain an induced morphism $\colim GI \to F(X^1_1, \dots, X^1_n)$.
Sending $(f_1, \dots, f_n)$ to this morphism defines the object part of $\hat F$, and the morphism part is defined in the obvious way by the universal property.

\begin{definition}
The \emph{boundary inclusion} $\partial\cell\nq \incl \cell\nq$ is defined by the $(n+1)$-ary Leibniz construction
\[
\hat \square_n \left(
\begin{tikzcd}
{\partial\Delta[n]}
\arrow [d, hook]\\
{\Delta[n]}
\end{tikzcd};
\begin{tikzcd}
{\partial\Delta[q_1]}
\arrow [d, hook]\\
{\Delta[q_1]}
\end{tikzcd},
\dots,
\begin{tikzcd}
{\partial\Delta[q_n]}
\arrow [d, hook]\\
{\Delta[q_n]}
\end{tikzcd}
\right)
\]
where the first argument $\partial\Delta[n] \incl \Delta[n]$ is regarded as a map over $\Delta[n]$ in the obvious way.
\end{definition}
As its name suggests, this map is the ``usual'' boundary inclusion.
\begin{proposition}[{\cite[Observation 3.84]{Oury}}]\label{boundary is boundary}
The map $\partial\cell\nq \incl \cell\nq$ is (isomorphic to) the inclusion of the cellular subset consisting precisely of those maps into $\nq$ that factor through objects of lower dimension.
\end{proposition}
\begin{proposition}
The cellular subset $\partial\cell\nq \subset \cell\nq$ is generated by the hyperfaces of $\cell\nq$.
\end{proposition}
\begin{proof}
    This follows from \cref{Watson,boundary is boundary}.
\end{proof}

For example, when $\nq = [2;0,2]$ (see \cref{faces}):
\begin{itemize}
	\item $\square_2\bigl(\partial\Delta[2];\Delta[0],\Delta[2]\bigr) \subset \cell[2;0,2]$ is generated by $\delta_h^0$, $\delta_h^{1;\langle!,\id\rangle}$ and $\delta_h^2$;
	\item $\square_2\bigl(\Delta[2];\partial\Delta[0],\Delta[2]\bigr)$ is generated by $\delta_h^0$ and $[\{0\}]$; and
	\item $\square_2\bigl(\Delta[2];\Delta[0],\partial\Delta[2]\bigr)$ is generated by $\delta_v^{2;0}$, $\delta_v^{2;1}$ and $\delta_v^{2;2}$.
\end{itemize}
It can be seen from the defining colimit diagram that $\partial\cell[2;0,2]$ is the union of these three cellular subsets.
Thus $\partial\cell[2;0,2]$ is indeed generated by the hyperfaces of $\cell[2;0,2]$.

\begin{definition}
The \emph{$k$-th horizontal horn inclusion} $\horn_h^k\nq \incl \cell\nq$, where $0 \le k \le n$, is \[
\hat \square_n \left(
\begin{tikzcd}
{\horn^k[n]}
\arrow [d, hook]\\
{\Delta[n]}
\end{tikzcd};
\begin{tikzcd}
{\partial\Delta[q_1]}
\arrow [d, hook]\\
{\Delta[q_1]}
\end{tikzcd},
\dots,
\begin{tikzcd}
{\partial\Delta[q_n]}
\arrow [d, hook]\\
{\Delta[q_n]}
\end{tikzcd}
\right).
\]
It is called \emph{inner} if $1 \le k \le n-1$.
\end{definition}
\begin{proposition}\label{horizontal horn description}
The map $\horn_h^k\nq \incl \cell\nq$ is (isomorphic to) the inclusion of the cellular subset generated by all hyperfaces except for the $k$-th horizontal ones.
\end{proposition}
\begin{proof}
    It follows from \cref{boundary is boundary} and \cite[Lemma 3.11]{Oury} that this map is a monomorphism.
    Thus it suffices to check that it has the correct image, which can be done by considering the defining colimit diagram for $\horn_h^k\nq$.
\end{proof}
For example, when $\nq = [2;0,2]$ and $k = 1$:
\begin{itemize}
	\item $\square_2\bigl(\horn^1[2];\Delta[0],\Delta[2]\bigr)$ is generated by $\delta_h^0$ and $\delta_h^2$;
	\item $\square_2\bigl(\Delta[2];\partial\Delta[0],\Delta[2]\bigr)$ is generated by $\delta_h^0$ and $[\{0\}]$; and
	\item $\square_2\bigl(\Delta[2];\Delta[0],\partial\Delta[2]\bigr)$ is generated by $\delta_v^{2;0}$, $\delta_v^{2;1}$ and $\delta_v^{2;2}$.
\end{itemize}
Thus their union $\horn_h^1[2;0,2]$ is indeed generated by all hyperfaces except $\delta_h^{1;\langle!,\id\rangle}$.

\begin{remark}
	The faces $[\alpha;\aalpha] : \cell\mp \to \cell\nq$ not contained in the horizontal horn $\horn_h^k\nq$ are precisely the $k$-th horizontal ones.
	In particular, $\horn_h^k\nq$ may be missing faces of $\cell\nq$ that have codimension greater than $1$.
	For example, one can check that $\horn_h^1[2;1,1]$ is generated by the vertical hyperfaces
	\[
	\begin{gathered}
	\delta_v^{1;0} = \left\{
	\begin{tikzpicture}[baseline = -3]
	\filldraw
	(0,0) circle [radius = 1pt]
	(1,0) circle [radius = 1pt]
	(2,0) circle [radius = 1pt];
	\draw[->, gray!50!white] (0.1,0.1) .. controls (0.4,0.4) and (0.6,0.4) .. (0.9,0.1);
	\draw[->] (0.1,-0.1) .. controls (0.4,-0.4) and (0.6,-0.4) .. (0.9,-0.1);
	\draw[->] (1.1,0.1) .. controls (1.4,0.4) and (1.6,0.4) .. (1.9,0.1);
	\draw[->] (1.1,-0.1) .. controls (1.4,-0.4) and (1.6,-0.4) .. (1.9,-0.1);
	\draw[->, double, gray!50!white] (0.5,0.25)--(0.5,-0.25);
	\draw[->, double] (1.5,0.25)--(1.5,-0.25);
	\end{tikzpicture}\right\}, \hspace{10pt}
	\delta_v^{1;1} = \left\{
	\begin{tikzpicture}[baseline = -3]
	\filldraw
	(0,0) circle [radius = 1pt]
	(1,0) circle [radius = 1pt]
	(2,0) circle [radius = 1pt];
	\draw[->] (0.1,0.1) .. controls (0.4,0.4) and (0.6,0.4) .. (0.9,0.1);
	\draw[->, gray!50!white] (0.1,-0.1) .. controls (0.4,-0.4) and (0.6,-0.4) .. (0.9,-0.1);
	\draw[->] (1.1,0.1) .. controls (1.4,0.4) and (1.6,0.4) .. (1.9,0.1);
	\draw[->] (1.1,-0.1) .. controls (1.4,-0.4) and (1.6,-0.4) .. (1.9,-0.1);
	\draw[->, double, gray!50!white] (0.5,0.25)--(0.5,-0.25);
	\draw[->, double] (1.5,0.25)--(1.5,-0.25);
	\end{tikzpicture}\right\},\\
	\delta_v^{2;0} = \left\{
	\begin{tikzpicture}[baseline = -3]
	\filldraw
	(0,0) circle [radius = 1pt]
	(1,0) circle [radius = 1pt]
	(2,0) circle [radius = 1pt];
	\draw[->] (0.1,0.1) .. controls (0.4,0.4) and (0.6,0.4) .. (0.9,0.1);
	\draw[->] (0.1,-0.1) .. controls (0.4,-0.4) and (0.6,-0.4) .. (0.9,-0.1);
	\draw[->, gray!50!white] (1.1,0.1) .. controls (1.4,0.4) and (1.6,0.4) .. (1.9,0.1);
	\draw[->] (1.1,-0.1) .. controls (1.4,-0.4) and (1.6,-0.4) .. (1.9,-0.1);
	\draw[->, double] (0.5,0.25)--(0.5,-0.25);
	\draw[->, double, gray!50!white] (1.5,0.25)--(1.5,-0.25);
	\end{tikzpicture}\right\} \hspace{10pt} \text {and} \hspace{10pt}
	\delta_v^{2;1} = \left\{
	\begin{tikzpicture}[baseline = -3]
	\filldraw
	(0,0) circle [radius = 1pt]
	(1,0) circle [radius = 1pt]
	(2,0) circle [radius = 1pt];
	\draw[->] (0.1,0.1) .. controls (0.4,0.4) and (0.6,0.4) .. (0.9,0.1);
	\draw[->] (0.1,-0.1) .. controls (0.4,-0.4) and (0.6,-0.4) .. (0.9,-0.1);
	\draw[->] (1.1,0.1) .. controls (1.4,0.4) and (1.6,0.4) .. (1.9,0.1);
	\draw[->, gray!50!white] (1.1,-0.1) .. controls (1.4,-0.4) and (1.6,-0.4) .. (1.9,-0.1);
	\draw[->, double] (0.5,0.25)--(0.5,-0.25);
	\draw[->, double, gray!50!white] (1.5,0.25)--(1.5,-0.25);
	\end{tikzpicture}\right\}
	\end{gathered}
	\]
	and so it does not contain the face
	\[
	[\delta^1;\id,\id] = \left\{\begin{tikzpicture}[baseline = -3]
	\filldraw
	(0,0) circle [radius = 1pt]
	(2,0) circle [radius = 1pt];
	\draw[->, yshift = 1pt] (0.1,0.1) .. controls (0.4,0.4) and (0.6,0.4) .. (1,0) .. controls (1.4,0.4) and (1.6,0.4) .. (1.9,0.1);
	\draw[->, yshift = -1pt] (0.1,-0.1) .. controls (0.4,-0.4) and (0.6,-0.4) .. (1,0) .. controls (1.4,-0.4) and (1.6,-0.4) .. (1.9,-0.1);
	\draw[->, double] (0.5,0.25)--(0.5,-0.25);
	\end{tikzpicture}\right\}
	\]
	of codimension $2$.
	(The last face may equally well be depicted as $\left\{\begin{tikzpicture}[baseline = -3]
	\filldraw
	(0,0) circle [radius = 1pt]
	(2,0) circle [radius = 1pt];
	\draw[->, yshift = 1pt] (0.1,0.1) .. controls (0.4,0.4) and (0.6,0.4) .. (1,0) .. controls (1.4,0.4) and (1.6,0.4) .. (1.9,0.1);
	\draw[->, yshift = -1pt] (0.1,-0.1) .. controls (0.4,-0.4) and (0.6,-0.4) .. (1,0) .. controls (1.4,-0.4) and (1.6,-0.4) .. (1.9,-0.1);
	\draw[->, double] (1.5,0.25)--(1.5,-0.25);
	\end{tikzpicture}\right\}$; the position of the double arrow has no significance.)
	This differs from the more commonly found definition of a horn (\emph{e.g.}~\cite{Berger:nerve,Watson}) as ``boundary with one hyperface removed''.
	In \cref{alternative horns}, we show that for our purposes such alternative horns may be used in place of Oury's ones.
\end{remark}

\begin{definition}
The \emph{$(k;i)$-th vertical horn inclusion} $\horn_v^{k;i}\nq \incl \cell\nq$, where $0 \le k \le n$ satisfies $q_k \ge 1$ and $0 \le i \le q_k$, is
\[
\hat \square_n \left(
\begin{tikzcd}
{\partial\Delta[n]}
\arrow [d, hook]\\
{\Delta[n]}
\end{tikzcd};
\begin{tikzcd}
{\partial\Delta[q_1]}
\arrow [d, hook]\\
{\Delta[q_1]}
\end{tikzcd},
\dots,
\begin{tikzcd}
{\partial\Delta[q_{k-1}]}
\arrow [d, hook]\\
{\Delta[q_{k-1}]}
\end{tikzcd},
\begin{tikzcd}
{\horn^i[q_k]}
\arrow [d, hook]\\
{\Delta[q_k]}
\end{tikzcd},
\begin{tikzcd}
{\partial\Delta[q_{k+1}]}
\arrow [d, hook]\\
{\Delta[q_{k+1}]}
\end{tikzcd},
\dots,
\begin{tikzcd}
{\partial\Delta[q_n]}
\arrow [d, hook]\\
{\Delta[q_n]}
\end{tikzcd}
\right).
\]
It is called \emph{inner} if $1 \le i \le q_k-1$.
\end{definition}
The following proposition can be proved similarly to \cref{horizontal horn description}.
\begin{proposition}\label{vertical horn description}
The map $\horn_v^{k;i}\nq \incl \cell\nq$ is (isomorphic to) the inclusion of the cellular subset generated by all hyperfaces except for the $(k;i)$-th vertical ones.
\end{proposition}
For example, when $\nq= [2;0,2]$, $k = 2$ and $i = 1$:
\begin{itemize}
	\item $\square_2\bigl(\partial\Delta[2];\Delta[0],\Delta[2]\bigr)$ is generated by $\delta_h^0$, $\delta_h^{1;\langle!,\id\rangle}$ and $\delta_h^2$;
	\item $\square_2\bigl(\Delta[2];\partial\Delta[0],\Delta[2]\bigr)$ is generated by $\delta_h^0$ and $[\{0\}]$; and
	\item $\square_2\bigl(\Delta[2];\Delta[0],\horn^1[2]\bigr)$ is generated by $\delta_v^{2;0}$ and $\delta_v^{2;2}$.
\end{itemize}
Thus their union $\horn_v^{2;1}[2;0,2]$ is indeed generated by all hyperfaces except $\delta_v^{2;1}$.

\begin{definition}
A \emph{horizontal equivalence extension} is a map of the form
\[
\bigl(\cell[0] \overset{e}{\hookrightarrow} J\bigr) \hat \times \bigl(\partial\cell\nq \incl \cell\nq\bigr)
\]
where $\hat \times$ is the Leibniz construction of the usual binary product functor.
Here the simplicial set $J$ is regarded as a cellular set via the inclusion $\hatdelta \incl \hattheta$ described in \cref{cellular sets}.
\end{definition}

\begin{definition}
If $\nq \in \cell$ has $q_k = 0$ for some $1 \le k \le n$ then we denote by $\Psi^k\nq \incl \Phi^k\nq$ the \emph{vertical equivalence extension}
\[
\hat \square_n \left(
\begin{tikzcd}
{\partial\Delta[n]}
\arrow [d, hook]\\
{\Delta[n]}
\end{tikzcd};
\begin{tikzcd}
{\partial\Delta[q_1]}
\arrow [d, hook]\\
{\Delta[q_1]}
\end{tikzcd},
\dots,
\begin{tikzcd}
{\partial\Delta[q_{k-1}]}
\arrow [d, hook]\\
{\Delta[q_{k-1}]}
\end{tikzcd},
\begin{tikzcd}
{\Delta[0]}
\arrow [d, hook, "e"]\\
{J}
\end{tikzcd},
\begin{tikzcd}
{\partial\Delta[q_{k+1}]}
\arrow [d, hook]\\
{\Delta[q_{k+1}]}
\end{tikzcd},
\dots,
\begin{tikzcd}
{\partial\Delta[q_n]}
\arrow [d, hook]\\
{\Delta[q_n]}
\end{tikzcd}
\right).
\]
\end{definition}

\begin{definition}
	For any set $\mathcal{S}$ of morphisms in $\hattheta$, let $\celll(\mathcal{S})$ denote the closure of $\mathcal{S}$ under transfinite composition and taking pushouts along arbitrary maps.
\end{definition}

\begin{definition}
	Let $\H_h$, $\H_v$, $\E_h$, and $\E_v$ denote the sets of inner horizontal horn inclusions, inner vertical horn inclusions, horizontal equivalence extensions, and vertical equivalence extensions respectively.
	We write $\J_O$ for the union
	\[
	\J_O = \H_h \cup \H_v \cup \E_h \cup \E_v.
	\]
	By an \emph{O-anodyne extension} we mean an element $f$ of $\celll(\J_O)$, which is \emph{elementary} if $f \in \J_O$.
\end{definition}
One of Oury's main results is the following.
\begin{theorem}[{\cite[Corollary 3.11 and Theorem 4.22]{Oury}}]\label{Oury anodyne theorem}
	The O-anodyne extensions are stable under taking Leibniz products with arbitrary monomorphisms.
\end{theorem}

\subsection{Vertebrae and spines}
Here we introduce the notions of \emph{vertebra} and of \emph{spine}.
The only vertebra of $\cell[0]$ is the identity map.
For $\nq \in \cell$ with $n \ge 1$:
\begin{itemize}
	\item if $1 \le k \le n$ and $q_k = 0$, then 
	\[
	[\{k-1,k\};\id] : \cell[1;0] \to \cell\nq
	\]
	is a vertebra; and
	\item if $1 \le k \le n$ and $q_k \ge 1$, then for each $1 \le i \le q_k$,
	\[
	[\{k-1,k\};\{i-1,i\}] : \cell[1;1] \to \cell\nq
	\]
	is a vertebra.
\end{itemize}
For example, $\cell[2;0,2]$ has three vertebrae
\[
\left\{\begin{tikzpicture}[baseline = -3]
\filldraw
(0,0) circle [radius = 1pt]
(1,0) circle [radius = 1pt];
\draw[gray!50!white, fill = gray!50!white]
(2,0) circle [radius = 1pt];
\draw[->] (0.1,0)--(0.9,0);
\draw[->, gray!50!white] (1.1,0.15) .. controls (1.4,0.6) and (1.6,0.6) .. (1.9,0.15);
\draw[->, gray!50!white] (1.1,-0.15) .. controls (1.4,-0.6) and (1.6,-0.6) .. (1.9,-0.15);
\draw[->, gray!50!white] (1.1,0)--(1.9,0);
\draw[->, double, gray!50!white] (1.5,-0.05)--(1.5,-0.4);
\draw[->, double, gray!50!white] (1.5,0.4)--(1.5,0.05);
\end{tikzpicture}\right\}, \hspace{5pt}
\left\{\begin{tikzpicture}[baseline = -3]
\filldraw
(2,0) circle [radius = 1pt]
(1,0) circle [radius = 1pt];
\draw[gray!50!white, fill = gray!50!white]
(0,0) circle [radius = 1pt];
\draw[->, gray!50!white] (0.1,0)--(0.9,0);
\draw[->] (1.1,0.15) .. controls (1.4,0.6) and (1.6,0.6) .. (1.9,0.15);
\draw[->, gray!50!white] (1.1,-0.15) .. controls (1.4,-0.6) and (1.6,-0.6) .. (1.9,-0.15);
\draw[->] (1.1,0)--(1.9,0);
\draw[->, double, gray!50!white] (1.5,-0.05)--(1.5,-0.4);
\draw[->, double] (1.5,0.4)--(1.5,0.05);
\end{tikzpicture}\right\}, \hspace{5pt} \text {and} \hspace{5pt}
\left\{\begin{tikzpicture}[baseline = -3]
\filldraw
(2,0) circle [radius = 1pt]
(1,0) circle [radius = 1pt];
\draw[gray!50!white, fill = gray!50!white]
(0,0) circle [radius = 1pt];
\draw[->, gray!50!white] (0.1,0)--(0.9,0);
\draw[->, gray!50!white] (1.1,0.15) .. controls (1.4,0.6) and (1.6,0.6) .. (1.9,0.15);
\draw[->] (1.1,-0.15) .. controls (1.4,-0.6) and (1.6,-0.6) .. (1.9,-0.15);
\draw[->] (1.1,0)--(1.9,0);
\draw[->, double] (1.5,-0.05)--(1.5,-0.4);
\draw[->, double, gray!50!white] (1.5,0.4)--(1.5,0.05);
\end{tikzpicture}\right\}.
\]
Let $\spine\nq \subset \cell\nq$ denote the cellular subset generated by the vertebrae of $\cell\nq$, and call it the \emph{spine} of $\cell\nq$.

If $\nq$ is $[0]$, $[1;0]$ or $[1;1]$, then $\cell\nq$ has a unique vertebra and $\spine\nq = \cell\nq$.
We will call these cells \emph{mono-vertebral}; otherwise $\nq$ is \emph{poly-vertebral}.

Note that if $[\alpha;\aalpha] : \mp \to \nq$ is inert then it restricts to a map between the spines as in
\[
\begin{tikzcd}[row sep = large]
\spine\mp \arrow [d] \arrow [r, hook, "\subset"] & \cell\mp \arrow [d, "{[\alpha;\aalpha]}"] \\
\spine\nq \arrow [r, hook, "\subset"] & \cell\nq
\end{tikzcd}
\]
and moreover this square is a pullback.

Observe that we left the map $\spine\mp \to \spine\nq$ unlabelled in the above square.
In general, we adopt the following convention.
\begin{convention}
	Whenever we draw a square of the form
	\[
	\begin{tikzcd}
	\cdot 
	\arrow [r, hook, "\subset"]
	\arrow [d] &
	\cdot
	\arrow [d, "f"] \\
	\cdot 
	\arrow [r, hook, "\subset"] &
	\cdot
	\end{tikzcd}
	\]
	the unlabelled map is assumed to be the appropriate restriction of $f$.
	Typically the square is a gluing square (defined in \cref{gluing}) and $f$ is a map of the form $\delta : \cell\mp \to X$ where $X \subset \cell\nq$, but this convention is not restricted to such situations.
\end{convention}

\subsection{Ara's model structure on $\hattheta$}\label{model structure}
In \cite{Ara:nqcat}, Ara defines a model structure on $\widehat{\Theta_n}$ whose fibrant objects (called \emph{$n$-quasi-categories}) model $(\infty,n)$-categories.
Here we recall a characterisation of this model structure, but specialise to the case $n=2$.

Recall that $e$ denotes the nerve of the inclusion $\{\lozenge\}\incl\{\lozenge\cong\blacklozenge\}$ so that its suspension $[\id;e] : \cell[1;0] \to \cell[1;J]$ is (isomorphic to) the nerve of the 2-functor
\[
\left\{
\begin{tikzpicture}[baseline = -3]
\filldraw
(0,0) circle [radius = 1pt]
(1,0) circle [radius = 1pt];
\draw[->] (0.1,0.1) .. controls (0.4,0.4) and (0.6,0.4) .. (0.9,0.1);
\end{tikzpicture}\right\}
\incl
\left\{
\begin{tikzpicture}[baseline = -3]
\filldraw
(0,0) circle [radius = 1pt]
(1,0) circle [radius = 1pt];
\draw[->] (0.1,0.1) .. controls (0.4,0.4) and (0.6,0.4) .. (0.9,0.1);
\draw[->] (0.1,-0.1) .. controls (0.4,-0.4) and (0.6,-0.4) .. (0.9,-0.1);
\node[rotate = -90] at (0.5,0) {$\cong$};
\end{tikzpicture}\right\}
\]
whose codomain is locally chaotic.
Let $\J_A$ denote the union of $\E_h$ and the closure of
\[
\bigl\{\spine\nq \incl \cell\nq:\nq \in \cell\bigr\} \cup \bigl\{[\id;e]\bigr\}
\]
under taking Leibniz products
\[
(-)\hat \times \bigl(\cell[0] \amalg \cell[0] \incl J\bigr)
\]
with the nerve of $\{\lozenge\} \amalg \{\blacklozenge\} \incl \{\lozenge \cong \blacklozenge\}$.
We will call elements of $\J_A$ \emph{elementary A-anodyne extensions}.
\begin{theorem}[{\cite[\textsection 2.10 and \textsection 5.17]{Ara:nqcat}}]\label{Ara characterisation}
	There is a model structure on $\hattheta$ characterised by the following properties:
	\begin{itemize}
		\item the cofibrations are precisely the monomorphisms; and
		\item a map $f : X \to Y$ into a fibrant cellular set $Y$ is a fibration if and only if it has the right lifting property with respect to all maps in $\J_A$.
	\end{itemize}
\end{theorem} 

In particular, the fibrant objects, called \emph{2-quasi-categories}, are precisely those objects with the right lifting property with respect to all elementary A-anodyne extensions.

This is the only model structure on $\hattheta$ with which we are concerned in this paper, and hence no confusion should arise in the following when we simply refer to ``trivial cofibrations'' without further qualification.

\subsection{Gluing}\label{gluing}
This paper contains only two kinds of results:
\begin{itemize}
	\item [(i)] the inclusion $\J \subset \celll(\J')$ holds for certain sets $\J$ and $\J'$ of maps in $\hattheta$; and
	\item [(ii)] a certain set $\J$ of monomorphisms (= cofibrations) in $\hattheta$ is contained in the class of trivial cofibrations.
\end{itemize}
We prove the results of the first kind by directly expressing each map in $\J$ as a transfinite composite of pushouts of maps in $\celll(\J')$.
For those of the second kind, we make use of the \emph{right cancellation property}, \emph{i.e.}~we show that $f$ and $gf$ are trivial cofibrations and then deduce that the cofibration $g$ must also be trivial.
In each case, the proof reduces to checking the existence of certain \emph{gluing squares}, as defined below.

Suppose we have a pullback square
\[
\begin{tikzcd}
W \arrow [r, hook, "\subset"] \arrow [d] \arrow [dr, phantom, "\lrcorner", very near start] & X \arrow [d, "f"] \\
Y \arrow [r, hook, "\subset"] & Z
\end{tikzcd}
\]
in $\hattheta$ such that $Z = f(X) \cup Y$, and $f$ is injective on $f^{-1}(Z \setminus Y) = X \setminus W$.
Then the square is also a pushout, and we will say $Z$ is obtained from $Y$ by \emph{gluing $X$ along $W$}.
Note that if $Y$ is generated by a set $S$ of cells in $Z$, then $W$ is generated by the pullbacks of $\begin{tikzcd} \cell\nq \arrow [r, "s"] & Z \end{tikzcd}$ along $f$ for all $s \in S$.

\section{O-anodyne extensions and Ara's model structure}\label{O-anodyne extensions suffice}
Here we show Ara's model structure on $\hattheta$, which was characterised using the spine inclusions $\spine\nq \incl \cell\nq$, can be alternatively characterised using the inner horn inclusions.
More precisely, we prove that elementary A-anodyne extensions are O-anodyne extensions, and also (elementary) O-anodyne extensions are trivial cofibrations.

\subsection{Elementary A-anodyne extensions are O-anodyne extensions}
In this subsection, we prove the following lemma.
\begin{lemma}\label{A-anodyne implies O-anodyne}
	Every map in $\J_A$ is an O-anodyne extension.
\end{lemma}

Since the O-anodyne extensions are closed under taking Leibniz products with arbitrary monomorphisms (\cref{Oury anodyne theorem}), and $[\id;e] : \cell[1;0] \to \cell[1;J]$ is isomorphic to the elementary O-anodyne extension $\Psi^1[1;0] \incl \Phi^1[1;0]$, it suffices to show that the spine inclusions $\spine\nq \incl \cell\nq$ (which are the remaining ``generating'' elements of $\J_A$) are O-anodyne extensions.

The corresponding result for quasi-categories has been proved by Joyal \cite[Proposition 2.13]{Joyal:applications}.
Our proof presented below is essentially Joyal's proof repeated twice, first in the vertical direction and then in the horizontal direction.
In each step, we decompose the spine inclusion $\spine\nq \incl \cell\nq$ into three inclusions which, when $\nq = [3;\zzero]$, look like
\[
\left\{\begin{tikzpicture}[baseline = 8, scale = 0.8]
\filldraw
(0,0) circle [radius = 1pt]
(0.5,1) circle [radius = 1pt]
(1.5,1) circle [radius = 1pt]
(2,0) circle [radius = 1pt];
\draw[->] (0.05,0.1) -- (0.45,0.9);
\draw[->] (0.6,1) -- (1.4,1);
\draw[->] (1.55,0.9) -- (1.95,0.1);
\end{tikzpicture}\right\}
\incl
\left\{\begin{tikzpicture}[baseline = 8, scale = 0.8]
\filldraw[black!20!white]
(0,0) -- (0.5,1) -- (1.5,1) -- cycle;
\filldraw
(0,0) circle [radius = 1pt]
(0.5,1) circle [radius = 1pt]
(1.5,1) circle [radius = 1pt]
(2,0) circle [radius = 1pt];
\draw[->] (0.05,0.1) -- (0.45,0.9);
\draw[->] (0.6,1) -- (1.4,1);
\draw[->] (1.55,0.9) -- (1.95,0.1);
\draw[->] (0.09,0.06) -- (1.41,0.94);
\end{tikzpicture}\right\}
\incl
\left\{\begin{tikzpicture}[baseline = 8, scale = 0.8]
\filldraw[black!20!white]
(0,0) -- (0.5,1) -- (1.5,1) -- cycle
(2,0) -- (0.5,1) -- (1.5,1) -- cycle;
\filldraw[black!40!white]
(1,2/3) -- (0.5,1) -- (1.5,1) -- cycle;
\filldraw
(0,0) circle [radius = 1pt]
(0.5,1) circle [radius = 1pt]
(1.5,1) circle [radius = 1pt]
(2,0) circle [radius = 1pt];
\draw[->] (0.05,0.1) -- (0.45,0.9);
\draw[->] (0.6,1) -- (1.4,1);
\draw[->] (1.55,0.9) -- (1.95,0.1);
\draw[->] (0.09,0.06) -- (1.41,0.94);
\draw[->] (0.59,0.94) -- (1.91,0.06);
\end{tikzpicture}\right\}
\incl
\left\{\begin{tikzpicture}[baseline = 8, scale = 0.8]
\filldraw[black!50!white]
(0,0) -- (0.5,1) -- (1.5,1) -- (2,0) -- cycle;
\filldraw
(0,0) circle [radius = 1pt]
(0.5,1) circle [radius = 1pt]
(1.5,1) circle [radius = 1pt]
(2,0) circle [radius = 1pt];
\draw[->] (0.05,0.1) -- (0.45,0.9);
\draw[->] (0.6,1) -- (1.4,1);
\draw[->] (1.55,0.9) -- (1.95,0.1);
\draw[->] (0.09,0.06) -- (1.41,0.94);
\draw[->] (0.59,0.94) -- (1.91,0.06);
\draw[->] (0.1,0) -- (1.9,0);
\end{tikzpicture}\right\}.
\]
In general, the first two maps glue the outer faces along lower dimensional spine(-like) inclusions.
The remaining non-degenerate cells are precisely those containing both of the ``endpoints'' (\emph{i.e.}~$0,q_k \in[q_k]$ in the vertical case and $0,n \in [n]$ in the horizontal case).
We can group such cells into pairs $\{x,y\}$ so that the only difference between $x$ and $y$ is whether they contain $1$ (meaning $1 \in [q_k]$ in the vertical case and $1 \in [n]$ in the horizontal case).
Such a pair necessarily satisfies $y = x\cdot\delta^1$ (up to interchanging $x$ and $y$), \emph{e.g.}
\[
\left\{\begin{tikzpicture}[baseline = 8, scale = 0.8]
\filldraw[black!20!white]
(0,0) -- (0.5,1) -- (2,0) -- cycle;
\filldraw
(0,0) circle [radius = 1pt]
(0.5,1) circle [radius = 1pt]
(2,0) circle [radius = 1pt];
\draw[->] (0.05,0.1) -- (0.45,0.9);
\draw[->] (0.59,0.94) -- (1.91,0.06);
\draw[->] (0.1,0) -- (1.9,0);
\draw[->, white] (0.6,1) -- (1.4,1);
\end{tikzpicture}\right\}
\overset{\delta^1}{\longmapsto}
\left\{\begin{tikzpicture}[baseline = 8, scale = 0.8]
\filldraw
(0,0) circle [radius = 1pt]
(2,0) circle [radius = 1pt];
\draw[->] (0.1,0) -- (1.9,0);
\draw[->, white] (0.6,1) -- (1.4,1);
\end{tikzpicture}\right\}
\hspace{10pt} \text {and} \hspace{10pt}
\left\{\begin{tikzpicture}[baseline = 8, scale = 0.8]
\filldraw[black!50!white]
(0,0) -- (0.5,1) -- (1.5,1) -- (2,0) -- cycle;
\filldraw
(0,0) circle [radius = 1pt]
(0.5,1) circle [radius = 1pt]
(1.5,1) circle [radius = 1pt]
(2,0) circle [radius = 1pt];
\draw[->] (0.05,0.1) -- (0.45,0.9);
\draw[->] (0.6,1) -- (1.4,1);
\draw[->] (1.55,0.9) -- (1.95,0.1);
\draw[->] (0.09,0.06) -- (1.41,0.94);
\draw[->] (0.59,0.94) -- (1.91,0.06);
\draw[->] (0.1,0) -- (1.9,0);
\end{tikzpicture}\right\}
\overset{\delta^1}{\longmapsto}
\left\{\begin{tikzpicture}[baseline = 8, scale = 0.8]
\filldraw[black!20!white]
(0,0) -- (1.5,1) -- (2,0) -- cycle;
\filldraw
(0,0) circle [radius = 1pt]
(1.5,1) circle [radius = 1pt]
(2,0) circle [radius = 1pt];
\draw[->] (1.55,0.9) -- (1.95,0.1);
\draw[->] (0.09,0.06) -- (1.41,0.94);
\draw[->] (0.1,0) -- (1.9,0);
\draw[->, white] (0.6,1) -- (1.4,1);
\end{tikzpicture}\right\}.
\]
Thus the last inclusion can be obtained by gluing the $x$'s along $\horn^1$.

\begin{definition}
	If $S$ is any set of faces of $\cell\nq$, we will write $\spine^S\nq \subset \cell\nq$ for the cellular subset generated by $\spine\nq$ and $S$.
\end{definition}

\begin{proof}[Proof of \cref{A-anodyne implies O-anodyne}]
Recall that $\spine\nq \incl \cell\nq$ for mono-vertebral $\nq$ (\emph{i.e.}~for $\nq = [0]$, $[1;0]$ or $[1;1]$) is the identity and hence trivially O-anodyne.
These serve as the base cases for our induction.

We first consider the case where $n=1$.
For any $q \ge 1$, let $\spine^\dagger[1;q] = \spine^{\{\delta_v^{1;q}\}}[1;q]$ and let $\spine^\ddagger[1;q] = \spine^{\{\delta_v^{1;0}, \delta_v^{1;q}\}}[1;q]$.
We prove by induction on $q$ that each of the inclusions
\[
\spine[1;q] \incl \spine^\dagger[1;q] \incl \spine^\ddagger[1;q] \incl \cell[1;q]
\]
is an O-anodyne extension.

Assuming $q \ge 2$, the first inclusion fits into the gluing square
\[
\begin{tikzcd}[row sep = large]
\spine[1;q-1] \arrow [r, hook, "\subset"] \arrow [d] \glue & \cell[1;q-1] \arrow [d, "\delta_v^{1;q}"]\\
\spine[1;q] \arrow [r, hook, "\subset"] & \spine^\dagger[1;q]
\end{tikzcd}
\]
where the upper horizontal map is O-anodyne by the inductive hypothesis.
Similarly, the second inclusion fits into the following gluing square:
\[
\begin{tikzcd}[row sep = large]
\spine^\dagger[1;q-1] \arrow [r, hook, "\subset"] \arrow [d] \glue & \cell[1;q-1] \arrow [d, "\delta_v^{1;0}"]\\
\spine^\dagger[1;q] \arrow [r, hook, "\subset"] & \spine^\ddagger[1;q]
\end{tikzcd}
\]

Then a face map $[\id;\alpha] : [1;p] \to [1;q]$ corresponds to a cell in $\cell[1;q] \setminus \spine^\ddagger[1;q]$ if and only if $0, q \in \im\alpha$.
Thus the last inclusion can be obtained by gluing the faces corresponding to those $\alpha$ with $0,1,q \in \im\alpha$ along $\horn^{1;1}_v[1;p]$ in increasing order of $p$.
This completes the proof for the special case $n=1$.

Now consider the general case.
For any $\nq \in \cell$, let $\spine'\nq = \spine^{\{\delta_h^n\}}\nq$ and let $\spine''\nq = \spine^{\{\delta_h^0, \delta_h^n\}}\nq$.
We prove by induction on $\dim\nq$ that each of the inclusions
\[
\spine\nq \incl \spine'\nq \incl \spine''\nq \incl \cell\nq
\]
is an O-anodyne extension.

If $n=1$ then the first two inclusions are the identity and the last inclusion was treated above.
So we may assume $n \ge 2$, in which case the first inclusion fits into the gluing square
\[
\begin{tikzcd}[row sep = large]
\spine\nqd \arrow [r, hook, "\subset"] \arrow [d] \glue & \cell\nqd \arrow [d, "\delta_h^n"]\\
\spine\nq \arrow [r, hook, "\subset"] & \spine'\nq
\end{tikzcd}
\]
where  $\mathbf{q'} = (q_1,\dots,q_{n-1})$.
The upper horizontal map is O-anodyne by the inductive hypothesis, and so the lower map is also O-anodyne.
Similarly, the second inclusion fits into the gluing square
\[
\begin{tikzcd}[row sep = large]
\spine'\nqdd \arrow [r, hook, "\subset"] \arrow [d] \glue & \cell\nqdd \arrow [d, "\delta_h^0"]\\
\spine'\nq \arrow [r, hook, "\subset"] & \spine''\nq
\end{tikzcd}
\]
where $\mathbf{q''} = (q_2,\dots,q_n)$.

Then a face map $[\alpha;\aalpha] : \mp \to \nq$ corresponds to a cell in $\cell\nq \setminus \spine''\nq$ if and only if $0,n \in \im\alpha$.
Thus the last inclusion $\spine''\nq \incl \cell\nq$ can be obtained by gluing the faces corresponding to those $[\alpha;\aalpha]$ with $0,1,n \in \im\alpha$ along $\horn^1_h\mp$ in increasing order of $\dim\mp$.
This completes the proof for the general case.
\end{proof}

\subsection{Oury's inner horn inclusions are trivial cofibrations}\label{horn section}
The aim of this subsection is to prove the following lemma.
\begin{lemma}\label{horn inclusions are trivial cofibrations}
	Every map in $\H_h \cup \H_v$ is a trivial cofibration.
\end{lemma}

In fact, we will prove a wider class of ``generalised inner horn inclusions'' is contained in the trivial cofibrations.
These horns are constructed from the spines by filling lower dimensional horns.
Then the right cancellation property applied to $\spine\nq \incl \horn\nq \incl \cell\nq$ implies the second factor is a trivial cofibration.
This general strategy is the same as that adopted by Joyal and Tierney to prove the corresponding result for quasi-categories \cite[Lemma 3.5]{JT} although the combinatorics here is much more involved.

We start by gluing the outer hyperfaces of $\cell\nq$ to $\spine\nq$ according to the following total order $\prec$:
\[
\delta_v^{1;0} \prec \delta_v^{2;0} \prec \dots \prec \delta_v^{n;0} \prec \delta_h^0 \prec \delta_h^n \prec \delta_v^{1;q_1} \prec \delta_v^{2;q_2} \prec \dots \prec \delta_v^{n;q_n}.
\]
(Note that not all of these hyperfaces may exist.
The face $\delta_h^0$ (respectively $\delta_h^n$) is a hyperface only if $q_0 = 0$ (resp.~if $q_n = 0$), and the hyperfaces $\delta_v^{k;0}$ and $\delta_v^{k;q_k}$ exist only if $q_k \ge 1$.)

\begin{lemma}\label{part 1}
	The inclusion $\spine^S\nq \incl \cell\nq$ is a trivial cofibration for any $\nq \in \cell$ and for any set $S$ of outer hyperfaces of $\cell\nq$ that is downward closed with respect to $\prec$.
\end{lemma}
\begin{proof}
We proceed by induction on $|S|$.
Fix $\nq \in \cell$ and a downward closed set $S$ of outer hyperfaces of $\cell\nq$.
If $S$ is empty then $\spine^S\nq = \spine\nq$ and so the result follows trivially.
So suppose $|S| \ge 1$.
Let $\delta:\cell\mp \to \cell\nq$ be the $\prec$-maximum element in $S$ and let $S' = S \setminus \{\delta\}$.
Then $\spine^{S'}\nq \incl \cell\nq$ is a trivial cofibration by the inductive hypothesis, and hence it suffices to show $\spine^{S'}\nq \incl \spine^S\nq$ is also a trivial cofibration.
Since $\spine^S\nq$ can be obtained by gluing $\cell\mp$ to $\spine^{S'}\nq$ along the pullback $X$ in the gluing square
\[
\begin{tikzcd}[row sep = large]
X \arrow [r, hook, "\subset"] \arrow [d] \glue & \cell{\mp} \arrow [d, "{\delta}"] \\
\spine^{S'}{\nq} \arrow [r, hook, "\subset"] & \spine^S{\nq}
\end{tikzcd}
\]
this reduces to showing we have $X = \spine^T\mp$ for some downward closed set $T$ of outer hyperfaces of $\mp$ with $|T| < |S|$.
Since $\delta$ is an outer hyperface and hence inert, pulling back $\spine\nq$ along $\delta$ yields $\spine\mp$.
To describe the remaining cells in $X$, we have to consider the following cases separately.
\begin{itemize}
	\item[(1)] $\delta = \delta_v^{k;0}$:
	In this case $S' = \{\delta_v^{\ell;0}: \ell < k,~q_\ell \ge 1\}$.
	Thus $X$ is generated by $\spine\np$ (where $\pp = (q_1,\dots,q_k-1,\dots,q_n)$) and the pullbacks of these faces $\delta_v^{\ell;0} \in S'$ along $\delta_v^{k;0}$.
	For any $\ell < k$ with $q_\ell \ge 1$, the pullback of $\delta_v^{\ell;0}$ along $\delta_v^{k;0}$ is $\delta_v^{\ell;0}[n;\pp]$, \emph{i.e.}~the square
	\[
	\begin{tikzcd}[row sep = large, column sep = small]
	{\cell[n;q_1,\dots,q_\ell-1,\dots,q_k-1,\dots,q_n]}
	\arrow [r, "\delta_v^{\ell;0}"]
	\arrow [d, "\delta_v^{k;0}", swap]
	\pullback &
	{\cell[n;q_1,\dots,q_k-1,\dots,q_n]}
	\arrow [d, "\delta_v^{k;0}"] \\
	{\cell[n;q_1,\dots,q_\ell-1,\dots,q_n]}
	\arrow [r, "\delta_v^{\ell;0}", swap] &
	{\spine^S\nq}
	\end{tikzcd}
	\]
	is a pullback.
	Hence $X = \spine^T\np$ where
	\[
	T  = \bigl\{\delta_v^{\ell;0}\np: \ell < k,~q_\ell \ge 1\bigr\} = \bigl\{\delta_v^{\ell;0}\np: \ell < k,~p_\ell \ge 1\bigr\}.
	\]
	
	\item [(2)] $\delta = \delta_h^0$:
	In this case $S' = \{\delta_v^{k;0} : q_k \ge 1\}$.
	Note that since $\delta_h^0$ is a hyperface, we must have $q_1 = 0$ and hence $k \neq 1$ for all $\delta_v^{k;0} \in S'$.
	It then follows that the pullback of $\delta_v^{k;0}$ along $\delta_h^0$ is $\delta_v^{k-1;0}[n-1;\pp]$ where $\pp = (q_2,\dots,q_n)$, \emph{i.e.}~the square
	\[
	\begin{tikzcd}[row sep = large, column sep = small]
	{\cell[n-1;q_2,\dots,q_k-1,\dots,q_n]}
	\arrow [r, "\delta_v^{k-1;0}"]
	\arrow [d, "\delta_h^0", swap]
	\pullback &
	{\cell[n-1;q_2,\dots,q_n]}
	\arrow [d, "\delta_h^0"]\\
	{\cell[n;q_1,\dots,q_k-1,\dots,q_n]}
	\arrow [r, "\delta_v^{k;0}", swap] &
	{\spine^S\nq}
	\end{tikzcd}
	\]
	is a pullback.
	Therefore $X = \spine^T[n-1;\pp]$  and
	\[
	T = \bigl\{\delta_v^{k-1;0}[n-1;\pp] : q_k \ge 1\bigr\} = \bigl\{\delta_v^{k;0}[n-1;\pp] : p_k \ge 1\bigr\}.
	\]
	(The second equality holds because $p_{k-1} = q_k$.)
	
	\item [(3)]
	$\delta = \delta_h^n$:
	This case can be treated similarly to the previous one except we may have $\delta_h^0 \in S'$.
	If this is the case, the pullback of $\delta_h^0$ along $\delta_h^n$ is $\delta_h^0[n-1;\pp]$ where $\pp = (q_1,\dots,q_{n-1})$, hence $X = \spine^T[n-1;\pp]$ where
	\[
	T = \bigl\{\delta_v^{k;0}[n-1;\pp] : p_k \ge 1\bigr\} \cup \bigl\{\delta_h^0[n-1;\pp]\bigr\}.
	\]
	Since $p_1 = q_1 = 0$ (where the second equality follows from our assumption that $\delta_h^0 \in S'$), $\delta_h^0[n-1;\pp]$ is indeed a hyperface of $\cell[n-1;\pp]$.
	
	\item[(4a)] $\delta = \delta_v^{k;q_k}$ and $q_k \ge 2$:
	The pullback of $\delta_v^{\ell;0}$ along $\delta_v^{k;q_k}$ is $\delta_v^{\ell;0}\np$ (where $\pp = (q_1,\dots,q_k-1,\dots,q_n)$) for all $\ell$, and similarly for $\delta_v^{\ell;q_\ell}$.
	If $q_1 = 0$, then we know $k \neq 1$ and the pullback of $\delta_h^0 \in S'$ along $\delta_v^{k;q_k}$ is $\delta_h^0\np$. 
	Note in this case $\delta_h^0\np$ is a hyperface of $\cell\np$ since $p_1 = q_1 = 0$.
	Conversely, if $p_1 = 0$ then we must have $q_1 = 0$ and so $\delta_h^0 \in S'$.
	Similarly, $p_n = 0$ if and only if $q_n = 0$, in which case the pullback of $\delta_h^n \in S'$ along $\delta_v^{k;q_k}$ is the hyperface $\delta_h^n\np$.
	Therefore $X = \spine^T\np$ where:
	\begin{itemize}
		\item $\delta_v^{\ell;0}\np \in T$ iff $p_\ell \ge 1$;
		\item $\delta_v^{\ell;q_\ell}\np = \delta_v^{\ell;p_\ell}\np \in T$ iff $\ell < k$ and $p_\ell \ge 1$;
		\item $\delta_h^0\np \in T$ iff $p_1 = 0$; and
		\item $\delta_h^n\np \in T$ iff $p_n = 0$.
	\end{itemize}
	
	\item [(4b)] $\delta = \delta_v^{1;q_1}$ and $q_1 = 1$:
	The difference between this case and the previous one is that the pullback of $\delta_v^{1;0}$ along $\delta_v^{1;q_1} = \delta_v^{1;1}$ is generated by the horizontal hyperface $\delta_h^0\np$ of $\cell\np = \cell[n;0,q_2,\dots,q_n]$ and the point $[\{0\}]$.
	(This is essentially the intersection of two semicircles
	\[
	\left\{\begin{tikzpicture}[baseline = -3]
	\filldraw
	(0,0) circle [radius = 1pt]
	(1,0) circle [radius = 1pt];
	\draw[->] (0.1,0.1) .. controls (0.4,0.4) and (0.6,0.4) .. (0.9,0.1);
	\draw[->, gray!50!white] (0.1,-0.1) .. controls (0.4,-0.4) and (0.6,-0.4) .. (0.9,-0.1);
	\draw[->, double, gray!50!white] (0.5,0.25)--(0.5,-0.25);
	\end{tikzpicture}\right\}
	\cap
	\left\{\begin{tikzpicture}[baseline = -3]
	\filldraw
	(0,0) circle [radius = 1pt]
	(1,0) circle [radius = 1pt];
	\draw[->, gray!50!white] (0.1,0.1) .. controls (0.4,0.4) and (0.6,0.4) .. (0.9,0.1);
	\draw[->] (0.1,-0.1) .. controls (0.4,-0.4) and (0.6,-0.4) .. (0.9,-0.1);
	\draw[->, double, gray!50!white] (0.5,0.25)--(0.5,-0.25);
	\end{tikzpicture}\right\}
	=
	\left\{\begin{tikzpicture}[baseline = -3]
	\filldraw
	(0,0) circle [radius = 1pt]
	(1,0) circle [radius = 1pt];
	\draw[->, gray!50!white] (0.1,0.1) .. controls (0.4,0.4) and (0.6,0.4) .. (0.9,0.1);
	\draw[->, gray!50!white] (0.1,-0.1) .. controls (0.4,-0.4) and (0.6,-0.4) .. (0.9,-0.1);
	\draw[->, double, gray!50!white] (0.5,0.25)--(0.5,-0.25);
	\end{tikzpicture}\right\}
	\]
	horizontally composed with $[n-1;q_2,\dots,q_n]$.)
	Hence $X = \spine^T\np$ where:
	\begin{itemize}
		\item $\delta_v^{\ell;0}\np \in T$ iff $p_\ell \ge 1$;
		\item $\delta_h^0\np \in T$; and
		\item $\delta_h^n\np \in T$ iff $p_n = 0$.
	\end{itemize}
	
	\item [(4c)] $\delta = \delta_v^{n;q_n}$ and $q_n = 1$:
	This case is similar to the previous one, and we can deduce $X = \spine^T\np$ where $\pp = (q_1,\dots,q_{n-1},0)$ and:
	\begin{itemize}
		\item $\delta_v^{\ell;0}\np \in T$ iff $p_\ell \ge 1$;
		\item $\delta_v^{\ell;q_\ell}\np = \delta_v^{\ell;p_\ell}\np \in T$ iff $\ell < n$ and $p_\ell \ge 1$;
		\item $\delta_h^0\np \in T$ iff $p_1 = 0$; and
		\item $\delta_h^n\np \in T$.
	\end{itemize}
	
	\item [(4d)] $\delta = \delta_v^{k;q_k}$ for some $2 \le k \le n-1$ and $q_k = 1$:
	In this case, we have $\pp = (q_1,\dots,q_k-1,\dots,q_n)$ and the pullback of $\delta_v^{k;0}$ along $\delta_v^{k;q_k}$ is generated by
	\[
	[\{0,\dots,k-1\};\iid] : \cell[k-1;p_1,\dots,p_{k-1}] \to \cell\np
	\]
	and
	\[
	[\{k,\dots,n\};\iid] : \cell[n-k;p_{k+1},\dots,p_n] \to \cell\np.
	\]
	Observe that $[\{0,\dots,k-1\};\iid]$ is contained in the hyperface $\delta_h^n\np$ if $p_n = 0$, and in the hyperface $\delta_v^{n;0}\np$ if $p_n \ge 1$.
	Similarly, $[\{k,\dots,n\};\iid]$ is contained in $\delta_h^0\np$ or $\delta_v^{1;0}\np$.
	Therefore $X = \spine^T\np$ where:
	\begin{itemize}
		\item $\delta_v^{\ell;0}\np \in T$ iff $\ell \neq k$ and $p_\ell \ge 1$;
		\item $\delta_v^{\ell;q_\ell}\np = \delta_v^{\ell;p_\ell}\np \in T$ iff $\ell < k$ and $p_\ell \ge 1$;
		\item $\delta_h^0\np \in T$ iff $p_1 = 0$; and
		\item $\delta_h^n\np \in T$ iff $p_n = 0$.
	\end{itemize}
\end{itemize}
In each of these cases, it is straightforward to check that $T$ is a downward closed set of outer hyperfaces of $\cell\mp$.
Moreover, since the elements of $T$ are obtained by pulling back the elements in $S'$, we have $|T| \le |S'| < |S|$.
This completes the proof of \cref{part 1}.
\end{proof}

We are particularly interested in the instance of \cref{part 1} where $S$ is the set of all outer hyperfaces of $\cell\nq$.
Note that if $\nq$ is poly-vertebral (\emph{i.e.}~$\nq$ is not $[0]$, $[1;0]$ or $[1;1]$) then each vertebra of $\nq$ is an outer face.
Thus in this case it follows from \cref{inert factors through outer} that $\spine^S\nq$ is generated by the outer hyperfaces of $\cell\nq$ alone.
This is why the following definition does not mention the spine.

\begin{definition}\label{Upsilon}
	For any set $S$ of faces of $\cell\nq$, let $\Upsilon^S\nq \subset \cell\nq$ denote the cellular subset generated by all outer hyperfaces of $\cell\nq$ and the faces in $S$.
\end{definition}

We first consider the case where $S$ is some set of inner vertical hyperfaces.

\begin{definition}\label{admissible 1}
	A set $S$ of inner vertical hyperfaces of $\cell\nq$ is called \emph{admissible} if it is not the set of all inner hyperfaces.
\end{definition}
Note that if $S$ is a non-admissible set of inner vertical hyperfaces of $\cell\nq$, then all inner hyperfaces of $\cell\nq$ must be vertical.
Therefore we must have $n = 1$ and $S = \{\delta_v^{1;k} : 1 \le k \le q_1-1\}$.

\begin{lemma}\label{part 2}
	The inclusion $\Upsilon^S\nq \incl \cell\nq$ is a trivial cofibration for any poly-vertebral $\nq \in \cell$ and for any admissible set $S$ of inner vertical hyperfaces of $\cell\nq$.
\end{lemma}
\begin{proof}
Again, we proceed by induction on $|S|$.
If $S = \varnothing$ then the lemma follows from \cref{part 1}.
So we may assume $|S| \ge 1$.
Choose an element $\delta_v^{k;i} \in S$, which then necessarily satisfies $1 \le k \le n$ and $1 \le i \le q_k-1$.
Let $S' = S \setminus \{\delta_v^{k;i}\}$.
By a similar argument to that presented above for \cref{part 1}, what we must prove reduces to showing that $X$ in the gluing square
\[
\begin{tikzcd}[row sep = large]
X \arrow [r, hook, "\subset"] \arrow [d] \glue & \cell[n;\pp] \arrow [d, "\delta_v^{k;i}"] \\
\Upsilon^{S'}{\nq} \arrow [r, hook, "\subset"] & \Upsilon^S{\nq}
\end{tikzcd}
\]
is of the form $X = \Upsilon^T\np$ (where $\pp = (q_1,\dots,q_k-1,\dots,q_n)$) for some admissible set $T$ of inner vertical hyperfaces of $\cell\np$ with $|T| < |S|$.
Note that $\np$ must be poly-vertebral as the only cell with an inner vertical hyperface of mono-vertebral shape is $\cell[1;2]$, and for $\nq = [1;2]$ the only admissible $S$ is the empty set.

We first show that the square
\[
\begin{tikzcd}[row sep = large]
{\Upsilon^\varnothing\np}
\arrow [r, hook, "\subset"]
\arrow [d] &
{\cell\np}
\arrow [d, "\delta_v^{k;i}"] \\
{\Upsilon^\varnothing\nq}
\arrow [r, hook, "\subset"] &
{\Upsilon^S\nq}
\end{tikzcd}
\]
is a pullback.
Since $q_k \ge 2$ and $p_\ell = q_\ell$ for $\ell \neq k$, we have $p_1 = 0$ if and only if $q_1 = 0$.
Moreover, if $p_1 = q_1 = 0$ then the pullback of $\delta_h^0$ along $\delta_v^{k;i}$ is $\delta_h^0\np$.
Similarly, $p_n = 0$ if and only if $q_n = 0$, in which case the pullback of $\delta_h^n$ along $\delta_v^{k;i}$ is $\delta_h^n\np$.
For the outer vertical hyperfaces, if $q_\ell \ge 1$ and either $j = 0$ or $j = q_\ell$ then the pullback of $\delta_v^{\ell;j}$ along $\delta_v^{k;i}$ is $\delta_v^{\ell;j}\np$ except when $(\ell,j) = (k,q_k)$, in which case the pullback is $\delta_v^{k;q_k-1}\mp = \delta_v^{k;p_k}\mp$.
Thus the above square is indeed a pullback.

It then follows that $X = \Upsilon^T\np$ where $T$ consists of the pullbacks of elements of $S'$ along $\delta_v^{k;i}$.
Similarly to the outer case considered above, the pullback of $\delta_v^{\ell;j}\in S'$ along $\delta_v^{k;i}$ is $\delta_v^{\ell;j}\np$ except when $\ell = k$ and $j > i$, in which case the pullback is $\delta_v^{k;j-1}\np$.
Hence $T$ is a set of inner vertical hyperfaces of $\cell\np$.
Moreover pulling back along $\delta_v^{k;i}$ gives a bijection between $S'$ and $T$ and hence $|T| = |S|-1$.
Thus it remains to show that $T$ is admissible.
Suppose otherwise, then as we mentioned before the statement of \cref{part 2}, we must have $\np = [1;p_1]$ and
\[
|T| = \bigl|\{\delta_v^{1;\ell}[1;p_1]: 1 \le \ell \le p_1-1\}\bigr| = p_1-1.
\]
This implies $|S| = |T|+1 = p_1 = q_1-1$.
But then $S$ contains all of the inner hyperfaces of $\cell\nq = \cell[1;q_1]$, which contradicts our assumption that $S$ is admissible.
This completes the proof of \cref{part 2}.
\end{proof}

Now we consider the inner horizontal hyperfaces of $\cell\nq$.
Recall that for each $1 \le k \le n-1$, we have a family of $k$-th horizontal hyperfaces $\delta_h^{k;\langle\alpha,\alpha'\rangle}$ indexed by $\langle\alpha,\alpha'\rangle \in \shuffle(q_k,q_{k+1})$.
\begin{definition}
	If $S$ is a set of faces of $\cell\nq$, we define
	\[
	\shuffle_S(q_k,q_{k+1}) \defeq \left\{\langle\alpha,\alpha'\rangle \in \shuffle(q_k,q_{k+1}) : \delta_h^{k;\langle\alpha,\alpha'\rangle} \in S\right\}.
	\]
\end{definition}

\begin{definition}\label{admissible 2}
	A set $S$ of inner hyperfaces of $\cell\nq$ is called \emph{admissible} if:
\begin{itemize}
	\item[(i)] $S$ is not the set of all inner hyperfaces of $\cell\nq$;
	\item[(ii)] there is at most one $1 \le k \le n-1$ such that
	\[
	\varnothing \neq \shuffle_S(q_k,q_{k+1}) \neq \shuffle(q_k,q_{k+1})
	\]
	(we will write $k_S$ for such $k$ if it exists); and
	\item[(iii)] if $k_S$ exists, then $\shuffle_S(q_{k_S},q_{k_S+1})$ is downward closed with respect to the order described in \cref{simplicial sets}.
\end{itemize}
\end{definition}
Note that \cref{admissible 2} reduces to \cref{admissible 1} if $S$ contains no horizontal hyperfaces.

\begin{remark}
    The role of \cref{admissible 2}(iii) is to ensure that the intersections (meaning pullbacks) of the hyperfaces in $S$ are well-behaved so that we do not have to worry about faces of $\cell\nq$ of codimension larger than $2$.
    For example, consider the case $\nq = [2;2,1]$.
    There are three inner horizontal hyperfaces in this case, corresponding to the three $(2,1)$-shuffles $\langle\alpha,\alpha'\rangle < \langle\beta,\beta'\rangle < \langle\gamma,\gamma'\rangle$; graphically, the shuffles
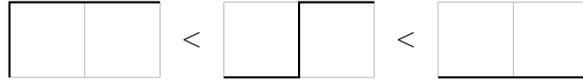
\[
\begin{tikzpicture}[baseline = 12]
\foreach \x in {0,1,2}
\draw[gray!50!white, thin] (\x,0) -- (\x,1);
\foreach \x in {0,1}
\draw[gray!50!white, thin] (0,\x) -- (2,\x);
\draw[thick] (0,0) -- (0,1) -- (2,1);
\end{tikzpicture}
\hspace{5pt} < \hspace{5pt}
\begin{tikzpicture}[baseline = 12]
\foreach \x in {0,1,2}
\draw[gray!50!white, thin] (\x,0) -- (\x,1);
\foreach \x in {0,1}
\draw[gray!50!white, thin] (0,\x) -- (2,\x);
\draw[thick] (0,0) -- (1,0) -- (1,1) -- (2,1);
\end{tikzpicture}
\hspace{5pt} < \hspace{5pt}
\begin{tikzpicture}[baseline = 12]
\foreach \x in {0,1,2}
\draw[gray!50!white, thin] (\x,0) -- (\x,1);
\foreach \x in {0,1}
\draw[gray!50!white, thin] (0,\x) -- (2,\x);
\draw[thick] (0,0) -- (2,0) -- (2,1);
\end{tikzpicture}
\]
correspond to the hyperfaces
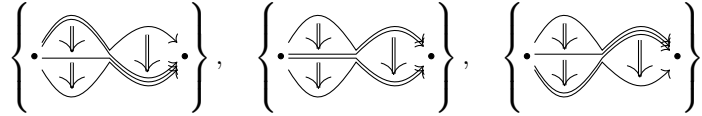
\[
\left\{
	\begin{tikzpicture}[baseline = -3]
	\filldraw
	(0,0) circle [radius = 1pt]
	(2,0) circle [radius = 1pt];
	\draw[->, yshift = 2.25pt] (0.1,0.1) .. controls (0.4,0.6) and (0.6,0.6) .. (1,0) .. controls (1.4,0.4) and (1.6,0.4) .. (1.9,0.1);
	\draw[->, yshift = 0.75pt] (0.1,0.1) .. controls (0.4,0.6) and (0.6,0.6) .. (1,0) .. controls (1.4,-0.4) and (1.6,-0.4) .. (1.9,-0.1);
	\draw[->, yshift = -0.75pt] (0.1,0) -- (1,0) .. controls (1.4,-0.4) and (1.6,-0.4) .. (1.9,-0.1);
	\draw[->, yshift = -2.25pt] (0.1,-0.1) .. controls (0.4,-0.6) and (0.6,-0.6) .. (1,0) .. controls (1.4,-0.4) and (1.6,-0.4) .. (1.9,-0.1);
	\draw[->, double] (0.5,0.4) -- (0.5,0.05);
	\draw[->, double, yshift = -1pt] (0.5,-0.05) -- (0.5,-0.4);
	\draw[->, double, yshift = 1pt] (1.5,0.25) -- (1.5,-0.25);
	\end{tikzpicture}\right\}, \hspace{10pt}
	\left\{
	\begin{tikzpicture}[baseline = -3]
	\filldraw
	(0,0) circle [radius = 1pt]
	(2,0) circle [radius = 1pt];
	\draw[->, yshift = 2.25pt] (0.1,0.1) .. controls (0.4,0.6) and (0.6,0.6) .. (1,0) .. controls (1.4,0.4) and (1.6,0.4) .. (1.9,0.1);
	\draw[->, yshift = 0.75pt] (0.1,0) -- (1,0) .. controls (1.4,0.4) and (1.6,0.4) .. (1.9,0.1);
	\draw[->, yshift = -0.75pt] (0.1,0) -- (1,0) .. controls (1.4,-0.4) and (1.6,-0.4) .. (1.9,-0.1);
	\draw[->, yshift = -2.25pt] (0.1,-0.1) .. controls (0.4,-0.6) and (0.6,-0.6) .. (1,0) .. controls (1.4,-0.4) and (1.6,-0.4) .. (1.9,-0.1);
	\draw[->, double, yshift = 1pt] (0.5,0.4) -- (0.5,0.05);
	\draw[->, double, yshift = -1pt] (0.5,-0.05) -- (0.5,-0.4);
	\draw[->, double] (1.5,0.25) -- (1.5,-0.25);
	\end{tikzpicture}\right\}, \hspace{10pt}
	\left\{
	\begin{tikzpicture}[baseline = -3]
	\filldraw
	(0,0) circle [radius = 1pt]
	(2,0) circle [radius = 1pt];
	\draw[->, yshift = 2.25pt] (0.1,0.1) .. controls (0.4,0.6) and (0.6,0.6) .. (1,0) .. controls (1.4,0.4) and (1.6,0.4) .. (1.9,0.1);
	\draw[->, yshift = 0.75pt] (0.1,0) -- (1,0) .. controls (1.4,0.4) and (1.6,0.4) .. (1.9,0.1);
	\draw[->, yshift = -0.75pt] (0.1,-0.1) .. controls (0.4,-0.6) and (0.6,-0.6) .. (1,0) .. controls (1.4,0.4) and (1.6,0.4) .. (1.9,0.1);
	\draw[->, yshift = -2.25pt] (0.1,-0.1) .. controls (0.4,-0.6) and (0.6,-0.6) .. (1,0) .. controls (1.4,-0.4) and (1.6,-0.4) .. (1.9,-0.1);
	\draw[->, double, yshift = 1pt] (0.5,0.4) -- (0.5,0.05);
	\draw[->, double] (0.5,-0.05) -- (0.5,-0.4);
	\draw[->, double, yshift = -1pt] (1.5,0.25) -- (1.5,-0.25);
	\end{tikzpicture}\right\}
\]
respectively.
The intersection of $\delta_h^{1;\langle\alpha,\alpha'\rangle}$ and $\delta_h^{1;\langle\gamma,\gamma'\rangle}$ is then the face
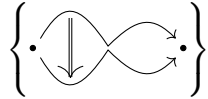
\[
\left\{
	\begin{tikzpicture}[baseline = -3]
	\filldraw
	(0,0) circle [radius = 1pt]
	(2,0) circle [radius = 1pt];
	\draw[->, yshift = 0.75pt] (0.1,0.1) .. controls (0.4,0.6) and (0.6,0.6) .. (1,0) .. controls (1.4,0.4) and (1.6,0.4) .. (1.9,0.1);
	\draw[->, yshift = -0.75pt] (0.1,-0.1) .. controls (0.4,-0.6) and (0.6,-0.6) .. (1,0) .. controls (1.4,-0.4) and (1.6,-0.4) .. (1.9,-0.1);
	\draw[->, double] (0.5,0.4) -- (0.5,-0.4);
	\end{tikzpicture}\right\}
\]
of codimension $3$, which is ``too small''.
If $S$ is an admissible set containing $\delta_h^{1;\langle\alpha,\alpha'\rangle}$ and $\delta_h^{1;\langle\gamma,\gamma'\rangle}$, then (iii) implies that $S$ also contains $\delta_h^{1;\langle\beta,\beta'\rangle}$.
Since this ``too small'' face is contained in the intersection of  $\delta_h^{1;\langle\alpha,\alpha'\rangle}$ (or $\delta_h^{1;\langle\gamma,\gamma'\rangle}$) and $\delta_h^{1;\langle\beta,\beta'\rangle}$, we may essentially disregard it.

There are two obviously downward closed subsets of $\shuffle(q_k,q_{k+1})$, namely $\varnothing$ and $\shuffle(q_k,q_{k+1})$.
\cref{admissible 2}(ii) asks that we always have one of these two subsets for any value of $k$, with a possible exception of $k = k_S$.
This simplifies the proof and in particular the descriptions of the sets $T_1$ and $T'_1$ defined below, but it is not essential.
Indeed, it seems possible to prove a variant of \cref{part 3} where (ii) is removed from \cref{admissible 2} and (iii) is replaced by:
\begin{itemize}
    \item[(iii')] $\shuffle_S(q_k,q_{k+1})$ is downward closed for all $1 \le k \le n-1$.
\end{itemize}
Although this modification makes \cref{part 3} slightly more general, we see no use in this extra generality.
\end{remark}

\begin{lemma}\label{part 3}
	The inclusion $\Upsilon^S\nq \incl \cell\nq$ is a trivial cofibration for any poly-vertebral $\nq \in \cell$ and for any admissible set $S$ of inner hyperfaces of $\cell\nq$.
\end{lemma}
\begin{proof}
Let $S_h \subset S$ denote the set of horizontal hyperfaces in $S$.
We proceed by induction on $\dim\nq$ and $|S_h|$.
If $S_h = \varnothing$ then the result follows from \cref{part 2}, so we may assume $|S_h|\ge 1$.
Choose $1 \le k \le n-1$ so that $S$ contains a $k$-th horizontal hyperface, where we take $k = k_S$ if the latter exists.
Let $\langle \alpha, \alpha' \rangle \in \shuffle_S (q_k,q_{k+1})$ be a maximal one.
Then $S' = S \setminus \bigl\{\delta_h^{k;\langle\alpha,\alpha'\rangle}\bigr\}$ is admissible, and so once again it suffices to prove that $X$ in the gluing square
\[
\begin{tikzcd}[row sep = large]
X \arrow [r, hook, "\subset"] \arrow [d] \glue & \cell[n-1;\pp] \arrow [d, "\delta_h^{k;\langle\alpha,\alpha'\rangle}"] \\
\Upsilon^{S'}{\nq} \arrow [r, hook, "\subset"] & \Upsilon^S{\nq}
\end{tikzcd}
\]
(where $\pp = (q_1,\dots,q_{k-1},q_k+q_{k+1},q_{k+2},\dots,q_n)$) is of the form $X = \Upsilon^T[n-1;\pp]$ for some admissible $T$.
By a similar argument to that presented in the proof of \cref{part 2}, $[n-1;\pp]$ must be poly-vertebral.

\begin{claim}\label{claim 0}
Let $Y \subset \cell[n-1;\pp]$ be the cellular subset defined by the following pullback square:
\[
\begin{tikzcd}[row sep = large]
Y
\pullback
\arrow [r, hook, "\subset"]
\arrow [d] &
{\cell[n-1;\pp]}
\arrow [d, "\delta_h^{k;\langle\alpha,\alpha'\rangle}"] \\
{\Upsilon^\varnothing\nq}
\arrow [r, hook, "\subset"] &
{\Upsilon^S\nq}
\end{tikzcd}
\]
Then $Y$ is generated by the outer hyperfaces of $\cell[n-1;\pp]$, \emph{i.e.}~$Y = \Upsilon^\varnothing[n-1;\pp]$.
\end{claim}
\begin{proof}
We first show the containment $Y \subset \Upsilon^\varnothing[n-1;\pp]$.
If $q_1 = 0$, then the pullback of the hyperface $\delta_h^0$ along $\delta_h^{k;\langle\alpha,\alpha'\rangle}$ is $\delta_h^0[n-1;\pp]$.
Since $\delta_h^0[n-1;\pp]$ is an outer face of $\cell[n-1;\pp]$ (of codimension $q_2+1$ if $k = 1$ and of codimension $1$ otherwise), it is contained in $\Upsilon^\varnothing[n-1;\pp]$ by \cref{inert factors through outer}.
The hyperface $\delta_h^n$ (if it exists) can be treated dually.

Next we consider the vertical hyperfaces of $\cell\nq$.
Fix $1 \le \ell \le n$ with $q_\ell \ge 1$.
Then the pullback of $\delta_v^{\ell;0}$ along $\delta_h^{k;\langle\alpha,\alpha'\rangle}$ is:
\begin{itemize}
	\item $\delta_v^{\ell;0}[n-1;\pp]$ if $\ell <k$;
	\item $\delta_v^{\ell-1;0}[n-1;\pp]$ if $\ell >k+1$; and
	\item contained in $\delta_v^{k;0}[n-1;\pp]$ if $\ell = k$ or $\ell = k+1$.
\end{itemize}
The hyperfaces $\delta_v^{\ell;q_\ell}$ can be treated dually.
This proves $Y \subset \Upsilon^\varnothing[n-1;\pp]$.

For the other containment $\Upsilon^\varnothing[n-1;\pp] \subset Y$, we must show that any outer hyperface of $\cell[n-1;\pp]$ can be obtained by pulling back some outer hyperface of $\cell\nq$ along $\delta_h^{k;\langle\alpha,\alpha'\rangle}$.
If $p_1 = 0$, then $q_1 = 0$ (because $p_1 = q_1$ if $k \neq 1$ and $p_1 = q_1 + q_2$ if $k = 1$) and the hyperface $\delta_h^0[n-1;\pp]$ is precisely the pullback of $\delta_h^0$ along $\delta_h^{k;\langle\alpha,\alpha'\rangle}$.
The other horizontal hyperface $\delta_h^{n-1}[n-1;\pp]$ (if it exists) can be treated dually.

Now we consider the vertical hyperfaces of $\cell[n-1;\pp]$.
Fix $1 \le \ell \le m = n-1$ with $p_\ell \ge 1$.
Then the hyperface $\delta_v^{\ell;0}[n-1;\pp]$ is the pullback (along $\delta_v^{k;\langle\alpha,\alpha'\rangle}$) of:
\begin{itemize}
	\item
	$\delta_v^{\ell;0}$ if $\ell < k$;
	\item
	$\delta_v^{\ell+1;0}$ if $\ell > k$;
	\item
	$\delta_v^{k;0}$ if $\ell = k$ and $\alpha(1) = 1$; and
	\item
	$\delta_v^{k+1;0}$ if $\ell = k$ and $\alpha'(1) = 1$.
\end{itemize}
Note that if $\delta_v^{k;0}[n-1;\pp]$ exists then $p_k \ge 1$ so $\alpha(1) \in [q_k]$ and $\alpha'(1) \in [q_{k+1}]$ are well-defined.
Moreover, $\langle \alpha, \alpha' \rangle \in \shuffle(q_k,q_{k+1})$ implies that we must have either $\alpha(1) = 1$ or $\alpha'(1) = 1$.
Thus the above list indeed covers all possible cases.

The remaining hyperfaces $\delta_v^{\ell;p_\ell}[n-1;\pp]$ can be treated dually, and this completes the proof of \cref{claim 0}.
\end{proof}

It now follows from the following claims that $X = \Upsilon^T[n-1;\pp]$ holds for
\[
T = T_1 \cup T'_1 \cup T_2 \cup T_3 \cup T'_3 \cup T_4 \cup T'_4
\]
where
\[
\begin{split}
T_1 &= \bigl\{\delta_h^{\ell;\langle \gamma, \gamma' \rangle} [n-1;\pp] : 1 \le \ell \le k-1,~\shuffle_S(q_\ell,q_{\ell+1}) = \shuffle(q_\ell,q_{\ell+1}),\\
&\hspace{100pt}\langle\gamma,\gamma'\rangle \in \shuffle(p_\ell,p_{\ell+1})\bigr\},\\
T'_1 &=\bigl\{\delta_h^{\ell-1;\langle \gamma, \gamma' \rangle} [n-1;\pp] : k+1 \le \ell \le n-1,~\shuffle_S(q_\ell,q_{\ell+1}) = \shuffle(q_\ell,q_{\ell+1}),\\
&\hspace{110pt}\langle\gamma,\gamma'\rangle \in \shuffle(p_{\ell-1},p_{\ell})\bigr\},\\
T_2 &= \bigl\{\delta_v^{k;j}[n-1;\pp]: j \in \lrcorner\langle\alpha,\alpha'\rangle\bigr\},\\
T_3 &= \bigl\{\delta_v^{\ell;j}[n-1;\pp]: 1 \le \ell<k,~\delta_v^{\ell;j} \in S'\bigr\},\\
T'_3 &= \bigl\{\delta_v^{\ell-1;j}[n-1;\pp]: k+1<\ell \le n,~\delta_v^{\ell;j} \in S'\bigr\},\\
T_4 &= \bigl\{\delta_v^{k;j}[n-1;\pp] : (\exists i \in [q_k])\left[\delta_v^{k;i} \in S',~\alpha^{-1}(i)=\{j\}\right]\bigr\}, \text{and}\\
T'_4 &= \bigl\{\delta_v^{k;j}[n-1;\pp] : (\exists i \in [q_{k+1}])\left[\delta_v^{k+1;i} \in S',~(\alpha')^{-1}(i)=\{j\}\right]\bigr\}.
\end{split}
\]
(See \cref{corners} for the definition of $\lrcorner\langle\alpha,\alpha'\rangle$.)
For each $1 \le m \le 4$, Claim $m$ below relates the elements in $T_m$ (and $T'_m$) to appropriate inner hyperfaces in $S'$.

\begin{claim}\label{claim 1}
	Fix $1 \le \ell \le k-1$.
	Then:
	\begin{itemize}
		\item [(i)] for any $\langle \beta, \beta' \rangle \in \shuffle(q_\ell,q_{\ell+1})$, each cell in the pullback of $\delta_h^{\ell;\langle\beta,\beta'\rangle}$ along $\delta_h^{k;\langle\alpha,\alpha'\rangle}$ is contained in some $\delta_h^{\ell;\langle\gamma,\gamma'\rangle}[n-1;\pp]$ ; and
		\item [(ii)] for any $\langle\gamma,\gamma'\rangle \in \shuffle(p_\ell,p_{\ell+1})$, the hyperface $\delta_h^{\ell;\langle\gamma,\gamma'\rangle}[n-1;\pp]$ is contained in the pullback of some $\delta_h^{\ell;\langle\beta,\beta'\rangle}$ along $\delta_h^{k;\langle\alpha,\alpha'\rangle}$.
	\end{itemize}
\end{claim}
The dual version of this claim relates, for $k+1 \le \ell \le n$, the $\ell$-th horizontal hyperfaces of $\cell\nq$ to the $(\ell-1)$-th horizontal hyperfaces of $\cell[n-1;\pp]$.

\begin{proof}
If $1 \le \ell < k-1$ (note the strict inequality) then both (i) and (ii) are straightforward since 
\[
\shuffle(p_\ell,p_{\ell+1}) = \shuffle(q_\ell,q_{\ell+1})
\]
and the pullback of $\delta_h^{\ell;\langle\beta,\beta'\rangle}$ along $\delta_h^{k;\langle\alpha,\alpha'\rangle}$ is precisely $\delta_h^{\ell;\langle\beta,\beta'\rangle}[n-1;\pp]$ for any $\langle\beta,\beta'\rangle \in \shuffle(q_\ell,q_{\ell+1})$.

Now we prove (i) for the case $\ell = k-1$.
Let $\langle\beta,\beta'\rangle \in \shuffle(q_{k-1},q_k)$ and suppose we are given a commutative square
\[
\begin{tikzcd}[row sep = large]
{[u;\mathbf{r}]}
\arrow [r, "{[\zeta;\boldsymbol{\zeta}]}"]
\arrow [d, "{[\xi;\boldsymbol{\xi}]}", swap] &
{[n-1;\pp]}
\arrow [d, "\delta_h^{k;\langle\alpha,\alpha'\rangle}"] \\
{[n-1;q_1,\dots,q_{k-1}+q_k,\dots,q_n]}
\arrow [r, "\delta_h^{k-1;\langle\beta,\beta'\rangle}", swap] &
{\nq}
\end{tikzcd}
\]
in $\cell$.
Then the square
\[
\begin{tikzcd}[row sep = large, column sep = small]
{[u]}
\arrow [r, "\zeta"]
\arrow [d, "\xi", swap] &
{[n-1]}
\arrow [d, "\delta^k"] \\
{[n-1]}
\arrow [r, "\delta^{k-1}", swap] &
{[n]}
\end{tikzcd}
\]
in $\Delta$ commutes so $k-1 \notin \im(\zeta)$.
We will assume there is some $1 \le v \le u$ such that $\zeta(v-1) \le k-2$ and $\zeta(v) \ge k$.
(Otherwise either $\zeta(v) \le k-2$ for all $v$ or $\zeta(v) \ge k$ for all $v$, and in either case $[\zeta;\boldsymbol{\zeta}]$ obviously factors through $\delta_h^{k-1;\langle\gamma,\gamma'\rangle}[n-1;\pp]$ for any $\langle\gamma,\gamma'\rangle \in \shuffle(p_{k-1},p_k)$.)
Since the $(p_{k-1},p_k)$-shuffles are the maximal non-degenerate simplices in $\Delta[p_{k-1}]\times\Delta[p_k]$, the map $\langle\zeta_{k-1},\zeta_k\rangle$ admits a factorisation
\[
\begin{tikzcd}
{\Delta[r_v]}
\arrow [r, "\phi"] &
{\Delta[p_k+p_{k+1}]}
\arrow [r, "{\langle\gamma,\gamma'\rangle}"] &
\Delta[p_{k-1}]\times\Delta[p_k]
\end{tikzcd}
\]
such that $\langle\gamma,\gamma'\rangle$ is a $(p_{k-1},p_k)$-shuffle.
Then $[\zeta;\boldsymbol{\zeta}]$ clearly factors through the hyperface $\delta_h^{k-1;\langle\gamma,\gamma'\rangle}[n-1;\pp]$.
This proves the first part of the claim for $\ell = k-1$.

For (ii), let $\langle \gamma,\gamma' \rangle \in \shuffle(p_{k-1},p_k)$.
Since the $(q_{k-1},q_k)$-shuffles are the maximal non-degenerate simplices in $\Delta[q_{k-1}]\times \Delta[q_k]$, the composite
\[
\begin{tikzcd}
{\Delta[p_{k-1}+p_k]}
\arrow [r, "{\langle\gamma,\gamma'\rangle}"] &
{\Delta[p_{k-1}]\times\Delta[p_k] = \Delta[q_{k-1}]\times\Delta[q_k+q_{k+1}]}
\arrow [r, "\id \times \alpha"] &
{\Delta[q_{k-1}]\times\Delta[q_k]}
\end{tikzcd}
\]
admits a factorisation
\[
\begin{tikzcd}
{\Delta[p_{k-1}+p_k] = \Delta[q_{k-1}+q_k+q_{k+1}]}
\arrow [r, "\zeta"]&
{\Delta[q_{k-1}+q_k]}
\arrow [r, "{\langle\beta,\beta'\rangle}"] &
{\Delta[q_{k-1}]\times\Delta[q_k]}
\end{tikzcd}
\]
such that $\langle\beta,\beta'\rangle$  is a $(q_{k-1},q_k)$-shuffle.
Then $\delta_h^{k-1;\langle\gamma,\gamma'\rangle}[n-1;\pp]$ is contained in the pullback of $\delta_h^{k-1;\langle\beta,\beta'\rangle}$ along $\delta_h^{k;\langle\alpha,\alpha'\rangle}$ since the square
\[
\begin{tikzpicture}[scale = 1.5]
\node at (-1,3){$[n-2;q_1,\dots,q_{k-1}+q_k+q_{k+1},\dots,q_n]$};
\node at (-2,1) {$[n-1;q_1,\dots,q_{k-1}+q_k,\dots,q_n]$};
\node at (1,2) {$[n-1;q_1,\dots,q_k+q_{k+1},\dots,q_n]$};
\node at (0,0) {$\nq$};

\draw[->] (-0.6,2.8)--(0.6,2.2);
\draw[->] (-1.2,2.6)--(-1.8,1.4);
\draw[->] (-1.6,0.8)--(-0.4,0.2);
\draw[->] (0.8,1.6)--(0.2,0.4);

\node[scale = 0.7] at (0.5,2.6) {$\delta_h^{k-1;\langle\gamma,\gamma'\rangle}$};
\node[scale = 0.7] at (-1.9,2.1) {$\delta_h^{k-1;\langle\zeta,\alpha'\gamma'\rangle}$};
\node[scale = 0.7] at (-1.4, 0.3) {$\delta_h^{k-1;\langle\beta,\beta'\rangle}$};
\node[scale = 0.7] at (0.9,1) {$\delta_h^{k;\langle\alpha,\alpha'\rangle}$};
\end{tikzpicture}
\]
commutes.
This completes the proof of \cref{claim 1}.
\end{proof}

\begin{claim}\label{claim 2}\hfill
	\begin{itemize}
		\item [(i)] For any $\langle \beta, \beta'\rangle \in \shuffle(q_k,q_{k+1})$ with $\langle\beta,\beta'\rangle \ngeq \langle\alpha,\alpha'\rangle$, the pullback of $\delta_h^{k;\langle\beta,\beta'\rangle}$ along $\delta_h^{k;\langle\alpha,\alpha'\rangle}$ is contained in $\delta_v^{k;j}[n-1;\pp]$ for some $j \in \lrcorner\langle\alpha,\alpha'\rangle$.
		\item [(ii)]For any $j \in \lrcorner\langle\alpha,\alpha'\rangle$, the hyperface $\delta_v^{k;j}[n-1;\pp]$ is the pullback of $\delta_h^{k;\langle\beta,\beta'\rangle}$ along $\delta_h^{k;\langle\alpha,\alpha'\rangle}$ for some $\langle \beta, \beta'\rangle \in \shuffle(q_k,q_{k+1})$ with $\langle\beta,\beta'\rangle < \langle\alpha,\alpha'\rangle$.
	\end{itemize}
\end{claim}
\begin{proof}
For (i), suppose $\langle \beta, \beta'\rangle$ is a $(q_k,q_{k+1})$-shuffle with $\langle\beta,\beta'\rangle \ngeq \langle\alpha,\alpha'\rangle$.
Then by \cref{dominating shuffle}, we can choose $j \in \lrcorner\langle\alpha,\alpha'\rangle$ such that $(\alpha(j),\alpha'(j)) \neq (\beta(j),\beta'(j))$.
The pullback of $\delta_h^{k;\langle\beta,\beta'\rangle}$ along $\delta_h^{k;\langle\alpha,\alpha'\rangle}$ is contained in $\delta_v^{k;j}[n-1;\pp]$.

To prove (ii), suppose $j \in \lrcorner\langle\alpha,\alpha'\rangle$.
Then the hyperface $\delta_v^{k;j}[n-1;\pp]$ is the pullback of $\delta_h^{k;\langle\beta,\beta'\rangle}$ along $\delta_h^{k;\langle\alpha,\alpha'\rangle}$ where $\langle\beta,\beta'\rangle$ is the $(q_k,q_{k+1})$-shuffle corresponding to $j$ under \cref{immediate predecessor of shuffle}.
Note that $\langle\beta,\beta'\rangle$ is an immediate predecessor of $\langle\alpha,\alpha'\rangle$ and so in particular $\langle\beta,\beta'\rangle < \langle\alpha,\alpha'\rangle$.
\end{proof}

\begin{claim}\label{claim 3}
	For any $1 \le \ell < k$ (respectively $k+1 < \ell \le n$) and $1 \le i \le q_\ell-1$, the pullback of $\delta_v^{\ell;i}$ along $\delta_h^{k;\langle\alpha,\alpha'\rangle}$ is $\delta_v^{\ell;i}[n-1;\pp]$ (resp.~$\delta_v^{\ell-1;i}[n-1;\pp]$).
\end{claim}
\begin{proof}
This is straightforward to check.
\end{proof}

\begin{claim}\label{claim 4}
	Fix $1 \le i \le q_k-1$ (respectively $1 \le i \le q_{k+1}-1$).
	Then the pullback of  $\delta_v^{k;i}$ (resp.~$\delta_v^{k+1;i}$) along $\delta_h^{k;\langle\alpha,\alpha'\rangle}$ is:
	\begin{itemize}
		\item precisely $\delta_v^{k;j}[n-1;\pp]$ if $\alpha^{-1}(i) = \{j\}$ (resp.~$(\alpha')^{-1}(i) =\{j\}$) for some $1 \le j \le p_k-1$; and
		\item contained in $\delta_v^{k;j}[n-1;\pp]$ for some $j \in \lrcorner\langle\alpha,\alpha'\rangle$ otherwise.
	\end{itemize}
\end{claim}
\begin{proof}
We will only consider the hyperfaces $\delta_v^{k;i}$ as $\delta_v^{k+1;i}$ can be treated dually.
The first case is straightforward to check.
In the second case, let $j = \min(\alpha^{-1}(i))$.
Then clearly the pullback of $\delta_v^{k;i}$ along $\delta_h^{k;\langle\alpha,\alpha'\rangle}$ is contained in $\delta_v^{k;j}[n-1;\pp]$, and so it remains to show that $j \in \lrcorner\langle\alpha,\alpha'\rangle$.
Note that $1 \le i \le q_k-1$ and $\alpha(j) = i$ imply $1 \le j \le p_k-1$.
Moreover, $\alpha(j-1) = \alpha(j)-1$ by our choice of $j$, and $\alpha(j+1)=i =\alpha(j)$ since $|\alpha^{-1}(i)| \ge 2$.
Therefore $j \in \lrcorner\langle\alpha,\alpha'\rangle$.
\end{proof}

Now we go back to the proof of \cref{part 3}.
We can deduce from \cref{claim 1,claim 2,claim 3,claim 4} that $X = \Upsilon^T[n-1;\pp]$.
It thus remains to prove that $T$ is an admissible set of inner hyperfaces of $\cell[n-1;\pp]$.
It is clear from our definitions of $T_1$ and $T'_1$ that, for any $1 \le \ell \le n-2$, either $\shuffle_T(p_\ell,p_{\ell+1}) = \varnothing$ or $\shuffle_T(p_\ell,p_{\ell+1}) = \shuffle(p_\ell,p_{\ell+1})$.
Thus $T$ satisfies \cref{admissible 2}(ii) and (iii).
To prove $T$ also satisfies (i), we will assume otherwise (\emph{i.e.}~$T$ contains all of the inner hyperfaces of $\cell[n-1;\pp]$) and deduce then $S$ does not satisfy (i), which is a contradiction.

For any $1 \le \ell \le k-1$, we have $\shuffle_T(p_\ell,p_{\ell+1}) = \shuffle(p_\ell,p_{\ell+1})$ and so our definition of $T_1$ implies $\shuffle_S(q_\ell,q_{\ell+1}) = \shuffle(q_\ell,q_{\ell+1})$.
Dually, we have $\shuffle_S(q_\ell,q_{\ell+1}) = \shuffle(q_\ell,q_{\ell+1})$ for all $k+1 \le \ell \le n$.
Thus $S$ contains all of the $\ell$-th horizontal hyperfaces of $\cell\nq$ for all $1 \le \ell \le n-1$ with $\ell \neq k$.

Next we consider the $k$-th horizontal hyperfaces of $\cell\nq$.
Note that since $S$ is admissible, $S$ contains all of the $k$-th horizontal hyperfaces if and only if $\langle \alpha, \alpha'\rangle$ is the maximum $(q_k,q_{k+1})$-shuffle.
We will prove this latter statement.
For any $1 \le j \le p_k-1$, $T$ contains $\delta_v^{k;j}[n-1;\pp]$ and so our definitions of $T_2$, $T_4$ and $T'_4$ imply that one of the following must hold:
\begin{itemize}
	\item $j \in \lrcorner\langle\alpha,\alpha'\rangle$;
	\item $\alpha(j')\neq \alpha(j)$ for all $j' \in [p_k]$ with $j' \neq j$; or
	\item $\alpha'(j')\neq \alpha'(j)$ for all $j' \in [p_k]$ with $j' \neq j$.
\end{itemize}
Therefore $\ulcorner\langle\alpha,\alpha'\rangle = \varnothing$, or equivalently, $\langle \alpha, \alpha'\rangle$ is the maximum $(q_k,q_{k+1})$-shuffle (by \cref{immediate predecessor of shuffle}).

Lastly, we consider the inner vertical hyperfaces of $\cell\nq$.
For any $1 \le \ell < k$ and for any $1 \le i \le q_\ell-1$, $T$ contains $\delta_v^{\ell;j}[n-1;\pp]$ and so our definition of $T_3$ implies that $\delta_v^{\ell;j} \in S$.
Dually, $\delta_v^{\ell;j} \in S$ for all $k+1 < \ell \le n$ and for all $1 \le j \le q_\ell-1$.
Note that $\langle\alpha,\alpha'\rangle \in \shuffle(q_k,q_{k+1})$ is the maximum one and so we have
\[
\begin{aligned}
	\alpha &=\{0,1,\dots,q_k,\underbrace{q_k,\dots,q_k}_{q_{k+1}\text { times}}\},\\
	\alpha' &= \{\underbrace{0,\dots,0}_{q_k\text { times}},0,1,\dots,q_{k+1}\}.
\end{aligned}
\]
Thus for each $1 \le i \le q_k-1$, $\delta_v^{k;i}[n-1;\pp]\in T$ and our definition of $T_4$ imply that $\delta_v^{k;i}\in S$.
Similarly, for each $1 \le i \le q_{k+1}-1$, $\delta_v^{k;q_k+i} \in T$ and our definition of $T'_4$ imply that $\delta_v^{k+1;i} \in S$.
This completes the proof of \cref{part 3}.
\end{proof}

\begin{proof}[Proof of \cref{horn inclusions are trivial cofibrations}]
The desired result follows from \cref{part 3} since setting
\[
S = \left\{\text {all inner hyperfaces of $\cell\nq$ except for $\delta_v^{k;i}$}\right\}
\]
yields $\Upsilon^S\nq = \horn_v^{k;i}\nq$ by \cref{vertical horn description} and setting
\[
S = \left\{\text{all inner hyperfaces of $\cell\nq$ except for the $k$-th horizontal ones}\right\}
\]
yields $\Upsilon^S\nq = \horn^k_h\nq$ by \cref{horizontal horn description} for the appropriate ranges of $k$ and $i$.
\end{proof}

\subsection{Vertical equivalence extensions are trivial cofibrations}
We will prove the following lemma in this subsection.
\begin{lemma}\label{vertical equivalence extensions}
	Every map in $\E_v$ is a trivial cofibration.
\end{lemma}
Recall that for any $\nq \in \cell$ and $1 \le k \le n$ with $q_k = 0$, the map $\Psi^k\nq \incl \Phi^k\nq$ is by definition the Leibniz box product
\[
\hat \square_n \left(
\begin{tikzcd}
{\partial\Delta[n]}
\arrow [d, hook]\\
{\Delta[n]}
\end{tikzcd};
\begin{tikzcd}
{\partial\Delta[q_1]}
\arrow [d, hook]\\
{\Delta[q_1]}
\end{tikzcd},
\dots,
\begin{tikzcd}
{\partial\Delta[q_{k-1}]}
\arrow [d, hook]\\
{\Delta[q_{k-1}]}
\end{tikzcd},
\begin{tikzcd}
{\Delta[0]}
\arrow [d, hook, "e"]\\
{J}
\end{tikzcd},
\begin{tikzcd}
{\partial\Delta[q_{k+1}]}
\arrow [d, hook]\\
{\Delta[q_{k+1}]}
\end{tikzcd},
\dots,
\begin{tikzcd}
{\partial\Delta[q_n]}
\arrow [d, hook]\\
{\Delta[q_n]}
\end{tikzcd}
\right)
\]
where $e$ is the nerve of the inclusion $\{\lozenge\} \incl \{\lozenge\cong\blacklozenge\}$.
Hence one of the legs in the defining colimit cone for $\Psi^k\nq$ is the (monic) map
\[
\cell\nq \cong \square_n(\Delta[n];\Delta[q_1],\dots,\Delta[q_{k-1}],\Delta[0],\Delta[q_{k+1}],\dots,\Delta[q_n]) \to \Psi^k\nq.
\]
In this subsection, we regard $\cell\nq$ as a cellular subset of $\Psi^k\nq$ via this map.

\begin{proof}
We will prove \cref{vertical equivalence extensions} by induction on $\dim\nq$.
Note that the base case is trivial since $\Psi^1[1;0] \incl \Phi^1[1;0]$ is isomorphic to the elementary A-anodyne extension $[\id;e]:\cell[1;0]\incl\cell[1;J]$; indeed, both of these maps are isomorphic to the nerve of the 2-functor that looks like:
\[
\left\{
\begin{tikzpicture}[baseline = -3]
\filldraw
(0,0) circle [radius = 1pt]
(1,0) circle [radius = 1pt];
\draw[->] (0.1,0.1) .. controls (0.4,0.4) and (0.6,0.4) .. (0.9,0.1);
\end{tikzpicture}\right\}
\incl
\left\{
\begin{tikzpicture}[baseline = -3]
\filldraw
(0,0) circle [radius = 1pt]
(1,0) circle [radius = 1pt];
\draw[->] (0.1,0.1) .. controls (0.4,0.4) and (0.6,0.4) .. (0.9,0.1);
\draw[->] (0.1,-0.1) .. controls (0.4,-0.4) and (0.6,-0.4) .. (0.9,-0.1);
\node[rotate = -90] at (0.5,0) {$\cong$};
\end{tikzpicture}\right\}.
\]
For the inductive step, it suffices to show that both $\cell\nq \incl \Psi^k\nq$ and $\cell\nq \incl \Phi^k\nq$ are trivial cofibrations.
These facts follow from \cref{step 0,step 1,step 2,step 3,step 4} which concern intermediate cellular subsets
\[
\cell\nq \subset X_0 \subset X_1 \subset X_2 \subset X_3 \subset \Psi^k\nq.
\]
\end{proof}

We will illustrate our argument below by providing pictures for the special case where $\nq = [3;1,0,0]$ and $k = 2$.
In this case $\Phi^k\nq$ and $\cell\nq$ look like:
\[
\Phi^2[3;1,0,0] = \left\{
\begin{tikzpicture}[baseline = -3]
\filldraw
(0,0) circle [radius = 1pt]
(1,0) circle [radius = 1pt]
(2,0) circle [radius = 1pt]
(3,0) circle [radius = 1pt];
\draw[->] (0.1,0.1) .. controls (0.4,0.4) and (0.6,0.4) .. (0.9,0.1);
\draw[->] (0.1,-0.1) .. controls (0.4,-0.4) and (0.6,-0.4) .. (0.9,-0.1);
\draw[->] (1.1,0.1) .. controls (1.4,0.4) and (1.6,0.4) .. (1.9,0.1);
\draw[->] (1.1,-0.1) .. controls (1.4,-0.4) and (1.6,-0.4) .. (1.9,-0.1);
\draw[->, double] (0.5,0.25)--(0.5,-0.25);
\node at (1.5,0) {\rotatebox{270}{$\cong$}};
\draw[->] (2.1,0)--(2.9,0);
\node[scale = 0.7] at (1.5, 0.5) {$\lozenge$};
\node[scale = 0.7] at (1.5, -0.5) {$\blacklozenge$};
\end{tikzpicture}\right\},
\hspace{5pt}
\cell[3;1,0,0] = \left\{
\begin{tikzpicture}[baseline = -3]
\filldraw
(0,0) circle [radius = 1pt]
(1,0) circle [radius = 1pt]
(2,0) circle [radius = 1pt]
(3,0) circle [radius = 1pt];
\draw[->] (0.1,0.1) .. controls (0.4,0.4) and (0.6,0.4) .. (0.9,0.1);
\draw[->] (0.1,-0.1) .. controls (0.4,-0.4) and (0.6,-0.4) .. (0.9,-0.1);
\draw[->] (1.1,0.1) .. controls (1.4,0.4) and (1.6,0.4) .. (1.9,0.1);
\draw[->, gray!50!white] (1.1,-0.1) .. controls (1.4,-0.4) and (1.6,-0.4) .. (1.9,-0.1);
\draw[->, double] (0.5,0.25)--(0.5,-0.25);
\node[gray!50!white] at (1.5,0) {\rotatebox{270}{$\cong$}};
\draw[->] (2.1,0)--(2.9,0);
\end{tikzpicture}\right\}.
\]

Fix $\nq \in \cell$ and $1 \le k \le n$ such that $n \ge 2$ and $q_k = 0$.
Note that an $(m;\pp)$-cell in $\Phi^k\nq$ consists of $\alpha : [m] \to [n]$ in $\Delta$ and $\alpha_{k'} : \Delta[p_{\ell}] \to V_{k'}$ for $\alpha(\ell-1)<{k'}\le\alpha(\ell)$ where
\[
(V_1,\dots,V_n) = \bigl(\Delta[q_1],\dots,\Delta[q_{k-1}],J,\Delta[q_{k+1}],\dots,\Delta[q_n]\bigr).
\]
Such $[\alpha;\aalpha]$ factors through:
\begin{itemize}
	\item[(\textasteriskcentered)] $\cell\nq$ unless there exists $1 \le \ell \le m$ such that $\alpha(\ell-1) < k \le \alpha(\ell)$ and $\blacklozenge \in \im\alpha_k$; and
	\item[(\textasteriskcentered\textasteriskcentered)] $\Psi^k\nq$ unless $\alpha$ and all $\alpha_{\ell}$ are surjective for $\ell \neq k$ and $\blacklozenge \in \im\alpha_k$.
\end{itemize}
We may assume $k \le n-1$ since the dual argument covers the case $k \ge 2$ and our assumption $n \ge 2$ implies that at least one of $k \le n-1$ and $k \ge 2$ must hold.

First, glue $\cell[1;J]$ to $\cell\nq$ as in the square
\[
\begin{tikzcd}[row sep = large]
{\cell[1;0]} \arrow [d, "{[\{k-1,k\};\id]}", swap] \arrow [r, "{[\id;e]}", hook] \glue & {\cell[1;J]} \arrow [d] \arrow [ddr, bend left, "{[\{k-1,k\};\id]}"] & \\
\cell\nq \arrow [r, hook] \arrow [drr, hook, bend right, "\subset", swap] & X_0 \arrow [dr, dashed, hook] & \\
& & \Psi^k\nq
\end{tikzcd}
\]
to obtain $X_0 \subset \Psi^k\nq$.
In our example, the image of $[\{1,2\};\id]$ looks like:
\[
\left\{
\begin{tikzpicture}[baseline = -3]
\filldraw[gray!50!white]
(0,0) circle [radius = 1pt]
(3,0) circle [radius = 1pt];
\filldraw
(1,0) circle [radius = 1pt]
(2,0) circle [radius = 1pt];
\draw[->, gray!50!white] (0.1,0.1) .. controls (0.4,0.4) and (0.6,0.4) .. (0.9,0.1);
\draw[->, gray!50!white] (0.1,-0.1) .. controls (0.4,-0.4) and (0.6,-0.4) .. (0.9,-0.1);
\draw[->] (1.1,0.1) .. controls (1.4,0.4) and (1.6,0.4) .. (1.9,0.1);
\draw[->] (1.1,-0.1) .. controls (1.4,-0.4) and (1.6,-0.4) .. (1.9,-0.1);
\draw[->, double, gray!50!white] (0.5,0.25)--(0.5,-0.25);
\node at (1.5,0) {\rotatebox{270}{$\cong$}};
\draw[->, gray!50!white] (2.1,0)--(2.9,0);
\node[scale = 0.7] at (1.5, 0.5) {$\lozenge$};
\node[scale = 0.7] at (1.5, -0.5) {$\blacklozenge$};
\end{tikzpicture}\right\}
\]
The following lemma records our construction of $X_0$.
\begin{lemma}\label{step 0}
    The inclusion $\cell\nq \incl X_0$ is a pushout of $[\id;e] : \cell[1;0] 
    \incl \cell[1;J]$.
\end{lemma}

Let $X_1 \subset \Phi^k\nq$ be the cellular subset generated by $X_0$ and those $(m;\pp$)-cells $[\alpha;\aalpha]$ satisfying $\alpha(m) = k$.
Since we are assuming $k \le n-1$, this condition $\alpha(m) = k$ implies $X_1 \subset \Psi^k\nq$.
Note that a non-degenerate $(m;\pp)$-cell $[\alpha;\aalpha]$ in $\Phi^k\nq$ is contained in $X_1 \setminus X_0$ if and only if it satisfies:
\begin{itemize}
	\item[(1a)] $\alpha(0) < k-1$;
	\item[(1b)] $\alpha(m) = k$; and
	\item[(1c)] $\blacklozenge \in \im \alpha_k$.
\end{itemize}
Observe that for any such $[\alpha;\aalpha]$, either it additionally satisfies
\begin{itemize}
	\item[(1d)] $\alpha(m-1) = k-1$
\end{itemize}
or there is a unique $(m';\mathbf{p'})$-cell $[\beta;\bbeta]$ in $X_1 \setminus X_0$ satisfying (1a-d) such that $[\alpha;\aalpha]$ is an $(m'-1)$-th horizontal face of $[\beta;\bbeta]$ (not necessarily of codimension 1).
\emph{e.g.}
\[
\left\{\begin{tikzpicture}[baseline = -3]
\filldraw
(0,0) circle [radius = 1pt]
(2,0) circle [radius = 1pt];
\filldraw[gray!50!white]
(3,0) circle [radius = 1pt];
\draw[->, gray!50!white] (0.1,0.1) .. controls (0.4,0.4) and (0.6,0.4) .. (0.9,0.1);
\draw[->] (0.1,-0.1) .. controls (0.4,-0.4) and (0.6,-0.4) .. (1,0) .. controls (1.4,-0.4) and (1.6,-0.4) .. (1.9,-0.1);
\draw[->, gray!50!white] (1.1,0.1) .. controls (1.4,0.4) and (1.6,0.4) .. (1.9,0.1);
\draw[->, double, gray!50!white] (0.5,0.25)--(0.5,-0.25);
\node[gray!50!white] at (1.5,0) {\rotatebox{270}{$\cong$}};
\draw[->, gray!50!white] (2.1,0)--(2.9,0);
\end{tikzpicture}\right\},
\hspace{5pt}
\left\{\begin{tikzpicture}[baseline = -3]
\filldraw
(0,0) circle [radius = 1pt]
(2,0) circle [radius = 1pt];
\filldraw[gray!50!white]
(3,0) circle [radius = 1pt];
\draw[->, yshift = 1pt] (0.1,0.1) .. controls (0.4,0.4) and (0.6,0.4) .. (1,0) .. controls (1.4,0.4) and (1.6,0.4) .. (1.9,0.1);
\draw[->, yshift = -1pt] (0.1,0.1) .. controls (0.4,0.4) and (0.6,0.4) .. (1,0) .. controls (1.4,-0.4) and (1.6,-0.4) .. (1.9,-0.1);
\draw[->, gray!50!white] (0.1,-0.1) .. controls (0.4,-0.4) and (0.6,-0.4) .. (0.9,-0.1);
\draw[->, double, gray!50!white] (0.5,0.25)--(0.5,-0.25);
\draw[->, double] (1.5,0.25)--(1.5,-0.25);
\draw[->, gray!50!white] (2.1,0)--(2.9,0);
\end{tikzpicture}\right\}
\hspace{5pt}\text {and}\hspace{5pt}
\left\{\begin{tikzpicture}[baseline = -3]
\filldraw
(0,0) circle [radius = 1pt]
(2,0) circle [radius = 1pt];
\filldraw[gray!50!white]
(3,0) circle [radius = 1pt];
\draw[->] (0.1,0.1) .. controls (0.4,0.4) and (0.6,0.4) .. (1,0) .. controls (1.4,-0.4) and (1.6,-0.4) .. (1.9,-0.1);
\filldraw[white] (1,-.05)--(0.95,0)--(1,.05)--(1.05,0)--cycle;
\draw[->] (0.1,-0.1) .. controls (0.4,-0.4) and (0.6,-0.4) .. (1,0) .. controls (1.4,0.4) and (1.6,0.4) .. (1.9,0.1);
\draw[->, double] (0.5,0.25)--(0.5,-0.25);
\draw[->, gray!50!white] (2.1,0)--(2.9,0);
\end{tikzpicture}\right\}
\]
are $1$st horizontal faces of
\[
\left\{\begin{tikzpicture}[baseline = -3]
\filldraw
(0,0) circle [radius = 1pt]
(1,0) circle [radius = 1pt]
(2,0) circle [radius = 1pt];
\filldraw[gray!50!white]
(3,0) circle [radius = 1pt];
\draw[->, gray!50!white] (0.1,0.1) .. controls (0.4,0.4) and (0.6,0.4) .. (0.9,0.1);
\draw[->] (0.1,-0.1) .. controls (0.4,-0.4) and (0.6,-0.4) .. (0.9,-0.1);
\draw[->, gray!50!white] (1.1,0.1) .. controls (1.4,0.4) and (1.6,0.4) .. (1.9,0.1);
\draw[->] (1.1,-0.1) .. controls (1.4,-0.4) and (1.6,-0.4) .. (1.9,-0.1);
\draw[->, double, gray!50!white] (0.5,0.25)--(0.5,-0.25);
\node[gray!50!white] at (1.5,0) {\rotatebox{270}{$\cong$}};
\draw[->, gray!50!white] (2.1,0)--(2.9,0);
\end{tikzpicture}\right\},
\hspace{5pt}
\left\{\begin{tikzpicture}[baseline = -3]
\filldraw
(0,0) circle [radius = 1pt]
(1,0) circle [radius = 1pt]
(2,0) circle [radius = 1pt];
\filldraw[gray!50!white]
(3,0) circle [radius = 1pt];
\draw[->] (0.1,0.1) .. controls (0.4,0.4) and (0.6,0.4) .. (0.9,0.1);
\draw[->, gray!50!white] (0.1,-0.1) .. controls (0.4,-0.4) and (0.6,-0.4) .. (0.9,-0.1);
\draw[->] (1.1,0.1) .. controls (1.4,0.4) and (1.6,0.4) .. (1.9,0.1);
\draw[->] (1.1,-0.1) .. controls (1.4,-0.4) and (1.6,-0.4) .. (1.9,-0.1);
\draw[->, double, gray!50!white] (0.5,0.25)--(0.5,-0.25);
\draw[->, double] (1.5,0.25)--(1.5,-0.25);
\draw[->, gray!50!white] (2.1,0)--(2.9,0);
\end{tikzpicture}\right\}
\hspace{5pt}\text {and}\hspace{5pt}
\left\{\begin{tikzpicture}[baseline = -3]
\filldraw
(0,0) circle [radius = 1pt]
(1,0) circle [radius = 1pt]
(2,0) circle [radius = 1pt];
\filldraw[gray!50!white]
(3,0) circle [radius = 1pt];
\draw[->] (0.1,0.1) .. controls (0.4,0.4) and (0.6,0.4) .. (0.9,0.1);
\draw[->] (0.1,-0.1) .. controls (0.4,-0.4) and (0.6,-0.4) .. (0.9,-0.1);
\draw[->] (1.1,0.1) .. controls (1.4,0.4) and (1.6,0.4) .. (1.9,0.1);
\draw[->] (1.1,-0.1) .. controls (1.4,-0.4) and (1.6,-0.4) .. (1.9,-0.1);
\draw[->, double] (0.5,0.25)--(0.5,-0.25);
\draw[->, double] (1.5,-0.25)--(1.5,0.25);
\draw[->, gray!50!white] (2.1,0)--(2.9,0);
\end{tikzpicture}\right\}
\]
respectively.

\begin{lemma}\label{step 1}
	The inclusion $X_0 \incl X_1$ is in $\celll(\H_h)$.
\end{lemma}
\begin{proof}
The discussion above shows that the set of non-degenerate cells in $X_1 \setminus X_0$ can be partitioned into subsets of the form
\[
\bigl\{
\text {$[\alpha;\aalpha]$ and all of its $(m-1)$-th horizontal faces}
\bigr\}
\]
where $[\alpha;\aalpha]$ is an $(m;\pp)$-cell satisfying (1a-d).
We prove that $X_1$ may be obtained from $X_0$ by gluing such $[\alpha;\aalpha]$ along the horizontal horn $\horn_h^{m-1}\mp$ in increasing order of $\dim\mp$.
Note that this horn is inner by (1a) and (1d).

Fix a non-degenerate $(m;\pp)$-cell $[\alpha;\aalpha]$ satisfying (1a-d).
We must show that any cell in the image of the composite
\[
\begin{tikzcd}
\horn_h^{m-1}\mp
\arrow [r, hook] &
\cell\mp
\arrow [r, "{[\alpha;\aalpha]}"] &
\Phi^k\nq
\end{tikzcd}
\]
is contained either in $X_0$ or in some cell that satisfies (1a-d) and has dimension strictly smaller than $\dim\mp$.
It suffices to check this for the generating faces of $\horn_h^{m-1}\mp$ described in \cref{horizontal horn description}:
\begin{itemize}
	\item $[\alpha;\aalpha] \cdot \delta_h^0$:
	\begin{itemize}
		\item is contained in $X_0$ if $\alpha(1) = k-1$; and
		\item satisfies (1a-d) otherwise;
	\end{itemize}
	\item $[\alpha;\aalpha] \cdot \delta_h^m$ is contained in $\cell\nq$;
	\item $[\alpha;\aalpha] \cdot \delta_h^{\ell;\langle\gamma,\gamma'\rangle}$ satisfies (1a-d) for any $\ell \le m-2$ and for any $\langle\gamma,\gamma'\rangle$;
	\item $[\alpha;\aalpha] \cdot \delta_v^{\ell;j}$ satisfies (1a-d) for any $\ell \neq m$ and for any $j \in [p_\ell]$; and
	\item $[\alpha;\aalpha] \cdot \delta_v^{m;j}$ is:
	\begin{itemize}
		\item contained in $\cell\nq$ if $\alpha_k^{-1}(\blacklozenge) = \{j\}$; and
		\item a (possibly trivial) degeneracy of some cell that satisfies (1a-d) otherwise.
	\end{itemize}
\end{itemize}
(By the trivial degeneracy of a cell, we mean the cell itself.
Also, the codomain of any $\delta$ appearing in the form $[\alpha;\aalpha] \cdot \delta$ in this proof is assumed to be $\mp$ so that $[\alpha;\aalpha] \cdot \delta$ is well-defined.)
This completes the proof.
\end{proof}

Next, let $X_2 \subset \Phi^k\nq$ be the cellular subset generated by $X_1$ and
those cells $[\alpha;\aalpha]$ such that $\delta^k : [n-1] \to [n]$ does not factor through $\alpha$.
Then clearly $X_2 \subset \Psi^k\nq$.
Note that a non-degenerate $(m;\pp)$-cell $[\alpha;\aalpha]$ in $\Phi^k\nq$ is contained in $X_2 \setminus X_1$ if and only if it satisfies:
\begin{itemize}
	\item[(2a)] $\alpha(0) \le k-1$;
	\item[(2b)] $\alpha(m) > k$;
	\item[(2c)] $\delta^k \neq \alpha \neq \id$; and
	\item[(2d)] $\blacklozenge \in \im \alpha_k$.
\end{itemize}
Observe that for any such $[\alpha;\aalpha]$, either it additionally satisfies
\begin{itemize}
	\item[(2e)] there exists $1 \le \ell_\alpha \le m-1$ such that $\alpha(\ell_\alpha) = k$
\end{itemize}
or there is a unique $(m';\mathbf{p'})$-cell $[\beta;\bbeta]$ in $X_2 \setminus X_1$ satisfying (2a-e) such that $[\alpha;\aalpha]$ is an $\ell_\beta$-th horizontal face of $[\beta;\bbeta]$.
\emph{e.g.}
\[
\left\{
\begin{tikzpicture}[baseline = -3]
\filldraw
(3,0) circle [radius = 1pt]
(1,0) circle [radius = 1pt];
\filldraw[gray!50!white]
(0,0) circle [radius = 1pt];
\draw[->, gray!50!white] (0.1,0.1) .. controls (0.4,0.4) and (0.6,0.4) .. (0.9,0.1);
\draw[->, gray!50!white] (0.1,-0.1) .. controls (0.4,-0.4) and (0.6,-0.4) .. (0.9,-0.1);
\draw[->, yshift = 1pt] (1.1,0.1) .. controls (1.4,0.4) and (1.6,0.4) .. (2,0) -- (2.9,0);
\draw[->, yshift = -1pt] (1.1,-0.1) .. controls (1.4,-0.4) and (1.6,-0.4) .. (2,0) -- (2.9,0);
\draw[->, double, gray!50!white] (0.5,0.25)--(0.5,-0.25);
\draw[->, double] (1.5,-0.25)--(1.5,0.25);
\end{tikzpicture}
\right\}, \hspace{5pt}
\left\{\begin{tikzpicture}[baseline = -3]
\filldraw
(0,0) circle [radius = 1pt]
(3,0) circle [radius = 1pt];
\draw[->] (0.1,0.1) .. controls (0.4,0.4) and (0.6,0.4) .. (1,0) .. controls (1.4,-0.4) and (1.6,-0.4) .. (2,0) -- (2.9,0);
\draw[->, gray!50!white] (0.1,-0.1) .. controls (0.4,-0.4) and (0.6,-0.4) .. (0.9,-0.1);
\draw[->, gray!50!white] (1.1,0.1) .. controls (1.4,0.4) and (1.6,0.4) .. (1.9,0.1);
\draw[->, double, gray!50!white] (0.5,0.25)--(0.5,-0.25);
\node[gray!50!white] at (1.5,0) {\rotatebox{270}{$\cong$}};
\end{tikzpicture}\right\},\hspace{5pt} \text {and} \hspace{5pt}
\left\{\begin{tikzpicture}[baseline = -3]
\filldraw
(0,0) circle [radius = 1pt]
(3,0) circle [radius = 1pt];
\draw[->] (0.1,0.1) .. controls (0.4,0.4) and (0.6,0.4) .. (1,0) .. controls (1.4,-0.4) and (1.6,-0.4) .. (1.95,-0.05)-- (2.9,-0.05);
\filldraw[white] (1,-.05)--(0.95,0)--(1,.05)--(1.05,0)--cycle;
\draw[->] (0.1,-0.1) .. controls (0.4,-0.4) and (0.6,-0.4) .. (1,0) .. controls (1.4,0.4) and (1.6,0.4) .. (1.95,0.05) -- (2.9,0.05);
\draw[->, double] (0.5,0.25)--(0.5,-0.25);
\end{tikzpicture}\right\}
\]
are $1$st horizontal faces of
\[
\left\{
\begin{tikzpicture}[baseline = -3]
\filldraw
(3,0) circle [radius = 1pt]
(1,0) circle [radius = 1pt]
(2,0) circle [radius = 1pt];
\filldraw[gray!50!white]
(0,0) circle [radius = 1pt];
\draw[->, gray!50!white] (0.1,0.1) .. controls (0.4,0.4) and (0.6,0.4) .. (0.9,0.1);
\draw[->, gray!50!white] (0.1,-0.1) .. controls (0.4,-0.4) and (0.6,-0.4) .. (0.9,-0.1);
\draw[->] (1.1,0.1) .. controls (1.4,0.4) and (1.6,0.4) .. (1.9,0.1);
\draw[->] (1.1,-0.1) .. controls (1.4,-0.4) and (1.6,-0.4) .. (1.9,-0.1);
\draw[->, double, gray!50!white] (0.5,0.25)--(0.5,-0.25);
\draw[->, double] (1.5,-0.25)--(1.5,0.25);
\draw[->] (2.1,0)--(2.9,0);
\end{tikzpicture}
\right\}, \hspace{5pt}
\left\{\begin{tikzpicture}[baseline = -3]
\filldraw
(0,0) circle [radius = 1pt]
(3,0) circle [radius = 1pt]
(2,0) circle [radius = 1pt];
\draw[->] (0.1,0.1) .. controls (0.4,0.4) and (0.6,0.4) .. (1,0) .. controls (1.4,-0.4) and (1.6,-0.4) .. (1.9,-0.1);
\draw[->, gray!50!white] (0.1,-0.1) .. controls (0.4,-0.4) and (0.6,-0.4) .. (0.9,-0.1);
\draw[->, gray!50!white] (1.1,0.1) .. controls (1.4,0.4) and (1.6,0.4) .. (1.9,0.1);
\draw[->, double, gray!50!white] (0.5,0.25)--(0.5,-0.25);
\node[gray!50!white] at (1.5,0) {\rotatebox{270}{$\cong$}};
\draw[->] (2.1,0)--(2.9,0);
\end{tikzpicture}\right\},\hspace{5pt} \text {and} \hspace{5pt}
\left\{\begin{tikzpicture}[baseline = -3]
\filldraw
(0,0) circle [radius = 1pt]
(3,0) circle [radius = 1pt]
(2,0) circle [radius = 1pt];
\draw[->] (0.1,0.1) .. controls (0.4,0.4) and (0.6,0.4) .. (1,0) .. controls (1.4,-0.4) and (1.6,-0.4) .. (1.9,-0.1);
\filldraw[white] (1,-.05)--(0.95,0)--(1,.05)--(1.05,0)--cycle;
\draw[->] (0.1,-0.1) .. controls (0.4,-0.4) and (0.6,-0.4) .. (1,0) .. controls (1.4,0.4) and (1.6,0.4) .. (1.9,0.1);
\draw[->, double] (0.5,0.25)--(0.5,-0.25);
\draw[->] (2.1,0)--(2.9,0);
\end{tikzpicture}\right\}
\]
respectively.
\begin{lemma}\label{step 2}
	The inclusion $X_1 \incl X_2$ is in $\celll(\H_h)$.
\end{lemma}
\begin{proof}
The discussion above shows that the set of non-degenerate cells in $X_2 \setminus X_1$ can be partitioned into subsets of the form
\[
\bigl\{\text {$[\alpha;\aalpha]$ and all of its $\ell_\alpha$-th horizontal faces} \bigr\}
\]
where $[\alpha;\aalpha]$ is an $(m;\pp)$-cell satisfying (2a-e).
We prove that $X_2$ may be obtained from $X_1$ by gluing such $[\alpha;\aalpha]$ along the horizontal horn $\horn_h^{\ell_\alpha}\mp$ in increasing order of $\dim \mp$.
Note that this horn is inner by (2e).

Fix a non-degenerate $(m;\pp)$-cell $[\alpha;\aalpha]$ satisfying (2a-e).
Similarly to the proof of \cref{step 1}, we must check that the following faces of $[\alpha;\aalpha]$ are contained either in $X_1$ or in some cell that satisfies (2a-e) and has dimension strictly smaller than $\dim\mp$:
\begin{itemize}
	\item $[\alpha;\aalpha] \cdot \delta_h^0$:
	\begin{itemize}
		\item is contained in $\cell\nq$ if $\alpha(1) = k$; and
		\item satisfies (2a-e) otherwise;
	\end{itemize}
	\item $[\alpha;\aalpha] \cdot \delta_h^m$:
	\begin{itemize}
		\item is contained in $X_1$ if $\alpha(m-1) = k$; and
		\item satisfies (2a-e) otherwise;
	\end{itemize}
	\item $[\alpha;\aalpha] \cdot \delta_h^{\ell;\langle\gamma,\gamma'\rangle}$ satisfies (2a-e) for any $\ell \neq \ell_\alpha$ and for any $\langle\gamma,\gamma'\rangle$;
	\item $[\alpha;\aalpha] \cdot \delta_v^{\ell;j}$ satisfies (2a-e) for any $\ell \neq \ell_\alpha$ and for any $j \in [p_\ell]$; and
	\item $[\alpha;\aalpha] \cdot \delta_v^{\ell_\alpha;j}$, for any $j \in [p_\ell]$, is:
	\begin{itemize}
		\item contained in $\cell\nq$ if $\alpha_k^{-1}(\blacklozenge) = \{j\}$; and
		\item a (possibly trivial) degeneracy of some cell that satisfies (2a-e) otherwise.
	\end{itemize}
\end{itemize}
This completes the proof.
\end{proof}

Now let $X_3 \subset \Phi^k\nq$ be the cellular subset generated by $X_2$ and those cells $[\alpha;\aalpha]$ such that $\alpha_{\ell}$ is not surjective for some $\ell \neq k$.
Then clearly $X_3 \subset \Psi^k\nq$.
Note that a non-degenerate $(m;\pp)$-cell $[\alpha;\aalpha]$ in $\Phi^k\nq$ is contained in $X_3 \setminus X_2$ if and only if it satisfies: 
\begin{itemize}
	\item [(3a)] $\alpha = \delta^k$ or $\alpha = \id$;
	\item [(3b)] $\blacklozenge\in\im\alpha_k$; and
	\item [(3c)] there exists $1 \le \ell \le n$ such that $\ell \neq k$ and $\alpha_\ell$ is not surjective.
\end{itemize}
Observe that if $[\delta^k;\aalpha]$ satisfies (3b) and (3c), then there is a unique cell $[\id;\bbeta]$ satisfying (3b) and (3c) such that $[\delta^k;\aalpha]$ is a $k$-th horizontal face of $[\id;\bbeta]$.
\emph{e.g.}
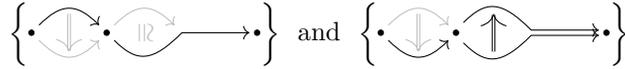
\[
\left\{
\begin{tikzpicture}[baseline = -3]
\filldraw
(0,0) circle [radius = 1pt]
(1,0) circle [radius = 1pt]
(3,0) circle [radius = 1pt];
\draw[->] (0.1,0.1) .. controls (0.4,0.4) and (0.6,0.4) .. (0.9,0.1);
\draw[->, gray!50!white] (0.1,-0.1) .. controls (0.4,-0.4) and (0.6,-0.4) .. (0.9,-0.1);
\draw[->, gray!50!white] (1.1,0.1) .. controls (1.4,0.4) and (1.6,0.4) .. (1.9,0.1);
\draw[->] (1.1,-0.1) .. controls (1.4,-0.4) and (1.6,-0.4) .. (2,0) -- (2.9,0);
\draw[->, double, gray!50!white] (0.5,0.25)--(0.5,-0.25);
\node[gray!50!white] at (1.5,0) {\rotatebox{270}{$\cong$}};
\end{tikzpicture}\right\}
\hspace{5pt} \text {and} \hspace{5pt}
\left\{
\begin{tikzpicture}[baseline = -3]
\filldraw
(0,0) circle [radius = 1pt]
(1,0) circle [radius = 1pt]
(3,0) circle [radius = 1pt];
\draw[->, gray!50!white] (0.1,0.1) .. controls (0.4,0.4) and (0.6,0.4) .. (0.9,0.1);
\draw[->] (0.1,-0.1) .. controls (0.4,-0.4) and (0.6,-0.4) .. (0.9,-0.1);
\draw[->, yshift = 1pt] (1.1,0.1) .. controls (1.4,0.4) and (1.6,0.4) .. (2,0) -- (2.9,0);
\draw[->, yshift = -1pt] (1.1,-0.1) .. controls (1.4,-0.4) and (1.6,-0.4) .. (2,0) -- (2.9,0);
\draw[->, double, gray!50!white] (0.5,0.25)--(0.5,-0.25);
\draw[->, double] (1.5,-0.25)--(1.5,0.25);
\end{tikzpicture}\right\}
\]
are $2$nd horizontal faces of
\[
\left\{
\begin{tikzpicture}[baseline = -3]
\filldraw
(0,0) circle [radius = 1pt]
(1,0) circle [radius = 1pt]
(2,0) circle [radius = 1pt]
(3,0) circle [radius = 1pt];
\draw[->] (0.1,0.1) .. controls (0.4,0.4) and (0.6,0.4) .. (0.9,0.1);
\draw[->, gray!50!white] (0.1,-0.1) .. controls (0.4,-0.4) and (0.6,-0.4) .. (0.9,-0.1);
\draw[->, gray!50!white] (1.1,0.1) .. controls (1.4,0.4) and (1.6,0.4) .. (1.9,0.1);
\draw[->] (1.1,-0.1) .. controls (1.4,-0.4) and (1.6,-0.4) .. (1.9,-0.1);
\draw[->, double, gray!50!white] (0.5,0.25)--(0.5,-0.25);
\node[gray!50!white] at (1.5,0) {\rotatebox{270}{$\cong$}};
\draw[->] (2.1,0)--(2.9,0);
\end{tikzpicture}\right\}
\hspace{5pt} \text {and} \hspace{5pt}
\left\{
\begin{tikzpicture}[baseline = -3]
\filldraw
(0,0) circle [radius = 1pt]
(1,0) circle [radius = 1pt]
(2,0) circle [radius = 1pt]
(3,0) circle [radius = 1pt];
\draw[->, gray!50!white] (0.1,0.1) .. controls (0.4,0.4) and (0.6,0.4) .. (0.9,0.1);
\draw[->] (0.1,-0.1) .. controls (0.4,-0.4) and (0.6,-0.4) .. (0.9,-0.1);
\draw[->] (1.1,0.1) .. controls (1.4,0.4) and (1.6,0.4) .. (1.9,0.1);
\draw[->] (1.1,-0.1) .. controls (1.4,-0.4) and (1.6,-0.4) .. (1.9,-0.1);
\draw[->, double, gray!50!white] (0.5,0.25)--(0.5,-0.25);
\draw[->, double] (1.5,-0.25)--(1.5,0.25);
\draw[->] (2.1,0)--(2.9,0);
\end{tikzpicture}\right\}
\]
respectively.
\begin{lemma}\label{step 3}
	The inclusion $X_2 \incl X_3$ is in $\celll(\H_h)$.
\end{lemma}
\begin{proof}
The discussion above shows that the set of non-degenerate cells in $X_3 \setminus X_2$ can be partitioned into subsets of the form
\[
\bigl\{\text {$[\id;\aalpha]$ and all of its $k$-th horizontal faces}\bigr\}
\]
where $[\id;\aalpha]$ is an $(n;\pp)$-cell satisfying (3b) and (3c).
We prove that $X_3$ may be obtained from $X_2$ by gluing such $[\id;\aalpha]$ along the horizontal horn $\horn_h^k\np$ in increasing order of $\dim\np$.
Note that this horn is inner since we are assuming $1 \le k \le n-1$.

Fix a non-degenerate $(n;\pp)$-cell $[\id;\aalpha]$ satisfying (3b) and (3c).
We must check that the following faces of $[\id;\aalpha]$ are contained either in $X_2$ or some $[\id;\bbeta]$ that satisfies (3b) and (3c) and has dimension strictly smaller than $\dim\np$:
\begin{itemize}
	\item $[\id;\aalpha] \cdot \delta_h^0$ is contained in $X_2$;
	\item $[\id;\aalpha] \cdot \delta_h^n$ is contained in $X_2$;
	\item $[\id;\aalpha] \cdot \delta_h^{\ell;\langle\gamma,\gamma'\rangle}$ is contained in $X_2$ for any $\ell \neq k$ and for any $\langle\gamma,\gamma'\rangle$;
	\item $[\id;\aalpha] \cdot \delta_v^{\ell;j}$ satisfies (3a-c) for any $\ell \neq k$ and for any $j \in [p_\ell]$; and
	\item $[\id;\aalpha] \cdot \delta_v^{k;j}$ is:
	\begin{itemize}
		\item contained in $\cell\nq$ if $\alpha_k^{-1}(\blacklozenge) = \{j\}$; and
		\item a (possibly trivial) degeneracy of some cell that satisfies (3a-c) otherwise.
	\end{itemize}
\end{itemize}
This completes the proof.
\end{proof}

Observe that the non-degenerate cells in $\Psi^k\nq \setminus X_3$ are precisely those $
[\delta^k;\aalpha] : \cell[n-1;\pp] \to \Phi^k\nq$ such that:
\begin{itemize}
	\item[(4a)] $\alpha_{\ell} = \id$ for $k \neq \ell \neq k+1$;
	\item[(4b)] $\alpha_{k+1}$ is surjective;
	\item[(4c)] $\blacklozenge \in \im \alpha_k$; and
	\item[(4d)] $\langle \alpha_k,\alpha_{k+1} \rangle : \Delta[p_k] \to J \times \Delta[q_{k+1}]$ is non-degenerate.
\end{itemize}
For our example $\Psi^2[3;1,0,0]$, these faces include
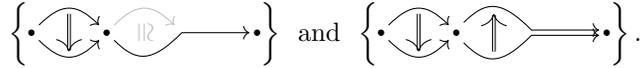
\[
\left\{\begin{tikzpicture}[baseline = -3]
\filldraw
(0,0) circle [radius = 1pt]
(1,0) circle [radius = 1pt]
(3,0) circle [radius = 1pt];
\draw[->] (0.1,0.1) .. controls (0.4,0.4) and (0.6,0.4) .. (0.9,0.1);
\draw[->] (0.1,-0.1) .. controls (0.4,-0.4) and (0.6,-0.4) .. (0.9,-0.1);
\draw[->, gray!50!white] (1.1,0.1) .. controls (1.4,0.4) and (1.6,0.4) .. (1.9,0.1);
\draw[->] (1.1,-0.1) .. controls (1.4,-0.4) and (1.6,-0.4) .. (2,0) -- (2.9,0);
\draw[->, double] (0.5,0.25)--(0.5,-0.25);
\node[gray!50!white] at (1.5,0) {\rotatebox{270}{$\cong$}};
\end{tikzpicture}\right\}
\hspace{5pt} \text {and} \hspace{5pt}
\left\{\begin{tikzpicture}[baseline = -3]
\filldraw
(0,0) circle [radius = 1pt]
(1,0) circle [radius = 1pt]
(3,0) circle [radius = 1pt];
\draw[->] (0.1,0.1) .. controls (0.4,0.4) and (0.6,0.4) .. (0.9,0.1);
\draw[->] (0.1,-0.1) .. controls (0.4,-0.4) and (0.6,-0.4) .. (0.9,-0.1);
\draw[->, yshift = 1pt] (1.1,0.1) .. controls (1.4,0.4) and (1.6,0.4) .. (2,0) -- (2.9,0);
\draw[->, yshift = -1pt] (1.1,-0.1) .. controls (1.4,-0.4) and (1.6,-0.4) .. (2,0) -- (2.9,0);
\draw[->, double] (0.5,0.25)--(0.5,-0.25);
\draw[->, double] (1.5,-0.25)--(1.5,0.25);
\end{tikzpicture}\right\}.
\]
In fact, there is a map
\[
\Phi^2[2;1,0] \to \Phi^3[2;1,0,0]
\]
which looks like
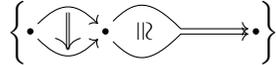
\[
\left\{\begin{tikzpicture}[baseline = -3]
\filldraw
(0,0) circle [radius = 1pt]
(1,0) circle [radius = 1pt]
(3,0) circle [radius = 1pt];
\draw[->] (0.1,0.1) .. controls (0.4,0.4) and (0.6,0.4) .. (0.9,0.1);
\draw[->] (0.1,-0.1) .. controls (0.4,-0.4) and (0.6,-0.4) .. (0.9,-0.1);
\draw[->, yshift = 1pt] (1.1,0.1) .. controls (1.4,0.4) and (1.6,0.4) .. (2,0) -- (2.9,0);
\draw[->, yshift = -1pt] (1.1,-0.1) .. controls (1.4,-0.4) and (1.6,-0.4) .. (2,0) -- (2.9,0);
\draw[->, double] (0.5,0.25)--(0.5,-0.25);
\node at (1.5,0) {\rotatebox{270}{$\cong$}};
\end{tikzpicture}\right\}
\]
and $\Psi^2[3;1,0,0] \setminus X_3$ is precisely the image of $\Phi^2[2;1,0] \setminus \Psi^2[2;1,0]$ under this map.

The above observation can be generalised to include all cases where $q_{k+1} = 0$. However it does not hold if $q_{k+1} \ge 1$, \emph{e.g.}~$\nq = [2;0,2]$ (with $k = 1$) in which case $\Phi^k\nq$ looks like
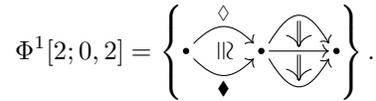
\[
\Phi^1[2;0,2] = \left\{
\begin{tikzpicture}[baseline = -3]
\filldraw
(0,0) circle [radius = 1pt]
(1,0) circle [radius = 1pt]
(2,0) circle [radius = 1pt];
\draw[->] (0.1,0.1) .. controls (0.4,0.4) and (0.6,0.4) .. (0.9,0.1);
\draw[->] (0.1,-0.1) .. controls (0.4,-0.4) and (0.6,-0.4) .. (0.9,-0.1);
\draw[->] (1.1,0) -- (1.9,0);
\draw[->] (1.1,0.1) .. controls (1.4,0.6) and (1.6,0.6) .. (1.9,0.1);
\draw[->] (1.1,-0.1) .. controls (1.4,-0.6) and (1.6,-0.6) .. (1.9,-0.1);
\draw[->, double] (1.5,0.4)--(1.5,0.05);
\draw[->, double] (1.5,-0.05)--(1.5,-0.4);
\node at (0.5,0) {\rotatebox{270}{$\cong$}};
\node[scale = 0.7] at (0.5, 0.5) {$\lozenge$};
\node[scale = 0.7] at (0.5, -0.5) {$\blacklozenge$};
\end{tikzpicture}\right\}.
\]\
In this case, given a non-degenerate cell $[\delta^k;\aalpha]$ satisfying (4a-d), let $i_\aalpha = \alpha_{k+1}\bigl(\min\bigl(\alpha_k^{-1}(\blacklozenge)\bigr)\bigr)$ and
\[
\begingroup
\renewcommand{\arraystretch}{1.5}
j_\aalpha = \left\{\begin{array}{ll}
\min\bigl(\alpha_{k+1}^{-1}(i_\aalpha)\bigr) & \text {if $i_\aalpha \ge 1$,}\\
\max\bigl(\alpha_{k+1}^{-1}(0)\bigr) & \text {if $i_\aalpha = 0$.}
\end{array}\right.
\endgroup
\]
Then the cell $[\delta^k;\aalpha]$ either satisfies
\begin{itemize}
	\item [(4e)] $\alpha_k(j_\aalpha) = \lozenge$
\end{itemize}
or there is a unique $(n-1;\mathbf{p'})$-cell $[\delta^k;\bbeta]$ in $\Psi^k\nq \setminus X_3$ satisfying (4a-e) such that $[\delta^k;\aalpha]$ is the $(k;j_\bbeta)$-th vertical hyperface of $[\delta^k;\bbeta]$.
\emph{e.g.}
\[
\left\{\begin{tikzpicture}[baseline = -3]
\filldraw
(0,0) circle [radius = 1pt]
(2,0) circle [radius = 1pt];
\draw[->, gray!50!white] (0.1,0.1) .. controls (0.4,0.4) and (0.6,0.4) .. (0.9,0.1);
\draw[->, yshift = 0.5pt] (0.1,-0.1) .. controls (0.4,-0.4) and (0.6,-0.4) .. (1,0) .. controls (1.4,0.6) and (1.6,0.6) .. (1.9,0.1);
\draw[->, yshift = -1pt] (0.1,-0.1) .. controls (0.4,-0.4) and (0.6,-0.4) .. (1,0) -- (1.9,0);
\draw[->, yshift = -2.5pt] (0.1,-0.1) .. controls (0.4,-0.4) and (0.6,-0.4) .. (1,0) .. controls (1.4,-0.6) and (1.6,-0.6) .. (1.9,-0.1);
\draw[->, double, yshift = -0.5pt] (1.5,0.4)--(1.5,0.05);
\draw[->, double, yshift = -1.5pt] (1.5,-0.05)--(1.5,-0.4);
\node[gray!50!white] at (0.5,0) {\rotatebox{270}{$\cong$}};
\end{tikzpicture}\right\}, \hspace{5pt}
\left\{\begin{tikzpicture}[baseline = -3]
\filldraw
(0,0) circle [radius = 1pt]
(2,0) circle [radius = 1pt];
\draw[->, yshift = 2.25pt] (0.1,0.1) .. controls (0.4,0.4) and (0.6,0.4) .. (1,0) .. controls (1.4,0.6) and (1.6,0.6) .. (1.9,0.1);
\draw[->, yshift = -0.75pt] (0.1,-0.1) .. controls (0.4,-0.4) and (0.6,-0.4) .. (1,0) -- (1.9,0);
\draw[->, yshift = -2.25pt] (0.1,-0.1) .. controls (0.4,-0.4) and (0.6,-0.4) .. (1,0) .. controls (1.4,-0.6) and (1.6,-0.6) .. (1.9,-0.1);
\draw[->, double, yshift = 1.5pt] (1.5,0.4)--(1.5,0.05);
\draw[->, double, yshift = -1.5pt] (1.5,-0.05)--(1.5,-0.4);
\end{tikzpicture}\right\},\hspace{5pt} \text {and} \hspace{5pt}
\left\{\begin{tikzpicture}[baseline = -3]
\filldraw
(0,0) circle [radius = 1pt]
(2,0) circle [radius = 1pt];
\draw[->, yshift = 2.25pt] (0.1,0.1) .. controls (0.4,0.4) and (0.6,0.4) .. (1,0) .. controls (1.4,0.6) and (1.6,0.6) .. (1.9,0.1);
\draw[->, yshift = 0.75pt] (0.1,0.1) .. controls (0.4,0.4) and (0.6,0.4) .. (1,0) -- (1.9,0);
\draw[->, yshift = -2.25pt] (0.1,-0.1) .. controls (0.4,-0.4) and (0.6,-0.4) .. (1,0) .. controls (1.4,-0.6) and (1.6,-0.6) .. (1.9,-0.1);
\draw[->, double, yshift = 1.5pt] (1.5,0.4)--(1.5,0.05);
\draw[->, double, yshift = -1.5pt] (1.5,-0.05)--(1.5,-0.4);
\end{tikzpicture}\right\}
\]
have $i_\aalpha = 0,1,$ and $2$ respectively, and they are moreover the appropriate vertical hyperfaces of
\[
\left\{\begin{tikzpicture}[baseline = -3]
\filldraw
(0,0) circle [radius = 1pt]
(2,0) circle [radius = 1pt];
\draw[->, yshift = 0.5pt] (0.1,-0.1) .. controls (0.4,-0.4) and (0.6,-0.4) .. (1,0) .. controls (1.4,0.6) and (1.6,0.6) .. (1.9,0.1);
\filldraw[white] (0.03,0.04) +(1,-.05)--+(0.95,0)--+(1,.05)--+(1.05,0)--cycle;
\draw[->] (0.1,0.1) .. controls (0.4,0.4) and (0.6,0.4) .. (1.05,0.025) .. controls (1.4,0.55) and (1.6,0.55) .. (1.9,0.05);
\draw[->, yshift = -1pt] (0.1,-0.1) .. controls (0.4,-0.4) and (0.6,-0.4) .. (1,0) -- (1.9,0);
\draw[->, yshift = -2.5pt] (0.1,-0.1) .. controls (0.4,-0.4) and (0.6,-0.4) .. (1,0) .. controls (1.4,-0.6) and (1.6,-0.6) .. (1.9,-0.1);
\draw[->, double] (0.5,-0.25)--(0.5,0.25);
\draw[->, double, yshift = -0.5pt] (1.5,0.4)--(1.5,0.05);
\draw[->, double, yshift = -1.5pt] (1.5,-0.05)--(1.5,-0.4);
\end{tikzpicture}\right\}
, \hspace{5pt}
\left\{\begin{tikzpicture}[baseline = -3]
\filldraw
(0,0) circle [radius = 1pt]
(2,0) circle [radius = 1pt];
\draw[->, yshift = 2.25pt] (0.1,0.1) .. controls (0.4,0.4) and (0.6,0.4) .. (1,0) .. controls (1.4,0.6) and (1.6,0.6) .. (1.9,0.1);
\draw[->, yshift = 0.75pt] (0.1,0.1) .. controls (0.4,0.4) and (0.6,0.4) .. (1,0) -- (1.9,0);
\draw[->, yshift = -0.75pt] (0.1,-0.1) .. controls (0.4,-0.4) and (0.6,-0.4) .. (1,0) -- (1.9,0);
\draw[->, yshift = -2.25pt] (0.1,-0.1) .. controls (0.4,-0.4) and (0.6,-0.4) .. (1,0) .. controls (1.4,-0.6) and (1.6,-0.6) .. (1.9,-0.1);
\draw[->, double] (0.5,0.25)--(0.5,-0.25);
\draw[->, double, yshift = 1.5pt] (1.5,0.4)--(1.5,0.05);
\draw[->, double, yshift = -1.5pt] (1.5,-0.05)--(1.5,-0.4);
\end{tikzpicture}\right\},\hspace{5pt} \text {and} \hspace{5pt}
\left\{\begin{tikzpicture}[baseline = -3]
\filldraw
(0,0) circle [radius = 1pt]
(2,0) circle [radius = 1pt];
\draw[->, yshift = 2.25pt] (0.1,0.1) .. controls (0.4,0.4) and (0.6,0.4) .. (1,0) .. controls (1.4,0.6) and (1.6,0.6) .. (1.9,0.1);
\draw[->, yshift = 0.75pt] (0.1,0.1) .. controls (0.4,0.4) and (0.6,0.4) .. (1,0) -- (1.9,0);
\draw[->, yshift = -0.75pt] (0.1,0.1) .. controls (0.4,0.4) and (0.6,0.4) .. (1,0) .. controls (1.4,-0.6) and (1.6,-0.6) .. (1.9,-0.1);
\draw[->, yshift = -2.25pt] (0.1,-0.1) .. controls (0.4,-0.4) and (0.6,-0.4) .. (1,0) .. controls (1.4,-0.6) and (1.6,-0.6) .. (1.9,-0.1);
\draw[->, double] (0.5,0.25)--(0.5,-0.25);
\draw[->, double, yshift = 1.5pt] (1.5,0.4)--(1.5,0.05);
\draw[->, double, yshift = -1.5pt] (1.5,-0.05)--(1.5,-0.4);
\end{tikzpicture}\right\}
\]
respectively.

The motivation behind the definitions of $i_\aalpha$ and $j_\aalpha$ is as follows.
Ideally, we would like to simply say ``the first 1-cell involving $\blacklozenge$ is preceded by an otherwise identical 1-cell that involves $\lozenge$'' in (4e) and use this extra $\lozenge$ to identify the interior/face pairs for the inner horns to be filled.
However, this horn is outer if (and only if) $i_{\aalpha} = 0$.
Thus we define $j_{\aalpha}$ differently in this case so that (4e) says ``the \emph{last} 1-cell of the form $\left\{\begin{tikzpicture}[baseline = -3]
\filldraw
(0,0) circle [radius = 1pt]
(2,0) circle [radius = 1pt];
\draw[->, gray!50!white] (0.1,0.1) .. controls (0.4,0.4) and (0.6,0.4) .. (0.9,0.1);
\draw[->, yshift = 1.5pt] (0.1,-0.1) .. controls (0.4,-0.4) and (0.6,-0.4) .. (1,0) .. controls (1.4,0.6) and (1.6,0.6) .. (1.9,0.1);
\draw[->, gray!50!white] (1.1,0) -- (1.9,0);
\draw[->, gray!50!white, yshift = -1.5pt] (1.1,-0.1) .. controls (1.4,-0.6) and (1.6,-0.6) .. (1.9,-0.1);
\draw[->, double, gray!50!white] (1.5,0.4)--(1.5,0.05);
\draw[->, double, gray!50!white] (1.5,-0.05)--(1.5,-0.4);
\node[gray!50!white] at (0.5,0) {\rotatebox{270}{$\cong$}};
\end{tikzpicture}\right\}$ is \emph{followed} by one of the form $\left\{\begin{tikzpicture}[baseline = -3]
\filldraw
(0,0) circle [radius = 1pt]
(2,0) circle [radius = 1pt];
\draw[->, gray!50!white] (0.1,-0.1) .. controls (0.4,-0.4) and (0.6,-0.4) .. (0.9,-0.1);
\draw[->, yshift = 1.5pt] (0.1,0.1) .. controls (0.4,0.4) and (0.6,0.4) .. (1,0) .. controls (1.4,0.6) and (1.6,0.6) .. (1.9,0.1);
\draw[->, gray!50!white] (1.1,0) -- (1.9,0);
\draw[->, gray!50!white, yshift = -1.5pt] (1.1,-0.1) .. controls (1.4,-0.6) and (1.6,-0.6) .. (1.9,-0.1);
\draw[->, double, gray!50!white] (1.5,0.4)--(1.5,0.05);
\draw[->, double, gray!50!white] (1.5,-0.05)--(1.5,-0.4);
\node[gray!50!white] at (0.5,0) {\rotatebox{270}{$\cong$}};
\end{tikzpicture}\right\}$'' instead.
\begin{lemma}\label{step 4}
    The inclusion $X_3 \incl \Phi^k\nq$ is in $\celll(\H_h)$.
    The inclusion $X_3 \incl \Psi^k\nq$ is:
    \begin{itemize}
        \item a pushout of $\Psi^\ell[n-1;\pp] \incl \Phi^\ell[n-1;\pp]$ for some $[n-1;\pp] \in \cell$ with $\dim[n-1;\pp] < \dim\nq$ if $q_{k+1} = 0$; and
        \item in $\celll(\H_v)$ if $q_{k+1} \ge 1$.
    \end{itemize}
\end{lemma}
\begin{proof}
For the inclusion $X_3 \incl \Phi^k\nq$, we can simply continue gluing the remaining cells $[\id;\aalpha] : \cell\np \to \Phi^k\nq$ satisfying (3b) (but not (3c)) along $\horn_h^k\np$ in increasing order of $\dim\np$.

Consider the inclusion $X_3 \incl \Psi^k\nq$.
For the case $q_{k+1} = 0$, recall that the functor $\wreath \to \hatdelta$ is a (split) cartesian fibration.
Thus there is a cartesian lift of the map $\delta^k :\Delta[n-1]\to\Delta[n]$ at the object
\[
\bigl(\Delta[q_1], \dots, \Delta[q_{k-1}], J, \Delta[q_{k+1}], \dots, \Delta[q_n]\bigr) \in \hatdelta^n \simeq \bigl(\wreath\bigr)_{\Delta[n]}.
\]
Applying the box product functor $\square$ to this lift yields a map
\[
[\delta^k;\id,\dots,!,\dots,\id] : \Phi^k[n-1;\pp] \to \Phi^k\nq
\]
where $\pp = (q_1,\dots,q_{k-1},0,q_{k+2},\dots,q_n)$.
This map factors through $\Psi^k\nq$ because its image is generated by the cells of the form $[\delta^k;\aalpha]$.
Moreover, one can check by comparing (4a-d) and (\textasteriskcentered\textasteriskcentered) (the latter of which appeared in the second paragraph after the proof of \cref{vertical equivalence extensions}) that this map fits into the following gluing square:
\[
\begin{tikzcd}[row sep = large]
\Psi^k[n-1;\pp] \arrow [r, hook, "\subset"] \arrow [d] \glue & \Phi^k[n-1;\pp] \arrow [d,  "{[\delta^k;\id,\dots,!,\dots,\id]}"]\\
X_3 \arrow [r, hook, "\subset"] & \Psi^k\nq
\end{tikzcd}
\]
This completes the proof for the first case.

Next consider the case $q_{k+1} \ge 2$.
The discussion before \cref{step 4} shows that the set of non-degenerate cells in $\Psi^k\nq \setminus X_3$ can be partitioned into subsets of the form
\[
\bigl\{[\delta^k;\aalpha], [\delta^k;\aalpha] \cdot \delta_v^{k;j_\aalpha}\bigr\}
\]
where $[\delta^k;\aalpha]$ is an $(n-1;\pp)$-cell satisfying (4a-e).
We prove that $\Psi^k\nq$ may be obtained from $X_3$ by gluing such $[\delta^k\aalpha]$ along the vertical horn $\horn_v^{k;j_\aalpha}[n-1;\pp]$ in lexicographically increasing order of $\dim[n-1;\pp]$ and $|\alpha_k^{-1}(\blacklozenge)|$.
Note that this horn is inner by the definition of $j_\aalpha$.

Fix a non-degenerate $(n-1;\pp)$-cell $[\delta^k;\aalpha]$ satisfying (4a-e).
We must check that the appropriate faces of $[\delta^k;\aalpha]$ are contained either in $X_3$ or in some $(n-1;\mathbf{p'})$-cell $[\delta^k;\bbeta]$ satisfying (4a-e) such that:
\begin{itemize}
	\item $\dim[n-1;\mathbf{p'}] < \dim[n-1;\pp]$; or
	\item $\dim[n-1;\mathbf{p'}] \le \dim[n-1;\pp]$ and $|\beta_k^{-1}(\blacklozenge)| < |\alpha_k^{-1}(\blacklozenge)|$.
\end{itemize}
If $i_\aalpha \ge 1$:
\begin{itemize}
	\item any horizontal hyperface of $[\delta^k;\aalpha]$ is contained in $X_2$;
	\item $[\delta^k;\aalpha] \cdot \delta_v^{\ell;j}$ is contained in $X_3$ for any $\ell \neq k$ and for any $j \in [p_\ell]$;
	\item $[\delta^k;\aalpha] \cdot \delta_v^{k;j}$, where $j_\aalpha \neq j \neq j_\aalpha+1$, is:
	\begin{itemize}
		\item contained in $X_3$ if $\alpha_{k+1} \cdot \delta^j$ is not surjective; and
		\item a (possibly trivial) degeneracy of some cell that satisfies (4a-e) otherwise; and
	\end{itemize}
	\item $[\delta^k;\aalpha] \cdot \delta_v^{k;j_\aalpha+1}$ is:
	\begin{itemize}
		\item contained in $\cell\nq$ if $\alpha_k(j) = \lozenge$ for all $j \neq j_{\aalpha}+1$; and
		\item a (possibly trivial) degeneracy of some $(n-1;\mathbf{p'})$-cell $[\delta^k;\bbeta]$ that satisfies (4a-d) otherwise.
	\end{itemize}
\end{itemize}
Note in the last clause, the cell $[\delta^k;\bbeta]$ may not satisfy (4e).
However, at least we know $\dim[n-1;\mathbf{p'}] < \dim[n-1;\pp]$ and $|\beta_k^{-1}(\blacklozenge)|<|\alpha_k^{-1}(\blacklozenge)|$.
Hence if $[\delta^k;\boldsymbol{\gamma}]$ is an $(n-1;\mathbf{p''})$-cell satisfying (4a-e) such that $
[\delta^k;\bbeta] = [\delta^k;\boldsymbol{\gamma}] \cdot \delta_v^{k;j_{\boldsymbol{\gamma}}}$ then
\[
\dim[n-1;\mathbf{p''}] = \dim[n-1;\mathbf{p'}] + 1 \le \dim[n-1;\pp]
\]
and
\[
|\gamma_k^{-1}(\blacklozenge)| = |\beta_k^{-1}(\blacklozenge)|<|\alpha_k^{-1}(\blacklozenge)|.
\]
A similar analysis can be done for the case $i_\aalpha = 0$ too, and this completes the proof.
\end{proof}

\section{Alternative horizontal horns}\label{alternative horns}
We now consider a slightly different set of horn inclusions.
\begin{definition}
	Given $\nq \in \cell$, $1 \le k \le n-1$ and a $(q_k,q_{k+1})$-shuffle $\langle \alpha, \alpha' \rangle$, let $\horn_h^{k;\langle\alpha,\alpha'\rangle}\nq \subset \cell\nq$ denote the cellular subset generated by all hyperfaces of $\cell\nq$ except for $\delta_h^{k;\langle\alpha,\alpha'\rangle}$.
\end{definition}

We write $\H'_h$ for the set of all such \emph{alternative inner horizontal horn inclusions} $\horn_h^{k;\langle\alpha,\alpha'\rangle}\nq \incl \cell\nq$.
We prove $\H_h \subset \celll(\H'_h \cup \H_v)$ and that $\H'_h$ is contained in the trivial cofibrations.

\subsection{Oury's horn inclusions can be obtained from the alternative ones}
The purpose of this subsection is to prove the following lemma.
\begin{lemma}\label{Oury horizontal horns from alternative ones}
	Every map in $\H_h$ is contained in $\celll(\H'_h \cup \H_v)$.
\end{lemma}
Similarly to the proof of \cref{horn inclusions are trivial cofibrations}, we must consider a wider class of horn inclusions.
\begin{definition}\label{horn definition}
	Given a set $S$ of hyperfaces of $\cell\nq$, let $\horn^S\nq \subset \cell\nq$ denote the cellular subset generated by all hyperfaces except for those in $S$.
\end{definition}

\begin{proposition}\label{horn is horn}
    For any set $S$ of inner hyperfaces of $\cell\nq$, the cellular subset $\horn^S\nq \subset \cell\nq$ is equal to $\Upsilon^T\nq$ where $T$ is the set of all inner hyperfaces of $\cell\nq$ that are not in $S$. 
\end{proposition}
\begin{proof}
    Compare \cref{Upsilon,horn definition}.
\end{proof}

Recall that if $S$ is a set of faces of $\cell\nq$ and $1 \le k \le n-1$ then we write
\[
\shuffle_S(q_k,q_{k+1}) = \left\{\langle\alpha,\alpha'\rangle \in (q_k,q_{k+1}) : \delta_h^{k;\langle\alpha,\alpha'\rangle} \in S\right\}.
\]
\begin{lemma}\label{generalised alternative horns}
	The inclusion $\horn^S\nq \incl \cell\nq$ is contained in:
	\begin{itemize}
		\item [(i)] $\celll(\H_v)$ if $S$ is a non-empty set of $k$-th vertical hyperfaces for some $1 \le k \le n$; and
		\item [(ii)] $\celll(\H'_h \cup \H_v)$ if $S$ is a non-empty set of $k$-th horizontal hyperfaces for some $1 \le k \le n-1$ and $\shuffle_S(q_k,q_{k+1})$ is upward closed.
	\end{itemize}
\end{lemma}
Note that \cref{Oury horizontal horns from alternative ones} follows from \cref{generalised alternative horns}(ii) by setting $S$ to be the set of all $k$-th horizontal hyperfaces of $\cell\nq$.
\begin{proof}
We will prove (i) by induction on $|S|$.
By assumption, we can write $S$ as
\[
S = \bigl\{\delta_v^{k;i} : i \in I_S\bigr\}
\]
for some $1 \le k \le n$ and $\varnothing \neq I_S \subset \{1, \dots, q_k-1\}$.
If $I_S = \{i\}$ is a singleton, then $\horn^S\nq = \horn_v^{k;i}\nq$ and hence the result follows trivially.
So assume $|S| \ge 2$.
Choose $i \in I_S$ and let $S' = \bigl\{\delta_v^{k;j} : j \in I_S \setminus \{i\}\bigr\}$.
Then $\horn^{S'}\nq \incl \cell\nq$ is in $\celll(\H_v)$ by the inductive hypothesis.
Therefore it suffices to prove that the upper horizontal map in the gluing square
\[
\begin{tikzcd}[row sep = large]
X \arrow [r, hook, "\subset"] \arrow [d] \glue & \cell[n;\pp] \arrow [d, "\delta_v^{k;i}"] \\
\horn^S\nq \arrow [r, hook, "\subset"] & \horn^{S'}\nq
\end{tikzcd}
\]
belongs to $\celll(\H_v)$, where $\pp = (q_1,\dots,q_k-1,\dots,q_n)$.
Indeed, one can check that $X = \horn^T[n;\pp]$ where
\[
T = \bigl\{\delta_v^{k;j} : j \in I_S,~j<i\bigr\} \cup \bigl\{\delta_v^{k;j-1} : j \in I_S,~j>i\bigr\}.
\]
Since $|T| = |S|-1$, the desired inclusion is in $\celll(\H_v)$ by the inductive hypothesis.

Now we prove (ii) by induction on $|S|$.
If $S = \bigl\{\delta_h^{k;\langle\alpha,\alpha'\rangle}\bigr\}$ is a singleton then $\horn^S\nq = \horn_h^{k;\langle\alpha,\alpha'\rangle}\nq$ and hence the result follows trivially.
So assume $|S| \ge 2$.
Choose a minimal element $\langle \alpha, \alpha' \rangle \in \shuffle_S(q_k,q_{k+1})$ and let $S' = S \setminus \bigl\{\delta_h^{k;\langle\alpha,\alpha'\rangle}\bigr\}$.
Then by the inductive hypothesis, $\horn^{S'}\nq \incl \cell\nq$ is in $\celll(\H'_h \cup \H_v)$.
Thus it suffices to prove that $\horn^S\nq \incl \horn^{S'}\nq$ too is in $\celll(\H'_h \cup \H_v)$.
Indeed, it follows from \cref{claim 0,claim 1,claim 2,claim 3,claim 4} in the proof of \cref{part 3} that we have a gluing square
\[
\begin{tikzcd}[row sep = large]
\horn^T[n-1;\pp]\arrow [r, hook, "\subset"] \arrow [d] \glue & \cell[n-1;\pp] \arrow [d, "\delta_h^{k;\langle\alpha,\alpha'\rangle}"] \\
\horn^S\nq \arrow [r, hook, "\subset"] & \horn^{S'}\nq
\end{tikzcd}
\]
where $\pp = (q_1,\dots,q_k+q_{k+1},\dots,q_n)$, and $T = \{\delta_v^{k;i} : i \in \ulcorner\langle\alpha,\alpha'\rangle\}$.
(In fact, this square is essentially the first square that appears in the proof of \cref{part 3}.)
Note that $\ulcorner\langle\alpha,\alpha'\rangle = \varnothing$ if and only if $\langle\alpha,\alpha'\rangle$ is the maximum $(q_k,q_{k+1})$-shuffle, but the latter is impossible since $|S| \ge 2$ and $\langle\alpha,\alpha'\rangle$ is minimal in $S$.
Hence $\horn^T[n-1;\pp] \incl \cell[n-1;\pp]$ is in $\celll(\H_v)$ by (i).
\end{proof}

\subsection{Alternative horn inclusions are trivial cofibrations}
The purpose of this subsection is to prove the following lemma.
\begin{lemma}\label{alternative horns are trivial cofibrations}
	Every map in $\H'_h$ is a trivial cofibration.
\end{lemma}
Once again, we consider a wider class of horn inclusions.
Suppose we have fixed $\nq \in \cell$, $1 \le k \le n-1$ and a $(q_k,q_{k+1})$-shuffle $\langle\zeta,\zeta'\rangle$.
(Note that the inequality $1 \le n-1$ in particular implies that $\nq$ is poly-vertebral.)
Let
\[
I = \bigl\{\langle\alpha,\alpha'\rangle \in \shuffle(q_k,q_{k+1}) : \langle \alpha,\alpha' \rangle \le \langle\zeta,\zeta'\rangle\bigr\}.
\]
\begin{lemma}\label{upward closed}
	If $U$ is a set of the form
	\[
	U = \bigl\{\delta_h^{k;\langle\alpha,\alpha'\rangle} : \langle \alpha,\alpha' \rangle \in I_U\bigr\}
	\]
	for some non-empty, upward closed subset $I_U \subset I$, then the inclusion $\horn^U\nq \incl \cell\nq$ is a trivial cofibration.
\end{lemma}
Here $I_U$ is upward closed in $I$ but not necessarily in $\shuffle(q_k,q_{k+1})$.
Thus in general $\shuffle(q_k,q_{k+1}) \setminus I_U$ is not downward closed, and this is why \cref{upward closed} does not follow directly from \cref{horn is horn,part 3}.

Note that since $\langle\zeta,\zeta'\rangle$ is the maximum element in $I$, any non-empty, upward closed $I_U \subset I$ will always have $\langle\zeta,\zeta'\rangle \in I_U$.
Also observe that \cref{alternative horns are trivial cofibrations} follows from \cref{upward closed} by setting $U = \bigl\{\delta_h^{k;\langle\zeta,\zeta'\rangle}\bigr\}$.
\begin{proof}
We prove \cref{upward closed} by induction on $|I \setminus I_U|$ (so we start with the case $I_U = I$ and progressively make $I_U$ smaller).
For the base case, observe that
\[
\bigl\{\langle\alpha,\alpha'\rangle \in \shuffle(q_k,q_{k+1}) : \langle\alpha,\alpha'\rangle \nleq \langle\zeta,\zeta'\rangle\bigr\}
\]
is an upward closed, proper subset of $\shuffle(q_k,q_{k+1})$.
Thus, when $I_U = I$ the inclusion $\horn^U\nq \incl \cell\nq$ is a trivial cofibration by \cref{horn is horn} and the dual of \cref{part 3}.

For the inductive step, assume $I_U \neq I$.
Choose a maximal element $\langle \alpha,\alpha' \rangle \in I \setminus I_U$ and let $I_{U'} = I_U \cup \{\langle\alpha,\alpha'\rangle\}$.
Then $\horn^{U'}\nq \incl \cell\nq$ is a trivial cofibration by the inductive hypothesis, and hence it suffices to show the upper horizontal map in
\[
\begin{tikzcd}[row sep = large]
X
\arrow [r, hook, "\subset"]
\arrow [d]
\glue &
\cell[n-1;\pp]
\arrow [d, "\delta_h^{k;\langle\alpha,\alpha'\rangle}"]\\
\horn^{U'}\nq
\arrow [r, hook, "\subset"] &
\horn^U\nq
\end{tikzcd}
\]
(where $\pp = (q_1, \dots, q_k+q_{k+1},\dots,q_n)$) is a trivial cofibration.
We again use \cref{part 3}.
More precisely, we claim that $X$ has the form $X = \Upsilon^T[n-1;\pp]$ for
\[
T = T_1 \cup T_2 \cup T_3 \cup T_4 \cup T'_4 \cup T_5
\]
where
\[
\begin{split}
T_1 &= \bigl\{\delta_h^{\ell;\langle \gamma, \gamma' \rangle} [n-1;\pp] : 1 \le \ell \le n-1,~\langle\gamma,\gamma'\rangle\in \shuffle(p_\ell,p_{\ell+1})\bigr\},\\
T_2 &= \bigl\{\delta_v^{k;j}[n-1;\pp]: j \in \lrcorner\langle\alpha,\alpha'\rangle\bigr\},\\
T_3 &= \bigl\{\delta_v^{\ell;j}[n-1;\pp]: \ell \neq k,~1 \le j \le q_\ell-1\bigr\},\\
T_4 &= \bigl\{\delta_v^{k;j}[n-1;\pp] : (\exists i \in [q_k])\left[1 \le i \le q_k-1,~\alpha^{-1}(i)=\{j\}\right]\bigr\},\\
T'_4 &= \bigl\{\delta_v^{k;j}[n-1;\pp] : (\exists i \in [q_{k+1}])\left[1 \le i \le q_{k+1}-1,~(\alpha')^{-1}(i)=\{j\}\right]\bigr\},\\
T_5 &= \bigl\{\delta_v^{k;j}[n-1;\pp] : j \in \ulcorner\langle\alpha,\alpha'\rangle,~\alpha(j) = \zeta(j)\bigr\}.
\end{split}
\]
Aside from $T_5$, these sets are essentially special cases of the sets with the same names in the proof of \cref{part 3}.
More precisely, we have set $S'$ to be the set of inner hyperfaces of $\cell\nq$ that are not in $U'$, then merged $T_1$ and $T'_1$ into a single set and similarly for $T_3$, and unwound the conditions involving $S'$.
Thus for much of the proof that $X = \Upsilon^T[n-1;\pp]$ holds, we can reuse \cref{claim 0,claim 1,claim 2,claim 3,claim 4} from the proof of \cref{part 3}.

The following claim relates the hyperfaces $\delta_h^{k;\langle\beta,\beta'\rangle}$ of $\cell\nq$ with $\langle\beta,\beta'\rangle \notin I$ to the elements of $T_5$.
Note that if $1 \le j \le p_k-1$ and $j \notin \ulcorner\langle\alpha,\alpha'\rangle$ then \cref{points} implies that $\delta_v^{k;j}[n-1;\pp]$ is contained in $T_2$, $T_4$ or $T'_4$.

\begin{claim}\label{claim 5}\hfill
	\begin{itemize}
		\item [(i)] For any $\langle\beta,\beta'\rangle \in \shuffle(q_k,q_{k+1})$ with $\langle\beta,\beta'\rangle \nleq \langle\zeta,\zeta'\rangle$, the pullback of $\delta_h^{k;\langle\beta,\beta'\rangle}$ along $\delta_h^{k;\langle\alpha,\alpha'\rangle}$ is contained in $\delta_v^{k;j}[n-1;\pp]$ for some $1 \le j \le p_k-1$ such that either $j \notin \ulcorner\langle\alpha,\alpha'\rangle$ or $\alpha(j) = \zeta(j)$.
		\item [(ii)] For any $j \in \ulcorner\langle\alpha,\alpha'\rangle$ with $\alpha(j) = \zeta(j)$, the hyperface $\delta_v^{k;j}[n-1;\pp]$ is the pullback of $\delta_h^{k;\langle\beta,\beta'\rangle}$ along $\delta_h^{k;\langle\alpha,\alpha'\rangle}$ for some $\langle\beta,\beta'\rangle \in \shuffle(q_k,q_{k+1})$ with $\langle\beta,\beta'\rangle \nleq \langle\zeta,\zeta'\rangle$.
	\end{itemize}
\end{claim}

\begin{proof}
	For (i), fix $\langle\beta,\beta'\rangle \in \shuffle(q_k,q_{k+1})$ with $\langle\beta,\beta'\rangle \nleq \langle\zeta,\zeta'\rangle$.
	Note that if $\beta(j) \neq \alpha(j)$ for some $j \in [q_k+q_{k+1}]$, then the pullback of $\delta_h^{k;\langle\beta,\beta'\rangle}$ along $\delta_h^{k;\langle\alpha,\alpha'\rangle}$ is contained in $\delta_v^{k;j}[n-1;\pp]$.
	Thus it suffices to prove that there exists some $j \in [q_k+q_{k+1}]$ such that $\beta(j) \neq \alpha(j)$ and either $j \notin \ulcorner\langle\alpha,\alpha'\rangle$ or $\alpha(j) = \zeta(j)$.
	
	Suppose otherwise.
	Then in particular $\beta(j) = \alpha(j)$ for all $j \in [q_k+q_{k+1}] \setminus \ulcorner\langle\alpha,\alpha'\rangle$.
	Since $\ulcorner\langle\alpha,\alpha'\rangle$ contains no two consecutive integers, it follows that $\beta(j) \neq \alpha(j)$ implies $\beta(j) = \alpha(j)+1$ for any $j \in [q_k+q_{k+1}]$.
	Now for each $j\in [q_k+q_{k+1}]$:
	\begin{itemize}
		\item if $\beta(j) = \alpha(j)$ then $\beta(j) \le \zeta(j)$ because $\langle\alpha,\alpha'\rangle < \langle\zeta,\zeta'\rangle$; and
		\item if $\beta(j) \neq \alpha(j)$, then $\alpha(j) < \zeta(j)$ by assumption and hence $\beta(j) = \alpha(j)+1 \le \zeta(j)$.
	\end{itemize}
	Therefore we have $\langle\beta,\beta'\rangle \le \langle\zeta,\zeta'\rangle$, which is the desired contradiction.
	
	To prove (ii), let $j \in \ulcorner\langle\alpha,\alpha'\rangle$ and suppose $\alpha(j) = \zeta(j)$.
	Then $\langle \beta,\beta'\rangle \nleq \langle\zeta,\zeta'\rangle$ for the immediate successor $\langle \beta,\beta'\rangle$ of $\langle\alpha,\alpha'\rangle$ corresponding to $j$, and $\delta_v^{k;j}[n-1;\pp]$ is the pullback of $\delta_h^{k;\langle\beta,\beta'\rangle}$ along $\delta_h^{k;\langle\alpha,\alpha'\rangle}$.
\end{proof}
We can deduce $X = \Upsilon^T[n-1;\pp]$ from \cref{claim 0,claim 1,claim 2,claim 3,claim 4,claim 5}, and it remains to check that $T$ is admissible, \emph{i.e.}~$T$ satisfies \cref{admissible 2}(i-iii).
Since it contains all of the inner horizontal hyperfaces of $\cell[n-1;\pp]$, $T$ clearly satisfies (ii) and (iii).
For (i), observe that $\langle\alpha,\alpha'\rangle < \langle\zeta,\zeta'\rangle$ implies there exists an immediate successor $\langle\beta,\beta'\rangle$ of $\langle\alpha,\alpha'\rangle$ such that $\langle\beta,\beta'\rangle \le \langle\zeta,\zeta'\rangle$.
If $j \in \ulcorner\langle\alpha,\alpha'\rangle$ is the element corresponding to $\langle\beta,\beta'\rangle$, then $\alpha(j) = \beta(j)-1 < \zeta(j)$ and so $T$ does not contain $\delta_v^{k;j}[n-1;\pp]$.
Therefore $T$ is not the set of all inner hyperfaces of $\cell[n-1;\pp]$.
\end{proof}

\section{Most horizontal equivalence extensions are redundant}\label{horizontal equivalence section}
The aim of this very short section is to prove the following lemma.
\begin{lemma}\label{horizontal equivalence extensions}
	For any $[0] \neq \nq \in \cell$, the horizontal equivalence extension
	\[
	(\cell[0] \overset{e}{\hookrightarrow} J) \hat \times (\partial\cell\nq \incl \cell\nq)
	\]
	is contained in $\celll(\H_h)$.
\end{lemma}
\begin{proof}
Fix $[0] \neq \nq \in \cell$ and consider $e \hat \times (\partial\cell\nq \incl \cell\nq)$, whose domain we denote by $X$.
Let $Y \subset J \times \cell\nq$ be the cellular subset generated by $X$ and all cells that do not contain the vertex $(\blacklozenge,n)$.
Then for any non-degenerate $\phi : \cell\mp \to Y$ that does not factor through $X$, there is unique $1 \le  k_\phi \le m$ such that $\pi_1 \circ \phi(k_\phi-1) = \blacklozenge$ and $\pi_1 \circ \phi(k) = \lozenge$ for all $k_\phi \le k \le m$, where $\pi_1 : J \times \cell\nq \to J$ is the projection.
Observe that for any non-degenerate $\phi$ in $Y \setminus X$, either $\phi$ satisfies
\begin{itemize}
	\item[($\dagger$)] $\pi_2 \circ \phi(k_\phi) = \pi_2 \circ \phi(k_\phi-1)$
\end{itemize}
or there is a unique non-degenerate cell $\psi$ in $Y \setminus X$ satisfying ($\dagger$) such that $\phi$ is a (unique) $k_\psi$-th horizontal hyperface of $\psi$.
Therefore the non-degenerate cells in $Y \setminus X$ can be partitioned into pairs of the form
\[
\bigl\{\phi, \phi \cdot \delta_h^{k_\alpha;\langle!,\id\rangle}\bigr\}
\]
where $\phi$ is an $(m;\pp)$-cell satisfying ($\dagger$) (which necessarily has $p_{k_\phi} = 0$).
We prove that $Y$ may be obtained from $X$ by gluing such $\phi$ along the horn $\horn_h^{k_\phi}\mp$ in lexicographically increasing order of $\dim\mp$ and $\bigl|(\pi_1\circ\phi)^{-1}(\blacklozenge)\bigr|$.
(Here $|-|$ counts the number of objects.)

Fix an $(m;\pp)$-cell $\phi$ in $Y \setminus X$ satisfying ($\dagger$).
We must check that all hyperfaces of $\phi$ except for the (unique) $k_\alpha$-th horizontal one are contained either in $X$ or in some $(m';\mathbf{p'})$-cell $\psi$ satisfying ($\dagger$) such that either:
\begin{itemize}
	\item $\dim[m';\mathbf{p'}] < \dim\mp$; or
	\item $\dim[m';\mathbf{p'}] = \dim\mp$ and $\bigl|(\pi_1\circ\psi)^{-1}(\blacklozenge)\bigr| < \bigl|(\pi_1\circ\phi)^{-1}(\blacklozenge)\bigr|$.
\end{itemize}
Indeed:
\begin{itemize}
	\item $\phi \cdot \delta_h^{k_\phi-1;\langle\id,!\rangle}$ may or may not satisfy ($\dagger$), but we know
	\[
	\bigl|\bigl(\pi_1\circ(\phi\cdot\delta^{k_\phi-1})\bigr)^{-1}(\blacklozenge)\bigr| = \bigl|(\pi_1\circ\phi)^{-1}(\blacklozenge)\bigr|-1;
	\]
	\item for any $0 \le k < k_\phi-1$ with
	\[
	\bigl|(\pi_2\circ\phi)^{-1}(\pi_2\circ\phi(k))\bigr| \ge 2,
	\]
	any $k$-th horizontal hyperface of $\phi$ is a (possibly trivial) degeneracy of some cell satisfying ($\dagger$) with dimension strictly lower than $\dim\mp$; and
	\item any other hyperface of $\phi$ (excluding the $k_\alpha$-th horizontal one) is contained in $X$.
\end{itemize}
Moreover $\phi(k_\phi-1) = (\blacklozenge, \pi_2\circ\phi(k_\phi))$ and hence $\pi_2\circ\phi(k_\phi) \neq n$.
This implies $k_\phi \neq m$ and it follows that the horn $\horn_h^{k_\phi}\mp$ is inner.
Thus the inclusion $X \incl Y$ is in $\celll(\H_h)$.

Now consider the remaining non-degenerate cells $\phi : \cell\mp \to J \times \cell\nq$ that are not in $Y$.
Let $k_\phi \in [m]$ be the smallest such that $\pi_2\circ\phi(k_\phi) = n$.
Note that $\nq \neq [0]$ implies $k_\phi \neq 0$.
Observe that for any non-degenerate cell $\phi$ in $(J \times \cell\nq)\setminus Y$, either $\phi$ satisfies
\begin{itemize}
	\item[($\ddagger$)] $\pi_1\circ\phi(k_\phi) = \lozenge$
\end{itemize}
or there is a unique non-degenerate cell $\psi$ in $(J \times \cell\nq)\setminus Y$ satisfying ($\ddagger$) such that $\phi$ is a (unique) $k_\psi$-th horizontal hyperface of $\psi$.
Therefore the non-degenerate cells in $(J \times \cell\nq)\setminus Y$ can be partitioned into pairs of the form
\[
\bigl\{\phi,\phi \cdot \delta_h^{k_\phi;\langle\id,!\rangle} \bigr\}
\]
where $\phi$ is an $(m;\pp)$-cell satisfying ($\ddagger$) (which necessarily has $p_{k_\phi+1} = 0$).
We prove that $J \times \cell\nq$ may be obtained from $Y$ by gluing such $\phi$ along the horn $\horn_h^{k_\phi}\mp$ in increasing order of $\dim\mp$.

Fix an $(m;\pp)$-cell $\phi$ in $(J \times \cell\nq)\setminus Y$ satisfying ($\ddagger$).
We must check that all hyperfaces of $\phi$ except for the (unique) $k_\alpha$-th one are contained either in $Y$ or in some cell that satisfies ($\ddagger$) and has dimension strictly smaller than $\dim\mp$.
Indeed:
\begin{itemize}
	\item the unique $(k_\phi+1)$-th horizontal hyperface of $\phi$ (which may be inner or outer depending on whether $k_\phi+1 = m$) is:
	\begin{itemize}
		\item a degeneracy of some non-degenerate cell in $(J \times \cell\nq)\setminus Y$ satisfying ($\ddagger$) of dimension $\dim\mp-2$ if $k_\phi+3 \le m$ (in which case we necessarily have $\phi(k_\phi+3) = (\blacklozenge,n)$); and
		\item contained in $Y$ otherwise;
	\end{itemize}
	\item for any $k_\phi+1 < k \le m$, the unique $k$-th horizontal hyperface of $\phi$ is a (possibly trivial) degeneracy of some cell satisfying ($\ddagger$) of dimension strictly lower than $\dim\mp$;
	\item for any $0 \le k < k_\phi$ with
	\[
	\bigl|(\pi_2\circ\phi)^{-1}(\pi_2\circ\phi(k))\bigr| \ge 2,
	\]
	any $k$-th horizontal hyperface of $\phi$ is a (possibly trivial) degeneracy of some cell satisfying ($\ddagger$) of dimension strictly lower than $\dim\mp$; and
	\item any other hyperface of $\phi$ (excluding the $k_\alpha$-th horizontal one) is contained in $Y$.
\end{itemize}
Moreover, the horn $\horn_h^{k_\alpha}\mp$ is inner since ($\ddagger$) implies $k_\phi \neq m$.
This completes the proof.
\end{proof}

\section{Characterisation of fibrations into 2-quasi-categories}\label{last section}
Recall the sets $\H_h,\H_v,\E_v$ defined in \cref{O-anodyne}, the set $\J_A$ defined in \cref{model structure} and the set $\H'_h$ defined in \cref{alternative horns}.
By combining \cref{Ara characterisation} and all of the results we have proved, we obtain the following theorem.
\begin{theorem}\label{main theorem}
	Let $f : X \to Y$ be a map in $\hattheta$ and suppose that $Y$ is a 2-quasi-category.
	Then the following are equivalent:
	\begin{itemize}
		\item [(i)] $f$ is a fibration with respect to Ara's model structure;
		\item [(ii)] $f$ has the right lifting property with respect to all maps in $\J_A$;
		\item [(iii)] $f$ has the right lifting property with respect to all maps in $\H_h \cup \H_v \cup \E_v \cup \{e\}$; and
		\item [(iv)] $f$ has the right lifting property with respect to all maps in $\H'_h \cup \H_v \cup \E_v \cup \{e\}$.
	\end{itemize}
\end{theorem}
\begin{proof}
	(i) $\Leftrightarrow$ (ii):
	This equivalence is part of \cref{Ara characterisation}.
	
	(i) $\Rightarrow$ (iii):
	The elements of $\H_h$ and $\H_v$ are trivial cofibrations by \cref{horn inclusions are trivial cofibrations}, and similarly for $\E_v$ by \cref{vertical equivalence extensions}.
	The horizontal equivalence extension $e$ is also a trivial cofibration since $e \in \E_h \subset \J_A$.
	
	(iii) $\Rightarrow$ (ii):
	We have the containment $\J_A \subset \celll(\J_O) = \celll(\H_h \cup \H_v \cup \E_h \cup \E_v)$ by \cref{A-anodyne implies O-anodyne}.
	But $\E_h \subset \{e\} \cup \celll(\H_h)$ holds by \cref{horizontal equivalence extensions} and so $\celll(\J_O) = \celll(\H_h \cup \H_v \cup \E_v \cup \{e\})$.
	
	(i) $\Rightarrow$ (iv):
	The elements of $\H'_h$ are trivial cofibrations by \cref{alternative horns are trivial cofibrations}.
	
	(iv) $\Rightarrow$ (iii):
	This follows from the containment $\H_h \subset \celll(\H'_h \cup \H_v)$, which is precisely \cref{Oury horizontal horns from alternative ones}.
\end{proof}

Since $e$ admits a retraction, we obtain the following corollary by setting $Y$ to be the terminal cellular set $\cell[0]$.
\begin{corollary}\label{fibrant}
	Let $X \in \hattheta$ be a cellular set.
	Then the following are equivalent:
	\begin{itemize}
		\item [(i)] $X$ is a 2-quasi-category;
		\item [(ii)] $X$ has the right lifting property with respect to all maps in $\J_A$;
		\item [(iii)] $X$ has the right lifting property with respect to all maps in $\H_h \cup \H_v \cup \E_v$; and
		\item [(iv)] $X$ has the right lifting property with respect to all maps in $\H'_h \cup \H_v \cup \E_v$.
	\end{itemize}
\end{corollary}

The following corollary says that, when detecting left Quillen functors out of $\hattheta$, we may replace the infinite family $\E_v$ by a single map $[\id;e] : \cell[1;0] \to \cell[1;J]$.
Note that \cref{elegance} implies $\celll(\I)$ is precisely the set of monomorphisms in $\hattheta$ where
\[
\I = \bigl\{\partial\cell\nq \incl \cell\nq : \nq \in \cell\bigr\}.
\]

\begin{corollary}\label{left Quillen}
	Let $F : \hattheta \to \mathscr{M}$ be a left adjoint functor into a model category $\mathscr{M}$.
	Then the following are equivalent:
	\begin{itemize}
		\item [(i)] $F$ is a left Quillen functor;
		\item [(ii)] $F$ sends each map in $\I$ to a cofibration and each map in $\H_h \cup \H_v \cup \bigl\{e, [\id;e]\bigr\}$ to a trivial cofibration; and
		\item [(iii)] $F$ sends each map in $\I$ to a cofibration and each map in $\H'_h \cup \H_v \cup \bigl\{e, [\id;e]\bigr\}$ to a trivial cofibration
	\end{itemize}
\end{corollary}

\begin{proof}
	(i) $\Rightarrow$ (iii) follows from \cref{horn inclusions are trivial cofibrations,alternative horns are trivial cofibrations}, and (iii) $\Rightarrow$ (ii) follows from \cref{Oury horizontal horns from alternative ones}.
	
	For (ii) $\Rightarrow$ (i), suppose that $F$ satisfies (ii).
	Recall that for any $[1;0] \neq \nq \in \cell$ and any $1 \le k \le n$ satisfying $q_k = 0$, both of $\cell\nq \incl \Psi^k\nq$ and $\cell\nq \incl \Phi^k\nq$ are in
	\[
	\celll\Bigl(\H_h \cup \H_v \cup \bigl\{\Psi^\ell\mp \incl \Phi^\ell\mp : \dim\mp < \dim\nq\bigr\}\Bigr)
	\]
	by \cref{step 0,step 1,step 2,step 3,step 4}.
	It follows by induction on $\dim\nq$ that $F$ additionally sends all vertical equivalence extensions to trivial cofibrations.
	Thus \cref{main theorem} (and \cref{fibrant}) implies that the right adjoint to $F$ preserves fibrant objects and fibrations between them.
	
	Now let $f$ be a trivial cofibration in $\hattheta$.
	Since $F$ sends each map in $\I$ to a cofibration, $F(f)$ is a cofibration.
	Hence by \cite[Lemma 7.14]{JT}, $F(f)$ is trivial if and only if it has the left lifting property with respect to all fibrations between fibrant objects.
	The latter follows from the conclusion of the previous paragraph by taking the adjoint transpose.
\end{proof}

\section{Teaser: lax Gray tensor product}\label{teaser}
This section provides a peek into how the combinatorics developed in this paper will be used in our future work.

In (ordinary) 2-category theory, one can study various \emph{lax} notions where certain diagrams are required to commute not on the nose but up to appropriately coherent comparison 2-cells.
An example of such a notion is the (\emph{lax}) \emph{Gray tensor product} \cite[Theorem I.4.9]{Gray}
\[
\ttensor : \twoCat \times \twoCat \to \twoCat.
\]
If $f: a \to a'$ and $g : b \to b'$ are 1-cells in 2-categories $\A$ and $\B$ respectively, then in $\A \ttensor \B$ the square
\[
\begin{tikzpicture}[scale = 2]
\node at (0,1) {$(a,b)$};
\node at (0,0) {$(a,b')$};
\node at (1.25,1) {$(a',b)$};
\node at (1.25,0) {$(a',b')$};
\draw[->] (0,0.85) -- (0,0.15);
\draw[->] (1.25,0.85) -- (1.25,0.15);
\draw[->] (0.25,0) -- (1,0);
\draw[->] (0.25,1) -- (1,1);
\node[scale = 0.7] at (-0.2,0.5) {$(a,g)$};
\node[scale = 0.7] at (1.45,0.5) {$(a',g)$};
\node[scale = 0.7] at (5/8,-0.1) {$(f,b')$};
\node[scale = 0.7] at (5/8,1.1) {$(f,b)$};
\draw[->, double] (0.25,0.2) -- node[fill = white, scale = 0.9] {$\gamma_{f,g}$} (1,0.8);
\end{tikzpicture}
\]
admits a (not necessarily invertible) comparison 2-cell $\gamma_{f,g}$, and these $\gamma$'s are compatible with the 2-category structures of $\A$ and $\B$ in an appropriate sense.

The Gray tensor product is an indispensable tool in 2-category theory, and it is desirable to have an analogous construction in the $(\infty,2)$-context.
For instance, in their book on derived algebraic geometry \cite{Gaitsgory;Rozenblyum}, Gaitsgory and Rozenblyum listed and exploited various properties such a tensor product of $(\infty,2)$-categories should have (but they did not prove their construction indeed yields a tensor product with those properties).
As we mentioned in the introduction, the content of this paper was originally developed as a tool for proving the following theorem.

\begin{theorem}\label{Gray theorem}
	The 2-quasi-categorical Gray tensor product
	\[
	\tensor : \hattheta \times \hattheta \to \hattheta
	\]
	(defined below) is left Quillen.
\end{theorem}
This theorem corresponds to Proposition 3.2.6 in \cite[Chapter 10]{Gaitsgory;Rozenblyum}, an unproven result in that book.
More precisely, the proposition states that their Gray tensor product preserves $\infty$-categorical colimits in each variable.
Our Gray tensor product preserves (1-categorical) colimits in each variable by construction, and \cref{Gray theorem} allows us to upgrade this preservation property to a homotopical version.

By (the binary version of) \cref{left Quillen}, this theorem follows if we can prove that:
\begin{itemize}
	\item [(i)] $f \hat \tensor g$ is a cofibration for any $f,g \in \I$; and
	\item [(ii)] $f \hat \tensor g$ is a trivial cofibration whenever one of $f$ and $g$ is in $\I$ and the other is in $\H_h \cup \H_v \cup \bigl\{e, [\id;e]\bigr\}$.
\end{itemize}
A general proof of these facts will be given in our future paper.
In this section, we sketch the proof for a particular instance of (ii) in order to illustrate the role that the inner horns will play in that paper.

\begin{remark}
	For \emph{complicial sets} (which model $(\infty,\infty)$-categories), a relatively simple definition of the Gray tensor product was given by Verity in \cite{Verity:I,Verity:II}, where he also proved the complicial counterpart of \cref{Gray theorem} (and much more).
	One drawback of complicial sets is that there is only one obvious duality operation, namely the \emph{odd dual} induced by the automorphism $(-)\op$ on $\Delta$, although one would expect to be able to reverse the $n$-cells for any $n \ge 1$.
	On the other hand, for 2-quasi-categories both the horizontal and vertical duals are easy to describe, but the Gray tensor product does not admit a concrete description.
\end{remark}

\begin{definition}
	We define the \emph{Gray tensor product} $\tensor : \hattheta \times \hattheta \to \hattheta$ of cellular sets by extending the functor
	\[
	\begin{tikzcd}
	\cell \times \cell
	\arrow [r, hook] &
	\twoCat \times \twoCat
	\arrow [r, "\ttensor"] &
	\twoCat
	\arrow [r, "N"] &
	\hattheta
	\end{tikzcd}
	\]
	cocontinuously in each variable.
\end{definition}

We will illustrate how one can use the inner horns to show the Leibniz Gray tensor product
\[
\left(\begin{tikzcd}
{\horn_h^1[2;0,0]}
\arrow [d, hook] \\
{\cell[2;0,0]}
\end{tikzcd}\right)
\hat \tensor
\left(\begin{tikzcd}
{\partial\cell[1;0]}
\arrow [d, hook] \\
{\cell[1;0]}
\end{tikzcd}\right) =
\left(\begin{tikzcd}
{\spine[2;0,0]}
\arrow [d, hook] \\
{\cell[2;0,0]}
\end{tikzcd}\right)
\hat \tensor
\left(\begin{tikzcd}
{\partial\cell[1;0]}
\arrow [d, hook] \\
{\cell[1;0]}
\end{tikzcd}\right)
\]
is a trivial cofibration.
We will take it for granted that (i) above holds and hence this map is at least a cofibration (= monomorphism).

By construction of the Gray tensor product, the codomain $\cell[2;0,0] \tensor \cell[1;0]$ of this map is the nerve of the 2-category that looks like
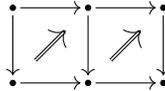
\[
\begin{tikzpicture}
\filldraw
(0,0) circle [radius = 1pt]
(1,0) circle [radius = 1pt]
(2,0) circle [radius = 1pt]
(0,1) circle [radius = 1pt]
(1,1) circle [radius = 1pt]
(2,1) circle [radius = 1pt];
\draw[->] (0.1,0) -- (0.9,0);
\draw[->] (1.1,0) -- (1.9,0);
\draw[->] (0.1,1) -- (0.9,1);
\draw[->] (1.1,1) -- (1.9,1);
\draw[->] (0, 0.9) -- (0, 0.1);
\draw[->] (1, 0.9) -- (1, 0.1);
\draw[->] (2, 0.9) -- (2, 0.1);
\draw[->, double] (0.3, 0.3) -- (0.7, 0.7);
\draw[->, double] (1.3, 0.3) -- (1.7, 0.7);
\end{tikzpicture}
\]
and its domain is the cellular subset
\[
X = \bigl(\horn_h^1[2;0,0] \tensor \cell[1;0]\bigr) \cup \bigl(\cell[2;0,0] \tensor \partial\cell[1;0]\bigr)
\]
of $\cell[2;0,0] \tensor \cell[1;0]$.
The first part, the cellular subset $\horn_h^1[2;0,0] \tensor \cell[1;0]$, is generated by the nerves of the sub-2-categories
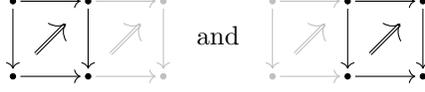
\[
\begin{tikzpicture}[baseline = 12]
\filldraw
(0,0) circle [radius = 1pt]
(1,0) circle [radius = 1pt]
(0,1) circle [radius = 1pt]
(1,1) circle [radius = 1pt];
\filldraw[gray!50!white]
(2,0) circle [radius = 1pt]
(2,1) circle [radius = 1pt];
\draw[->] (0.1,0) -- (0.9,0);
\draw[->, gray!50!white] (1.1,0) -- (1.9,0);
\draw[->] (0.1,1) -- (0.9,1);
\draw[->, gray!50!white] (1.1,1) -- (1.9,1);
\draw[->] (0, 0.9) -- (0, 0.1);
\draw[->] (1, 0.9) -- (1, 0.1);
\draw[->, gray!50!white] (2, 0.9) -- (2, 0.1);
\draw[->, double] (0.3, 0.3) -- (0.7, 0.7);
\draw[->, double, gray!50!white] (1.3, 0.3) -- (1.7, 0.7);
\end{tikzpicture} \hspace{10pt} \text {and} \hspace{10pt}
\begin{tikzpicture}[baseline = 12]
\filldraw
(2,0) circle [radius = 1pt]
(1,0) circle [radius = 1pt]
(2,1) circle [radius = 1pt]
(1,1) circle [radius = 1pt];
\filldraw[gray!50!white]
(0,0) circle [radius = 1pt]
(0,1) circle [radius = 1pt];
\draw[->, gray!50!white] (0.1,0) -- (0.9,0);
\draw[->] (1.1,0) -- (1.9,0);
\draw[->, gray!50!white] (0.1,1) -- (0.9,1);
\draw[->] (1.1,1) -- (1.9,1);
\draw[->, gray!50!white] (0, 0.9) -- (0, 0.1);
\draw[->] (1, 0.9) -- (1, 0.1);
\draw[->] (2, 0.9) -- (2, 0.1);
\draw[->, double, gray!50!white] (0.3, 0.3) -- (0.7, 0.7);
\draw[->, double] (1.3, 0.3) -- (1.7, 0.7);
\end{tikzpicture}
\]
and $\cell[2;0,0] \tensor \partial\cell[1;0]$ is generated by the nerves of
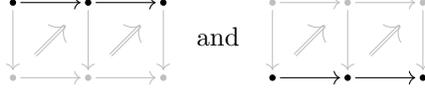
\[
\begin{tikzpicture}[baseline = 12]
\filldraw[gray!50!white]
(0,0) circle [radius = 1pt]
(1,0) circle [radius = 1pt]
(2,0) circle [radius = 1pt];
\filldraw
(0,1) circle [radius = 1pt]
(1,1) circle [radius = 1pt]
(2,1) circle [radius = 1pt];
\draw[->, gray!50!white] (0.1,0) -- (0.9,0);
\draw[->, gray!50!white] (1.1,0) -- (1.9,0);
\draw[->] (0.1,1) -- (0.9,1);
\draw[->] (1.1,1) -- (1.9,1);
\draw[->, gray!50!white] (0, 0.9) -- (0, 0.1);
\draw[->, gray!50!white] (1, 0.9) -- (1, 0.1);
\draw[->, gray!50!white] (2, 0.9) -- (2, 0.1);
\draw[->, double, gray!50!white] (0.3, 0.3) -- (0.7, 0.7);
\draw[->, double, gray!50!white] (1.3, 0.3) -- (1.7, 0.7);
\end{tikzpicture}\hspace{10pt} \text {and} \hspace{10pt}
\begin{tikzpicture}[baseline = 12]
\filldraw
(0,0) circle [radius = 1pt]
(1,0) circle [radius = 1pt]
(2,0) circle [radius = 1pt];
\filldraw[gray!50!white]
(0,1) circle [radius = 1pt]
(1,1) circle [radius = 1pt]
(2,1) circle [radius = 1pt];
\draw[->] (0.1,0) -- (0.9,0);
\draw[->] (1.1,0) -- (1.9,0);
\draw[->, gray!50!white] (0.1,1) -- (0.9,1);
\draw[->, gray!50!white] (1.1,1) -- (1.9,1);
\draw[->, gray!50!white] (0, 0.9) -- (0, 0.1);
\draw[->, gray!50!white] (1, 0.9) -- (1, 0.1);
\draw[->, gray!50!white] (2, 0.9) -- (2, 0.1);
\draw[->, double, gray!50!white] (0.3, 0.3) -- (0.7, 0.7);
\draw[->, double, gray!50!white] (1.3, 0.3) -- (1.7, 0.7);
\end{tikzpicture}
\]
We wish to show that $X \incl \cell[2;0,0] \tensor \cell[1;0]$ is a trivial cofibration.
We separate the non-degenerate cells in $(\cell[2;0,0] \tensor \cell[1;0]) \setminus X$ into six kinds according to their ``silhouette''.
The cells
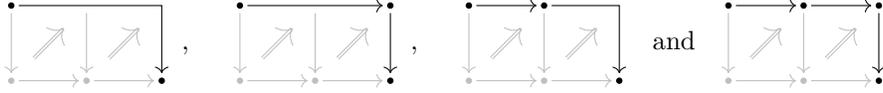
\[
\begin{tikzpicture}[baseline = 12]
\filldraw[gray!50!white]
(0,0) circle [radius = 1pt]
(1,0) circle [radius = 1pt];
\filldraw
(0,1) circle [radius = 1pt]
(2,0) circle [radius = 1pt];
\draw[->, gray!50!white] (0.1,0) -- (0.9,0);
\draw[->, gray!50!white] (1.1,0) -- (1.9,0);
\draw[->] (0.1,1) -- (2,1) -- (2,0.1);
\draw[->, gray!50!white] (0, 0.9) -- (0, 0.1);
\draw[->, gray!50!white] (1, 0.9) -- (1, 0.1);
\draw[->, double, gray!50!white] (0.3, 0.3) -- (0.7, 0.7);
\draw[->, double, gray!50!white] (1.3, 0.3) -- (1.7, 0.7);
\end{tikzpicture}\hspace{5pt}, \hspace{15pt}
\begin{tikzpicture}[baseline = 12]
\filldraw[gray!50!white]
(0,0) circle [radius = 1pt]
(1,0) circle [radius = 1pt];
\filldraw
(0,1) circle [radius = 1pt]
(2,1) circle [radius = 1pt]
(2,0) circle [radius = 1pt];
\draw[->, gray!50!white] (0.1,0) -- (0.9,0);
\draw[->, gray!50!white] (1.1,0) -- (1.9,0);
\draw[->] (0.1,1) -- (1.9,1);
\draw[->] (2,0.9) -- (2,0.1);
\draw[->, gray!50!white] (0, 0.9) -- (0, 0.1);
\draw[->, gray!50!white] (1, 0.9) -- (1, 0.1);
\draw[->, double, gray!50!white] (0.3, 0.3) -- (0.7, 0.7);
\draw[->, double, gray!50!white] (1.3, 0.3) -- (1.7, 0.7);
\end{tikzpicture}\hspace{5pt}, \hspace{15pt}
\begin{tikzpicture}[baseline = 12]
\filldraw[gray!50!white]
(0,0) circle [radius = 1pt]
(1,0) circle [radius = 1pt];
\filldraw
(1,1) circle [radius = 1pt]
(0,1) circle [radius = 1pt]
(2,0) circle [radius = 1pt];
\draw[->, gray!50!white] (0.1,0) -- (0.9,0);
\draw[->, gray!50!white] (1.1,0) -- (1.9,0);
\draw[->] (0.1,1) -- (0.9,1);
\draw[->] (1.1,1) -- (2,1) -- (2,0.1);
\draw[->, gray!50!white] (0, 0.9) -- (0, 0.1);
\draw[->, gray!50!white] (1, 0.9) -- (1, 0.1);
\draw[->, double, gray!50!white] (0.3, 0.3) -- (0.7, 0.7);
\draw[->, double, gray!50!white] (1.3, 0.3) -- (1.7, 0.7);
\end{tikzpicture}\hspace{10pt} \text {and} \hspace{10pt}
\begin{tikzpicture}[baseline = 12]
\filldraw[gray!50!white]
(0,0) circle [radius = 1pt]
(1,0) circle [radius = 1pt];
\filldraw
(1,1) circle [radius = 1pt]
(0,1) circle [radius = 1pt]
(2,1) circle [radius = 1pt]
(2,0) circle [radius = 1pt];
\draw[->, gray!50!white] (0.1,0) -- (0.9,0);
\draw[->, gray!50!white] (1.1,0) -- (1.9,0);
\draw[->] (0.1,1) -- (0.9,1);
\draw[->] (1.1,1) -- (1.9,1);
\draw[->] (2,0.9) -- (2,0.1);
\draw[->, gray!50!white] (0, 0.9) -- (0, 0.1);
\draw[->, gray!50!white] (1, 0.9) -- (1, 0.1);
\draw[->, double, gray!50!white] (0.3, 0.3) -- (0.7, 0.7);
\draw[->, double, gray!50!white] (1.3, 0.3) -- (1.7, 0.7);
\end{tikzpicture}
\]
have the same silhouette ``$\begin{tikzpicture}[baseline = 1, scale = 0.25] \draw (0,1) -- (2,1) -- (2,0);\end{tikzpicture}$'', and similarly each of the silhouettes ``$\begin{tikzpicture}[baseline = 1, scale = 0.25] \draw (0,1) -- (1,1) -- (1,0) -- (2,0);\end{tikzpicture}$'' and ``$\begin{tikzpicture}[baseline = 1, scale = 0.25] \draw (0,1) -- (0,0) -- (2,0);\end{tikzpicture}$'' has four cells.
There are two cells
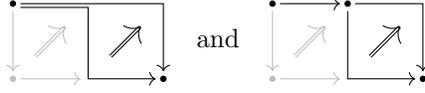
\[
\begin{tikzpicture}[baseline = 12]
\filldraw
(2,0) circle [radius = 1pt]
(0,1) circle [radius = 1pt];
\filldraw[gray!50!white]
(0,0) circle [radius = 1pt];
\draw[->, gray!50!white] (0.1,0) -- (0.9,0);
\draw[->] (0.1,1) -- (2,1) -- (2,0.1);
\draw[->] (0.1,0.95) -- (1,0.95) -- (1,0) -- (1.9,0);
\draw[->, gray!50!white] (0, 0.9) -- (0, 0.1);
\draw[->, double, gray!50!white] (0.3, 0.3) -- (0.7, 0.7);
\draw[->, double] (1.3, 0.3) -- (1.7, 0.7);
\end{tikzpicture} \hspace{10pt} \text {and} \hspace{10pt}
\begin{tikzpicture}[baseline = 12]
\filldraw
(2,0) circle [radius = 1pt]
(1,1) circle [radius = 1pt]
(0,1) circle [radius = 1pt];
\filldraw[gray!50!white]
(0,0) circle [radius = 1pt];
\draw[->, gray!50!white] (0.1,0) -- (0.9,0);
\draw[->] (1,0.9) -- (1,0) -- (1.9,0);
\draw[->] (0.1,1) -- (0.9,1);
\draw[->] (1.1,1) -- (2,1) -- (2,0.1);
\draw[->, gray!50!white] (0, 0.9) -- (0, 0.1);
\draw[->, double, gray!50!white] (0.3, 0.3) -- (0.7, 0.7);
\draw[->, double] (1.3, 0.3) -- (1.7, 0.7);
\end{tikzpicture}
\]
of silhouette ``$\begin{tikzpicture}[baseline = 1, scale = 0.25] \filldraw (0,1) -- (1,1) (1,1) -- (2,1) -- (2,0) -- (1,0) -- cycle;\end{tikzpicture}$'', and similarly for ``$\begin{tikzpicture}[baseline = 1, scale = 0.25] \filldraw (2,0) -- (3,0) (1,1) -- (2,1) -- (2,0) -- (1,0) -- cycle;\end{tikzpicture}$''.
Finally, the cells
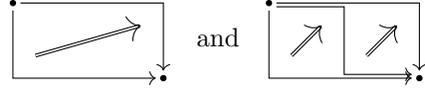
\[
\begin{tikzpicture}[baseline = 12]
\filldraw
(2,0) circle [radius = 1pt]
(0,1) circle [radius = 1pt];
\draw[->] (0.1,1) -- (2,1) -- (2,0.1);
\draw[->] (0,0.9) -- (0,0)-- (1.9,0);
\draw[->, double] (0.3, 0.3) -- (1.7, 0.7);
\end{tikzpicture} \hspace{10pt} \text {and} \hspace{10pt}
\begin{tikzpicture}[baseline = 12]
\filldraw
(2,0) circle [radius = 1pt]
(0,1) circle [radius = 1pt];
\draw[->] (0.1,1) -- (2,1) -- (2,0.1);
\draw[->] (0,0.9) -- (0,0)-- (1.9,0);
\draw[->] (0.1,0.95) -- (1,0.95) -- (1,0.05) -- (1.9,0.05);
\draw[->, double] (0.3, 0.3) -- (0.7, 0.7);
\draw[->, double] (1.3, 0.3) -- (1.7, 0.7);
\end{tikzpicture}
\]
have silhouette ``$\begin{tikzpicture}[baseline = 1, scale = 0.25] \filldraw (0,0) -- (2,0) -- (2,1) -- (0,1) -- cycle;\end{tikzpicture}$''.
We can associate a cut-point (= a point that disconnects the shape if removed) to each silhouette except for the last one as follows:
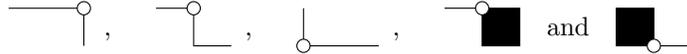
\[
\begin{tikzpicture}[baseline = 4, scale = 0.5]
\draw (0,1) -- (2,1) -- (2,0);
\draw[fill = white]
(2,1) circle [radius = 5pt];
\end{tikzpicture} \hspace{5pt},\hspace{15pt}
\begin{tikzpicture}[baseline = 4, scale = 0.5]
\draw (0,1) -- (1,1) -- (1,0) -- (2,0);
\draw[fill = white]
(1,1) circle [radius = 5pt];
\end{tikzpicture} \hspace{5pt},\hspace{15pt}
\begin{tikzpicture}[baseline = 4, scale = 0.5]
\draw (0,1) -- (0,0) -- (2,0);
\draw[fill = white]
(0,0) circle [radius = 5pt];
\end{tikzpicture} \hspace{5pt},\hspace{15pt}
\begin{tikzpicture}[baseline = 4, scale = 0.5]
\filldraw (0,1) -- (1,1) (1,1) -- (2,1) -- (2,0) -- (1,0) -- cycle;
\draw[fill = white]
(1,1) circle [radius = 5pt];
\end{tikzpicture} \hspace{10pt} \text {and} \hspace{10pt}
\begin{tikzpicture}[baseline = 4, scale = 0.5]
\filldraw (2,0) -- (3,0) (1,1) -- (2,1) -- (2,0) -- (1,0) -- cycle;
\draw[fill = white]
(2,0) circle [radius = 5pt];
\end{tikzpicture}
\]
Observe that the set of non-degenerate cells of these silhouettes can then be partitioned into pairs of the form $\bigl\{\phi, \phi \cdot \delta_h^{k_\phi}\bigr\}$ where the $k_\phi$-th vertex of $\phi$ is the cut-point associated to its silhouette.
We can glue such $\phi$ to $X$ along $\horn_h^{k_\phi}$ in increasing order of $\dim\phi$, and then glue the above $(1;2)$-cell of silhouette ``$\begin{tikzpicture}[baseline = 1, scale = 0.25] \filldraw (0,0) -- (2,0) -- (2,1) -- (0,1) -- cycle;\end{tikzpicture}$'' along $\horn_v^{1;1}[1;2]$.
This exhibits the inclusion $X \incl \cell[2;0,0] \tensor \cell[1;0]$ as a member of $\celll(\H_h \cup \H_v)$ and hence as a trivial cofibration by \cref{horn inclusions are trivial cofibrations}.

\section*{Acknowledgements}
The author would like to thank his supervisor Dominic Verity for helpful feedback on earlier versions of this paper.
He also gratefully acknowledges the support of an International Macquarie University Research Training Program Scholarship (Allocation Number: 2017127).
Thanks to the anonymous referees' comments, the readability of this paper has been greatly improved and an error in the original proof of \cref{part 1} has been corrected.

\bibliographystyle{plain}
\bibliography{ref}

\begin{thebibliography}{10}

\bibitem{Ara:nqcat}
Dimitri Ara.
\newblock Higher quasi-categories vs higher {R}ezk spaces.
\newblock {\em Journal of K-Theory. K-Theory and its Applications in Algebra,
  Geometry, Analysis \& Topology}, 14(3):701, 2014.

\bibitem{Berger:nerve}
Clemens Berger.
\newblock A cellular nerve for higher categories.
\newblock {\em Advances in Mathematics}, 169(1):118, 2002.

\bibitem{Berger:wreath}
Clemens Berger.
\newblock Iterated wreath product of the simplex category and iterated loop
  spaces.
\newblock {\em Advances in Mathematics}, 213(1):230, 2007.

\bibitem{Bergner;Rezk:Reedy}
Julia~E. Bergner and Charles Rezk.
\newblock Reedy categories and the {$\varTheta$}-construction.
\newblock {\em Math. Z.}, 274(1-2):499--514, 2013.

\bibitem{Gaitsgory;Rozenblyum}
Dennis Gaitsgory and Nick Rozenblyum.
\newblock {\em A study in derived algebraic geometry. {V}ol. {I}.
  {C}orrespondences and duality}, volume 221 of {\em Mathematical Surveys and
  Monographs}.
\newblock American Mathematical Society, Providence, RI, 2017.

\bibitem{Gray}
John~W. Gray.
\newblock {\em Formal category theory: adjointness for {$2$}-categories}.
\newblock Lecture Notes in Mathematics, Vol. 391. Springer-Verlag, Berlin-New
  York, 1974.

\bibitem{Joyal:Kan}
A.~Joyal.
\newblock Quasi-categories and {K}an complexes.
\newblock {\em J. Pure Appl. Algebra}, 175(1-3):207--222, 2002.
\newblock Special volume celebrating the 70th birthday of Professor Max Kelly.

\bibitem{Joyal:applications}
Andr\'{e} Joyal.
\newblock {The Theory of Quasi-Categories and its Applications}.
\newblock preprint.

\bibitem{JT}
Andr\'{e} Joyal and Myles Tierney.
\newblock Quasi-categories vs {S}egal spaces.
\newblock In {\em Categories in algebra, geometry and mathematical physics},
  volume 431 of {\em Contemp. Math.}, pages 277--326. Amer. Math. Soc.,
  Providence, RI, 2007.

\bibitem{HA}
Jacob Lurie.
\newblock Higher algebra.
\newblock available at \url{http://www.math.harvard.edu/~lurie/papers/HA.pdf}.

\bibitem{HTT}
Jacob Lurie.
\newblock {\em Higher topos theory.}
\newblock Annals of Mathematics Studies, 170. Princeton University Press,
  Princeton, NJ, 2009.

\bibitem{Oury}
David Oury.
\newblock {\em Duality for Joyal's category $\Theta$ and homotopy concepts for
  $\Theta_2$-sets}.
\newblock PhD thesis, Macquarie University, 2010.

\bibitem{Rezk:cartesian}
Charles Rezk.
\newblock A {C}artesian presentation of weak {$n$}-categories.
\newblock {\em Geom. Topol.}, 14(1):521--571, 2010.

\bibitem{Street:fibrations}
Ross Street.
\newblock Fibrations in bicategories.
\newblock {\em Cahiers Topologie G\'{e}om. Diff\'{e}rentielle}, 21(2):111--160,
  1980.

\bibitem{Verity:I}
D.~R.~B. Verity.
\newblock Weak complicial sets. {I}. {B}asic homotopy theory.
\newblock {\em Adv. Math.}, 219(4):1081--1149, 2008.

\bibitem{Verity:II}
Dominic Verity.
\newblock Weak complicial sets. {II}. {N}erves of complicial {G}ray-categories.
\newblock In {\em Categories in algebra, geometry and mathematical physics},
  volume 431 of {\em Contemp. Math.}, pages 441--467. Amer. Math. Soc.,
  Providence, RI, 2007.

\bibitem{Watson}
Nathaniel Watson.
\newblock {\em Non-Simplicial Nerves for Two-Dimensional Categorical
  Structures}.
\newblock PhD thesis, University of California, Berkeley, 2013.

\end{thebibliography}

\end{document}